\documentclass[11pt]{article}

\usepackage{fullpage,tikz-cd,amssymb,graphicx, amsmath,amsthm,amsfonts,sectsty,bm,url,MnSymbol,xcolor,mathtools,textcomp,dirtytalk}

\usepackage[letterpaper,hmargin=1.0in,vmargin=1.0in]{geometry}
\usepackage[linktocpage=true]{hyperref}

\hypersetup{
  colorlinks   = true, 
  urlcolor     = blue, 
  linkcolor    = blue, 
  citecolor   = purple 
}
\usepackage[capitalize]{cleveref}
\graphicspath{ {./} }

\newtheorem{theorem}{Theorem}[subsection]
\newtheorem{lemma}[theorem]{Lemma}
\newtheorem{proposition}[theorem]{Proposition}
\newtheorem{corollary}[theorem]{Corollary}
\newtheorem{claim}[theorem]{Claim}
\newtheorem{remark}[theorem]{Remark}

\theoremstyle{definition}
\newtheorem{definition}[theorem]{Definition}

\numberwithin{equation}{section}

\makeatletter
\def\BState{\State\hskip-\ALG@thistlm}
\makeatother

\def\S{\textsection}

\def\R{\ensuremath{\mathbf{R}}}

\def\C{\ensuremath{\mathbf{C}}}

\def\Q{\ensuremath{\mathbf{Q}}}

\def\N{\ensuremath{\mathbf{N}}}

\def\Z{\ensuremath{\mathbf{Z}}}
\def\Zplus{\ensuremath{\mathbb{Z}_{\small{\geq 0}}}}

\def\L{\ensuremath{\mathcal{L}}}
\def\U{\ensuremath{\mathcal{U}}}
\def\W{\ensuremath{\mathcal{W}}}
\def\V{\ensuremath{\mathcal{V}}}
\def\Linf{\ensuremath{L^{\infty}}}

\newcommand{\GG}{\mathbf{G}}
\newcommand{\Ha}{\operatorname{H}_a}
\newcommand{\Hb}{\operatorname{H}_b}

\begin{document}
\title{
  \textbf{Asymptotic Cohomology and Uniform Stability \\
  	for Lattices in Semisimple Groups}}
\author{Lev Glebsky\thanks{Universidad Aut´onoma de San Luis Potosi, Mexico \textbf{Email:}
 \textit{glebsky@cactus.iico.uaslp.mx}},  Alexander Lubotzky\thanks{Weizmann Institute of Science, Rehovot, Israel \textbf{Email:}\textit{ alex.lubotzky@mail.huji.ac.il}}, Nicolas Monod\thanks{École Polytechnique Fédérale de Lausanne (EPFL), Lausanne, Switzerland \textbf{Email:} \textit{nicolas.monod@epfl.ch}}, Bharatram Rangarajan\thanks{Hebrew University of Jerusalem, Jerusalem, Israel \textbf{Email:}\textit{ bharatrm.rangarajan@mail.huji.ac.il}}}

\setlength{\parskip}{1ex plus 0.5ex minus 0.2ex}
\maketitle
\begin{abstract}
		It is, by now, classical that lattices in higher rank semisimple groups have various rigidity properties. In this work, we add another such rigidity property to the list: uniform stability with respect to the family of unitary operators on finite-dimensional Hilbert spaces equipped with submultiplicative norms. Namely, we show that for (most) high-rank lattices, every finite-dimensional unitary "almost-representation" of $\Gamma$ is a small deformation of a (true) unitary representation. This extends a result of Kazhdan \cite{Kaz} for amenable groups and of Burger-Ozawa-Thom \cite{BOT} for $SL(n,Z)$ (for $n>2$). Towards this goal, we first build an elaborate cohomological theory capturing the obstruction to such stability, and show that the vanishing of second cohomology implies uniform stability in this setting. This cohomology can be roughly thought of as an asymptotic version of bounded cohomology, and sheds light on a question raised in \cite{inviteMonod} about a possible connection between vanishing of second bounded cohomology and Ulam stability.\\
  
  
	\end{abstract}
 
	\centerline{\textit{In memory of Robert J. Zimmer}}

	\section*{Introduction}
	Consider a semisimple group $G = \prod^k_{i = 1} \GG_i(K_i)$, where for $1\leq i \leq k$, $K_i$ is a local field, and $\GG_i$ is an almost $K_i$-simple group. If the rank $\sum^k_{i = 1} rk_{K_i} (\GG_i) \ge 2$, such a $G$ is referred to as a \emph{higher rank} semisimple group. The class of irreducible lattices $\Gamma$ in such groups $G$ (referred to as \emph{higher rank lattices}) form an interesting class of groups, which over the years, have been shown to satisfy many rigidity properties, such as local rigidity, Mostow strong-rigidity, Margulis super-rigidity (implying that they are arithmetic groups), Zimmer cocycle rigidity, quasi-isometric rigidity, first-order rigidity, etc. (see \cite{ES}, \cite{alm}, \cite{margulisbook}, \cite{brownfisher} and the references therein). A common feature of the classical rigidity results is that such a higher rank lattice $\Gamma$ has some clear family of representations, and all other representations are just easy variants of them.\\
	The goal of this paper is to demonstrate another type of rigidity phenomena of these lattices. Before stating the exact formulation, let us recall that Margulis super-rigidity, while usually not formulated this way, also gives a full classification of all the finite dimensional unitary representations of a higher rank lattice $\Gamma$ as above. Margulis super-rigidity implies that all such irreducible representations come from a combination of those that factor through finite quotients (and these are the only ones if $\Gamma$ is a non-uniform lattice) and from the representations of $\Gamma$ appearing naturally in its definition as an arithmetic group by Galois twisting (see \S$1.3$ in \cite{margulisbook}). The rigidity phenomenon we  study here, which is called \emph{uniform stability}, is the property that every unitary \say{almost-representation} of $\Gamma$ is a small deformation of a unitary representation.\\
	
	\subsection*{Uniform Stability of Groups}
	Let $\Gamma$ be a discrete group and $(G,d_G)$ be a metric group (where $d_G$ is a bi-invariant metric on $G$). For $\epsilon >0$, a map $\phi:\Gamma \to G$ is said to be an $\epsilon$-almost homomorphism (or \emph{$\epsilon$-homomorphism}) if $d_G(\phi(xy),\phi(x)\phi(y)) \leq \epsilon$ for every $x,y \in \Gamma$. The value $\sup_{x,y \in \Gamma}d_G(\phi(xy),\phi(x)\phi(y))$ is called the \emph{defect} of $\phi$. \\
	Let $\mathcal{G}$ be a family of metric groups. We say that $\Gamma$ is \emph{uniformly stable} with respect to $\mathcal{G}$ if for any $\epsilon>0$, there exists $\delta=\delta(\epsilon)$ with $\lim_{\epsilon \to 0}\delta(\epsilon)=0$ such that for any $\epsilon$-homomorphism $\phi:\Gamma \to G$ (for $G \in \mathcal{G}$), there exists a homomorphism $\psi \in Hom(\Gamma,G)$ with $\sup_{x \in \Gamma} \: d_G(\phi(x),\psi(x)) \leq \delta$. In other words, $\Gamma$ is uniformly stable with respect to $\mathcal{G}$ if any almost homomorphism of $\Gamma$ to any group in the family $\mathcal{G}$ is close to a (true) homomorphism. \\ 
	
	Questions of this nature were first raised and studied in \cite{turing}, \cite{vonneumann} and \cite{Ulam}, and of particular interest is the case when $\mathcal{G}$ is the family of unitary operators on Hilbert spaces and the metric is given by a norm (on the space of bounded operators), as studied in \cite{Kaz} and \cite{BOT}. Note that in this work, we will be interested solely in uniform stability, as opposed to \emph{pointwise} stability, as studied in \cite{DCGLT}, \cite{arz} and the references therein.\\
	
	The notion of uniform stability with respect to unitary operators on Hilbert spaces equipped with the operator norm is referred to in \cite{BOT} as \emph{strong Ulam stability}, while if we restrict the family to unitary operators on \emph{finite-dimensional} Hilbert spaces, it is referred to as \emph{Ulam stability}. In the pioneering work of Kazhdan \cite{Kaz} (and clarified further in \cite{shtern} and \cite{johnson2}), it is shown that 
	\begin{theorem}[\cite{Kaz}]
		Every (discrete) amenable group $\Gamma$ is Ulam stable (in fact, even strongly Ulam stable).
	\end{theorem}
	It is worth noting that the only known examples of strongly Ulam stable groups are amenable, and it is natural to ask if strong Ulam stability characterizes amenability.\\ 
	Let us mention at this point that one of the (innumerable) equivalent characterizations of amenability is given in terms of the vanishing of bounded cohomology with dual coefficients: $\Gamma$ is amenable iff $\Hb^n(\Gamma, V) = 0$ for every dual Banach $\Gamma$-module $V$ and $n >0$. Here $\Hb^n(\Gamma,V)$ denotes the $n$-th \emph{bounded} cohomology group of $\Gamma$ with coefficients in the Banach $\Gamma$-module $V$. Kazhdan's proof does not use this result explicitly but does use a notion of $\epsilon$-cocycles and approximate cohomology in degree $2$.\\
	
	Ulam stability was further studied in \cite{BOT} where they show more examples (and non-examples) of Ulam stable groups. It is shown there that if a group contains a non-abelian free subgroup, then it is \emph{not} strongly Ulam stable. In particular, this means that higher rank lattices are not strongly Ulam stable. On the positive side, they show:
	\begin{theorem}[\cite{BOT}]
		\label{BOT}
		Let $\mathcal{O}$ be the ring of integers of a number field, $S$ a finite set of primes, and $\mathcal{O}_S$ the corresponding localization. Then for every $n \geq 3$, $SL(n,\mathcal{O}_S)$ is Ulam stable.
	\end{theorem}
	The proof of this result in \cite{BOT} uses the fact that $SL(n,\mathcal{O}_S)$ (for $n \ge 3$) is boundedly generated by elementary matrices, and makes no reference to bounded cohomology (this result is further extended in \cite{gamm} in the case of $n=2$ when $\mathcal{O}_S$ has infinitely many units). However, note that for $\Gamma=SL(n,\mathcal{O}_S)$,  $\Hb^2\left(\Gamma,V\right) = 0$ for every dual separable $\Gamma$-module $V$. In fact, it is shown in \cite{burgerMonod} that for every higher rank lattice $\Gamma$ and any dual, separable Banach $\Gamma$-module $V$ with $V^{\Gamma}=\{0\}$, $\Hb^2(\Gamma, V) = 0$. All this hints at a possible connection between bounded cohomology and Ulam stability, as raised by Monod in his ICM talk \cite[Problem F]{inviteMonod}, and serves as one of the starting points for our current work.\\

	\section*{Main Results and Methods}
	In this paper, we generalize \cref{Kaz} and \cref{BOT} to a wider class of groups and metrics. We shall consider the question of uniform stability with respect to the family $\mathfrak{U}$ of groups of unitary operators on \emph{finite-dimensional} Hilbert spaces, with the metrics induced from submultiplicative norms on matrices (which we shall denote uniform $\mathfrak{U}$-stability). These include the $p$-Schatten norms for $1 \leq p \le \infty$ (and in particular, uniform $\mathfrak{U}$-stability subsumes Ulam stability). Furthermore, we shall show uniform $\mathfrak{U}$-stability \emph{with a linear estimate}, which means that the distance of an almost homomorphism from a homomorphism is linearly bounded by its defect, and all our results are proved in this stronger notion of stability.\\
	
	To this end, we build a new type of bounded cohomological theory that can capture obstructions to uniform $\mathfrak{U}$-stability, so that uniform $\mathfrak{U}$-stability follows as a consequence of the vanishing of the second cohomology group in this theory. While we shall develop this in full detail in \S\ref{sec-coho}, our technique involves the following two main steps:
	\begin{itemize}
		\item \textbf{Defect Diminishing}: Expressing the problem of uniform stability as a homomorphism lifting problem, we can treat it as a culmination of intermediate lifts so that at each step, the lifting kernel is abelian. This is a uniform variant of defect diminishing that was introduced in \cite{DCGLT} in the non-uniform setting, and is applicable when the relevant norms in the target groups are submultiplicative.
		\item \textbf{Asymptotic Cohomology}: Such a homomorphism lifting problem with abelian kernel naturally leads to a cohomological reformulation (as in \cite{DCGLT}). However, unlike in the non-uniform setting where ordinary group cohomology comes up, in our uniform setting we need to carefully construct a new  cohomology theory such that the vanishing of the second cohomology group in this model implies (uniform) defect diminishing, and hence uniform stability. 
	\end{itemize} 
	The cohomological theory we construct is an asymptotic variant of the bounded cohomology of the ultrapower ${}^*\Gamma$ with coefficients in a suitable ultraproduct Banach space $\W$, but \emph{restricted to the \say{internal} objects in this universe}, which we shall call the \emph{asymptotic cohomology} of $\Gamma$ denoted $\Ha^{\bullet}(\Gamma,\W)$. 
	\begin{theorem}
		Suppose $\Ha^2(\Gamma,\W)=0$, then $\Gamma$ is uniformly $\mathfrak{U}$-stable with a linear estimate. 
	\end{theorem}
	This new cohomology theory bears some similarity to the theory of bounded cohomology, and sometimes we can easily adapt arguments there to our model (for instance, we can show that $\Ha^2(\Gamma,\W)=0$ for an amenable group $\Gamma$, immediately implying that amenable groups are Ulam-stable as in \cite{Kaz},\cite{shtern},\cite{johnson2}), though other times serious difficulties arise (which are responsible for the length of this paper).\\
	
	The groups $\Gamma$ that we will be particularly interested in are lattices in higher rank semisimple groups. Unlike some of the other rigidity results which sometimes hold for some lattices in rank one simple groups, we also first show the following result: 
	\begin{proposition}\label{rankone}
		If $\Gamma$ is a lattice in a semisimple group of rank $1$, then $\Gamma$ is not uniformly $\mathfrak{U}$-stable.
	\end{proposition}
	It is shown in \cite{fujiwara} that for such a $\Gamma$, $\Hb^2(\Gamma,\R)$ is infinite dimensional. More precisely, Fujiwara constructs (many) non-trivial
	quasi-homomorphisms witnessing that the comparison map $c:\Hb^2(\Gamma,\R) \to H^2(\Gamma,\R)$ is not injective. By exponentiation, such quasimorphisms
	yield almost homomorphisms to $U(1)$ that are not close to any
	homomorphism (we shall discuss this in more detail in \S\ref{sec-u1}). In
	particular, a lattice of rank one is not uniformly U-stable, and to hope for uniform $\mathfrak{U}$-stability, the condition that rank of $\Gamma$ is at least $2$ is necessary.\\
	
	For the main result of the paper, we need some definitions capturing properties of the class of semisimple groups we will be interested in. For a locally compact group $G$, we denote by $\Hb^{\bullet}(G,\mathbf{R})$ the \emph{continuous} bounded cohomology of $G$ with trivial coefficients.
	\begin{itemize}
		\item A locally compact group $G$ is said to have the \textbf{2\textonehalf property} (of vanishing bounded cohomology) if $\Hb^2(G,\R)=0$ and $\Hb^3(G,\R)$ is Hausdorff.
		\item Let $G$ be a non-compact simple Lie group, and fix a minimal parabolic subgroup $P \leq G$. The group $G$ is said to have \textbf{Property-$G(\mathcal{Q}_1,\mathcal{Q}_2)$} if there exist two proper parabolic subgroups $Q_1$ and $Q_2$ containing $P$, both having the 2\textonehalf-property, such that $G$ is boundedly generated by the union $Q_1 \cup Q_2$. A semisimple group $G$ is said to have Property-$G(\mathcal{Q}_1,\mathcal{Q}_2)$ if all its simple factors have Property-$G(\mathcal{Q}_1,\mathcal{Q}_2)$. 
	\end{itemize}
	Note that if a semisimple group has Property-$G(\mathcal{Q}_1,\mathcal{Q}_2)$, then by definition, it must be of rank at least $2$. But note that not all simple groups have the property (for instance, $SL_3(\R)$). However, in \S\ref{ssec-semisimple}, we will show that many classes of groups do have this property, for example, all simple groups (of rank at least $2$) over $\C$ or over a non-archimedean field, and $SL_n(\R)$ for $n \geq 4$.\\
	We can now state our main result:
	\begin{theorem}\label{maintheorem11}
		Let $\Gamma$ be a lattice in a semisimple group $G$ that has Property-$G(\mathcal{Q}_1,\mathcal{Q}_2)$. Then $\Ha^2(\Gamma,\W)=0$, so in particular, $\Gamma$ is uniformly $\mathfrak{U}$-stable. 
	\end{theorem}
	The main result of our paper is thus concerned with showing that $\Ha^2(\Gamma,\W)=0$ for $\Gamma$ being a lattice in a higher rank Lie group (satisfying certain conditions). For this, we take inspiration from the results of \cite{burgerMonod} about the vanishing of bounded cohomology for such lattices. More specifically, our approach is inspired by a proof of \cite{monodshalom-cocyle} specifically in degree two, and a version of that result and proof technique are outlined below.
	\begin{theorem}[\cite{monodshalom-cocyle}]\label{MonodShalomProof}
		Let $G$ be a higher rank simple group, and $P \leq G$ be a minimal parabolic subgroup. Suppose $G$ contains two proper parabolic subgroups $Q_1$ and $Q_2$ such that $P \subseteq Q_1 \cap Q_2$, $G$ is generated by $Q_1 \cup Q_2$, and $\Hb^2(Q_1,\R)=\Hb^2(Q_2,\R)=0$. Then for any lattice $\Gamma$ in $G$ and a dual separable Banach $\Gamma$-module $W$, $\Hb^2(\Gamma,W)=0$. 
	\end{theorem}
	The proof of \cref{MonodShalomProof} proceeds in several steps briefly sketched below, where we also mention the corresponding steps and difficulties in the proof of \cref{maintheorem11} even when $G$ is simple:
	\begin{itemize}
		\item The first step is to use an Eckman-Shapiro induction to construct a dual, separable, continuous Banach $G$-module $V$ so that $\Hb^2(\Gamma,W)=\Hb^2(G,V)$, thus reducing the problem to showing that $\Hb^2(G,V)=0$. A similar inductive procedure, described in \S\ref{sec-induc}, allows us to construct an ultraproduct Banach space $\V$ with an asymptotic action of $G$ so that $\Ha^2(\Gamma,\W)=\Ha^2(G,\V)$. Note that in the setting of asymptotic cohomology, we actually work with ultrapowers ${}^*\Gamma$ and ${}^*G$, and so ${}^*G/{}^*\Gamma$ is not locally compact. However, the restriction to internal objects allows us to carefully work out an induction procedure as needed. The induced module $\V$ also has an internal continuity property (defined in \S\ref{ssec-basicG}) that we establish in \S\ref{sec-induc}.
		\item Since $P$ is amenable, the bounded cohomology $\Hb^{\bullet}(G,V)$ can be computed as the cohomology of the complex 
		$$\begin{tikzcd}
		0 \arrow[r] & V^G  \arrow[r, "\epsilon"] & \Linf(G/P,V)^G  \arrow[r, "d^0"]&  \Linf((G/P)^2,V)^G \arrow[r, "d^1"] & \Linf((G/P)^3,V)^G \arrow[r, "d^2"]   & \dots
		\end{tikzcd}$$ 
		Furthermore, for a parabolic subgroup $P \leq Q \leq G$, the bounded cohomology $\Hb^{\bullet}(Q,V)$ can be computed as the cohomology of the complex 
		$$\begin{tikzcd}
		0 \arrow[r] & V^Q  \arrow[r, "\epsilon"] & \Linf(G/P,V)^Q  \arrow[r, "d^0"]&  \Linf((G/P)^2,V)^Q \arrow[r, "d^1"] & \Linf((G/P)^3,V)^Q \arrow[r, "d^2"] & \dots
		\end{tikzcd}$$ 
		These steps too can be reworked in the asymptotic setting, analogous to the procedure in bounded cohomology theory, again thanks to the restriction on internality, and this is described in \S\ref{sec-coho}. 
		\item The motivation behind the preceding step is that we have at our disposal the following double ergodicity theorem: let $V$ be a continuous $G$-module and $\alpha \in \Linf\left((G/P)^2,V\right)^G$, then $\alpha$ is essentially constant. This theorem follows from the Mautner's lemma: let $V$ be a continuous $G$-module and $N \leq G$ be a non-compact subgroup, then $V^N = V^G$. Both these results are particularly useful in the context of $\Hb^2(G,V)$.\\
		In our setting, there are particular difficulties in obtaining an
		analoguous Maunter lemma due to the asymptotic nature of our model. We overcome them for our specific Banach module $\V$ by applying a suitable correction to exact cocycles using structure results for the semisimple group $G$; this is worked out in \S\ref{sec-mainproof}.
		\item For the parabolic subgroups $Q_1$ and $Q_2$ as in the hypothesis, an inflation-restriction sequence argument implies that $\Hb^2(Q_i,V)=\Hb^2(Q_i/N_i,V^{N_i})$ for $N_i$ being the (amenable) radical of $Q_i$ for $i \in \{1,2\}$. By Mautner's lemma and the hypothesis that $\Hb^2(Q_i,\R)=0$, one concludes that $\Hb^2(Q_i,V)=0$.\\
		The analogous hypothesis in our asymptotic setting is Property-$G(\mathcal{Q}_1,\mathcal{Q}_2)$, where the conditions that $\Hb^2(Q_i,\R)=0$ and $\Hb^3(Q_i,\R)$ is Hausdorff together are used to conclude that $\Ha^2(Q_i,\V)=0$ in \S\ref{ssec-invariants}.
		\item Let $\omega \in \Linf\left((G/P)^3,V\right)^G$ be a bounded $2$-cocycle for $G$. Since $\Hb^2(Q_1,V)=\Hb^2(Q_2,V)=0$, there exist $\alpha_1 \in \Linf\left((G/P)^2,V\right)^{Q_1}$ and $\alpha_2 \in \Linf\left((G/P)^2,V)\right)^{Q_2}$ such that $\omega=d\alpha_1=d\alpha_2$. In particular, $\alpha_1-\alpha_2$ is a $1$-cocycle for $Q_1 \cap Q_2$. Since $P \leq Q_1\cap Q_2$, $\alpha_1-\alpha_2$ is a $1$-cocycle for $P$ as well. But since $\Hb^1(P,V)=0$, one can show using the double ergodicity theorem, that $\alpha_1=\alpha_2$ ($=\alpha$, say), implying that $\alpha$ is equivariant with respect to both $Q_1$ and $Q_2$, and hence is $G$-equivariant. Thus $\omega=d\alpha$ for $\alpha \in \Linf\left((G/P)^2,V\right)^G$, proving that $\Hb^2(G,V)=0$. This step too goes through in our setting once we have our asymptotic variant of the ergodicity theorem used in the classical case. 
	\end{itemize}
	While in this paper, we focus on using the theory of asymptotic cohomology to prove uniform $\mathfrak{U}$-stability for lattices in semisimple groups, the framework and tools developed  here have also been used in subsequent work \cite{francesco} to prove uniform $\mathfrak{U}$-stability for other classes of groups such as lamplighter groups $\Gamma \wr \Lambda$ where $\Lambda$ is infinite and amenable, as well as several groups of dynamical origin such as Thompson's group $F$. The techniques there too are analogous to corresponding vanishing results of bounded cohomology in \cite{lamplighter}, yet again highlighting the connections between the theories of bounded cohomology and asymptotic cohomology.

	\subsection*{Outline of the Paper}
	We begin with the much simple setting of uniform stability with respect to the fixed group $U(1)$ (equipped with the absolute value metric) in \S\ref{sec-u1}. In this case, we can reduce the question of uniform $U(1)$-stability of $\Gamma$ to the injectivity of the comparison map $c:\Hb^2(\Gamma,\R) \to H^2(\Gamma,\R)$. After recalling how classical results from \cite{fujiwara} imply that lattices in Lie groups of rank $1$ are not uniformly $U(1)$-stable, we then show that lattices in higher rank Lie groups are uniformly $U(1)$-stable. While the connection between uniform $U(1)$-stability of a group $\Gamma$ and the second bounded cohomology $\Hb^2(\Gamma,\R)$ is classical, it motivates the idea of using the logarithm map on an almost homomorphism to construct a bounded $2$-cocycle of $\Gamma$ in a Banach space, which we develop in \S\ref{ssec-logarithm} for a more general setting.\\
	
	In \S\ref{sec-prelims}, we begin by defining the basic notions in full detail and rigor in \S\ref{ssec-definitions}. In particular we focus on interpreting stability in terms of sequences of maps and asymptotic homomorphisms, which we then refine further in \S\ref{ssec-ultraproducts} in the language of non-standard analysis. This formulation will allow us to reinterpret the question of uniform stability as a \emph{homomorphism lifting problem} on the lines of the approach used in \cite{DCGLT} \cite{arz}. While such a lifting problem motivates the attempt at constructing a cohomology, an obstacle here is that the kernel of the extension is not abelian. This issue is resolved in \cite{DCGLT} by considering lifts in small increments so that the kernel at each step is abelian. This idea, known as \emph{defect diminishing}, is explored in \S\ref{ssec-liftings}, and can be shown to imply uniform stability.\\
	
	In \S\ref{sec-asymgamma} we begin by focusing on a particular family of metric groups for which defect diminishing corresponds to a homomorphism lifting problem with an abelian kernel. This is the family of unitary groups equipped with  submultiplicative norms, discussed in \S\ref{ssec-abeliankernel}. In \cite{DCGLT}, defect diminishing combined with ordinary group cohomology with coefficients in the abelian kernel (which turns out to be a Banach $\Gamma$-module) is sufficient to study non-uniform stability. But in our setting, the uniformity condition involves subtleties that necessitate transfering to the internal Lie algebra with an internal \emph{asymptotic}-action of ${}^*\Gamma$ (the ultrapower of $\Gamma$), and the formulation of an \say{internal and asymptotic} bounded cohomology with cofficients in that internal space, denoted $\W$. This is motivated in \S\ref{ssec-logarithm}, and we conclude the section by demonstrating the machinery built so far in (re)proving the result of Kazhdan \cite{Kaz} that discrete amenable groups are Ulam stable.\\
	
	\S\ref{sec-coho} begins with the rigorous definitions of an internal Banach spaces and asymptotic ${}^*G$-modules for a locally compact group $G$ (defining our notions in the category of topological groups is necessary in order to deal with lattices in Lie groups, which shall involve an Eckmann-Shapiro induction of cohomologies explored in \S\ref{sec-induc}) in \S\ref{ssec-basicG}.  In \S\ref{ssec-coho}, we formally define $\Ha^{\bullet}(G,\V)$ using an internal $\Linf$-spaces, and study some functorial properties and different cochain complexes that can be used to compute $\Ha^{\bullet}(G,\V)$ in \S\ref{ssec-subg}. Many of the techniques used here have parallels in the theory of continuous bounded cohomology as in \cite{bookMonod}.\\
	
	In \S\ref{sec-induc}, we restrict our attention to a lattice $\Gamma$ in a Lie group $G$, and begin by studying an intermediate structure $\L_b^{\infty}(G,\W)^{\sim {}^*\Gamma}$ that is not quite the induction module $\V=\L^{\infty}(D,\W)$ in our machinery, but comes with an internal (true) ${}^*G$-action (as opposed to an action upto infinitesimals). This structure leads to useful results proved in \S\ref{ssec-trueAction}, and the actual induction module and an Eckmann-Shapiro induction procedure are discussed in \S\ref{ssec-shapiro}. We conclude the section with a continuity property of our module $\V$, which we use to define contracting elements to lay the groundwork for a Mautner's lemma to be proved in the next section. \\
	
	Finally, in \S\ref{sec-mainproof}, we come to the higher rank semisimple groups of interest, and begin by discussing an analogue of the Mautner's lemma in our setting, along with double ergodicity lemmas in \S\ref{ssec-invariants}, and use these to prove that $\Ha^2(Q,\V)=0$ for $Q\le G$ being a proper parabolic subgroup. All these techniques come together in \S\ref{ssec-mainproof} where we prove that $\Ha^2(G,\V)=0$ for semisimple groups $G$ that have Property-$G(\mathcal{Q}_1,\mathcal{Q}_2)$, thus implying uniform stability with a linear estimate for lattices in such groups. In \S\ref{ssec-semisimple}, we list out classes of simple groups $G$ that have Property-$G(\mathcal{Q}_1,\mathcal{Q}_2)$, and conclude in \S\ref{sec-conclusions} with some discussion on the limitations of our method and related open questions.
	
	\subsection*{Acknowledgements}
	The second author acknowledges with gratitude the hospitality and support of the Fields Institute (Toronto) and the Institute for Advanced Study (Princeton) where part of this work was carried on, as well as a grant by the European Research Council (ERC) under the European Union’s Horizon 2020 (grant agreement No $882751$). The results presented here are part of the fourth author's PhD thesis, also supported by the same grant. The authors would like to thank Andrei Rapinchuk for his guidance in the proof of \cref{rapinchuk}.\\
	
	\noindent \textit{This paper is dedicated to the memory of Bob Zimmer in honor of his remarkable achievements and influence on the study of rigidity of lattices in semisimple Lie groups.}
	
	\newpage
	
	\tableofcontents

	\newpage
	
	\section{$U(1)$-Stability of Groups}\label{sec-u1}
	We begin with a simpler setting, namely uniform stability of a discrete group $\Gamma$ with respect to the abelian group $U(1)$. This section can be read independently of the rest. 
	\begin{definition}
		For $\epsilon>0$, a map $\phi:\Gamma \to U(1)$ is said to be an $\epsilon$-homomorphism if 
		\begin{equation*}
		\sup_{x,y \in \Gamma}\lvert \phi(xy)-\phi(x)\phi(y) \rvert \leq \epsilon
		\end{equation*}
	\end{definition}
	The value $\sup_{x,y \in \Gamma}\lvert \phi(xy)-\phi(x)\phi(y) \rvert$ is called the \emph{defect} of $\phi$. 
	\begin{definition}
		A group $\Gamma$ is said to be uniformly $U(1)$-stable if for every $\delta>0$, there exists $\epsilon>0$ such that if $\phi:\Gamma \to U(1)$ is an $\epsilon$-homomorphism, there exists a homomorphism $\psi:\Gamma \to U(1)$ such that $\sup_{x \in \Gamma} \lvert \phi(x)-\psi(x) \rvert < \delta$. 
	\end{definition}
	A closely related notion is that of a quasimorphism. A \emph{quasimorphism} is a map $f:\Gamma \to \R$ such that there exists $D \geq 0$ such that for every $x,y \in \Gamma$,
	$$\lvert f(x)+f(y)-f(xy) \rvert \leq D$$
	Let $QM(\Gamma)$ denote the space of all quasimorphisms of $\Gamma$. A trivial example of a quasimorphism is obtained by perturbing a homomorphism by a bounded function, and such quasimorphisms form the subspace $Hom(\Gamma,\R) \oplus C_b(\Gamma,\R)$ of $QM(\Gamma)$ (where $C_b(\Gamma,\R)$ denotes the space of all bounded functions from $\Gamma$ to $\R$) A quasimorphism that is not at a bounded distance from any homomorphism is called a \emph{non-trivial quasimorphism}. In this setting, a question analogous to uniform stability is whether every quasimorphism is at a bounded distance from a homomorphism. That is, is $QM(\Gamma) = Hom(\Gamma,\R)\oplus C_b(\Gamma,\R)$?\\
	It is known that every quasimorphism class contains a unique \emph{homogenous} quasimorphism (a quasimorphism $f$ is said to be homogenous if for every $g \in \Gamma$ and $n \in \N$, $f(g^n)=nf(g)$). Suppose $f$ is a homogenous quasimorphism of $\Gamma$ that is not a homomorphism. Then its exponent $\mu \coloneq e^{2\pi i \epsilon f}:\Gamma \to U(1)$ is a function whose defect can be made arbitrarily small (by choosing $\epsilon \to 0$), but whose distance from homomorphisms is bounded below by a positive constant.
	\begin{proposition}[\cite{BOT} \cite{monodshalom-orbit}]\label{quasi1}
		If $\Gamma$ admits a non-trivial quasimorphism, then $\Gamma$ is not uniformly $U(1)$-stable. 
	\end{proposition}
	Quasimorphisms are also closely related to group cohomology (this is well-known and classical, see \cite{frig},\cite{inviteMonod} for
	references). Observe that given a quasimorphism $f:\Gamma \to \R$, the map 
	$$df:\Gamma \times \Gamma \to \R$$
	$$df(x,y) = f(x)+f(y)-f(xy)$$
	is a $2$-coboundary for $\Gamma$ in $\R$, while also being a bounded function that satisfies the $2$-cocycle condition (and hence a \emph{bounded $2$-cocycle}). This leads to the following characterization of quasimorphisms classes (\cite{frig}, \cite{inviteMonod}): 
	\begin{proposition}\label{quasi2}
		The kernel, denoted $E\Hb^2(\Gamma,\R)$, of the comparison map $c:\Hb^2(\Gamma,\R) \to H^2(\Gamma,\R)$ is isomorphic to the space of quasimorphisms modulo the suabspace of trivial quasimorphisms. 
		$$E\Hb^2(\Gamma,\R) \cong \frac{QM(\Gamma)}{C_b(\Gamma,\R) \oplus Hom(\Gamma,\R)}$$
	\end{proposition}
	Hence to show that $\Gamma$ is not $U(1)$-stable, it is sufficient to show that the comparison map $c:\Hb^2(\Gamma,\R) \to H^2(\Gamma,\R)$ is not injective. In \cite{fujiwara}, it is shown that for a lattice in a rank one semisimple Lie group, its comparison map has non-zero kernel, and hence
	\begin{theorem}
		Let $H$ be a semisimple group (as in the introduction) of rank one, and let $\Gamma$ be a lattice in $H$. Then $\Gamma$ is not uniformly $U(1)$-stable.
	\end{theorem} 
	The rest of this section is devoted to showing that higher rank lattices are $U(1)$-stable.  For this goal the bounded cohomology plays a central role. For simplicity, we endow $U(1)$ with the distance coming from seeing it as $\R/\Z$ (so we just have to apply a trigonometric formula if we prefer the norm distance).\\
	Recall the long exact sequences associated to $\Z \to \R \to \R/\Z$ for $H^n$ and $\Hb^n$. Together with the comparison maps, we get a commutative diagram:
	$$\begin{tikzcd}
	\cdots \arrow[r] & H^1(\Gamma, \R/\Z)   \arrow[r] \arrow[d] & \Hb^2(\Gamma, \Z)  \arrow[r] \arrow[d] &  \Hb^2(\Gamma, \R) \arrow[r] \arrow[d] & H^2(\Gamma, \R/\Z) \arrow[r] \arrow[d]  & \cdots\\
	\cdots \arrow[r] & H^1(\Gamma, \R/\Z)  \arrow[r] & H^2(\Gamma, \Z)   \arrow[r]&  H^2(\Gamma, \R) \arrow[r] & H^2(\Gamma, \R/\Z) \arrow[r]  & \cdots
	\end{tikzcd}$$
	
	Consider the subset $K\subseteq \Hb^2(\Gamma, \R)$ given by
	$$K=\mathrm{Ker}\Big(\Hb^2(\Gamma, \R) \to H^2(\Gamma, \R/\Z) \Big) = \mathrm{Image}\Big(\Hb^2(\Gamma, \Z) \to \Hb^2(\Gamma, \R)\Big)$$
	Recall that $\Hb^2(\Gamma, \R)$ is a Banach space; we can therefore consider the norm of elements in $K$.
	
	\begin{proposition}\label{prop:gap}
		The following are equivalent for a group $\Gamma$.
		\begin{enumerate}
			\item (Lipschitz $U(1)$ stability): Every $\epsilon$-representation to $U(1)$ with $\epsilon$ small enough is at distance less than $\epsilon$ from a representation. \label{pt:lipstable}
			\item (Linear $U(1)$ stability): There is a constant $c$ such that every $\epsilon$-representation to $U(1)$ is at distance less than $c \epsilon$ from a representation. \label{pt:linstable}
			\item ($U(1)$ stability): For each $\delta>0$ there is $\epsilon>0$ so that every $\epsilon$-representation to $U(1)$ is at distance less than $\delta$ from a representation. \label{pt:stable}
			\item ($U(1)$ $1/3$-stability): There is $\delta <1/3$ such that every $\epsilon$-representation to $U(1)$ with $\epsilon$ small enough is at distance less than $\delta$ from a representation. \label{pt:stable:13}
			\item (Cohomology gap): Non-zero elements of $K$ have norm bounded below by a positive constant.\label{pt:K}
		\end{enumerate}
	\end{proposition}
	
	\begin{remark}
		The kernel of the comparison map $\Hb^2(\Gamma, \R) \to H^2(\Gamma, \R)$ is contained in $K$, see above diagram. Since this kernel is a vector space, point \eqref{pt:K} fails as soon as this kernel is non-zero. This explains why quasimorphisms imply non-stability: it is a special case of the above as quasimorphisms \say{live} in $K$.
	\end{remark}
	
	\begin{proof}
		Trivially \eqref{pt:lipstable}$\Rightarrow$\eqref{pt:linstable}$\Rightarrow$\eqref{pt:stable}$\Rightarrow$\eqref{pt:stable:13}.
		
		We prove \eqref{pt:stable:13}$\Rightarrow$\eqref{pt:K} for the positive constant $\epsilon$ as in \eqref{pt:stable:13} which, without loss of generality, can be made to satisfy $\epsilon  < 1-3 \delta$. Consider a class $[\omega]$ in $K$ with $\|[\omega]\| < \epsilon$. We can choose the representative $\omega\in\ell^\infty(\Gamma^2)$ such that $\|\omega\|_\infty < \epsilon$. Since $[\omega]$ is in the kernel to $H^2(\Gamma, \R/\Z)$, there is $f\colon \Gamma\to \R$ such that $\omega \equiv d f$ modulo $\Z$. The map $\pi=\exp(2\pi i f)$ is an $\epsilon$-representation. Thus $\pi$ is at distance less than $\delta$ of an actual representation, which means that there is $b\colon \Gamma\to [-\delta, \delta]$ with $d(f+b) \equiv 0$ modulo $\Z$. This means that $\omega + d b$ is integer-valued. On the other hand, $\|d b\|_\infty$ is at most $3 \delta$. Thus, since $\epsilon + 3\delta < 1$ we deduce that $\omega + d b$ actually vanishes and thus $[\omega]$ is zero in $K$.
		
		We now prove \eqref{pt:K}$\Rightarrow$\eqref{pt:lipstable}. We show it for $0<\epsilon<1/4$ smaller than the constant given by~\eqref{pt:K}. For any $\pi\colon \Gamma\to U(1)$, choose $f\colon \Gamma\to\R$ such that $\pi = \exp(2 \pi i f)$. If $\pi$ is an $\epsilon$-representation, then there is $\omega\colon \Gamma^2\to [-\epsilon, \epsilon]$ such that $\omega \equiv d f$ modulo $\Z$. In particular, $d\omega\equiv 0$ modulo $\Z$, i.e.\ $d\omega$  is integer-valued. Given that $\|d\omega \|_\infty \leq 4 \epsilon$, we have in fact $d \omega = 0$ and hence we obtain a class $[\omega]$ in $K$ of norm less than $\epsilon$. Therefore, we can assume that $[\omega]$ is trivial, which means $\omega = d b$ for some $b\in \ell^\infty(\Gamma)$.
		
		In fact the operator $d$ on $\ell^\infty(\Gamma)$ has norm one: this was first observed by~\cite[p.~468]{mitsu} in a special case; the short proof is given in general in (the proof of) Corollary~2.7 in ~\cite{matsu}. We can therefore choose $b$ such that $\| b\|_\infty \leq \epsilon$. Now $\exp(2 \pi i (f-b))$ is a representation at distance less than $\epsilon$ of $\pi$. 
	\end{proof}
	
	\begin{theorem}\label{thm:U1:qm}
		Let $\Gamma$ be a group such that $\Hb^2(\Gamma, \R)$ is finite-dimensional and  injects into $H^2(\Gamma, \R)$. Then $\Gamma$ is uniformly $U(1)$-stable. 
	\end{theorem}
	
	\begin{lemma}\label{lem:Specker}
		Let $A$ be an abelian group and let $B$ be a subgroup of $Hom_\Z(A, \Z)$. If the image of $B$ in $Hom_\Z(A, \R)$ spans a subspace of finite $\R$-dimension $d$, then $B$ is free abelian of rank $d$. 
	\end{lemma}
	
	\begin{proof}[Proof of \cref{lem:Specker}]
		We can consider $B$ as a subgroup spanning a space of finite $\Q$-dimension $d$ in $Hom_\Z(A, \Q)$. Viewing $A$ as a quotient of a free abelian group on some set $X$, the group $Hom_\Z(A, \Z)$ is contained in $\Z^X$. While $\Z^X$ is not free abelian in general, in 
		Specker (Satz~I p.~133 in~\cite{specker}) it is proved that countable subgroups of $\Z^X$ are free abelian . To be more precise,
		Specker proved it for $X$ countable but his proof works in general; alternatively, the statement immediately reduces to the case $X$ countable by taking a subset of $X$ separating the points of $B$.  Our $B$, being a subgroup of a finite dimensional $\Q$-space, is countable and hence free (from the above mentioned result from ~\cite{specker}).  Spanning a $d$-dimensional $\Q$-vector space, its rank must be also $d$. 
	\end{proof}
	
	\begin{proof}[Proof of \cref{thm:U1:qm}]  It suffices to show that $\Gamma$ satisfies the cohomology gap --- i.e., condition~\eqref{pt:K} of \cref{prop:gap}.
		Let $\tilde B$ denote the image of $\Hb^2(\Gamma, \Z)$ in $H^2(\Gamma, \Z)$. Thus the image of $\tilde B$ in $H^2(\Gamma, \R)$ is precisely the image of $K$.

		$$\begin{tikzcd}
		\Hb^2(\Gamma, \Z)  \arrow[r] \arrow[d] & K \subseteq \Hb^2(\Gamma, \R) \arrow[d]\\
		\tilde B \subseteq H^2(\Gamma, \Z) \arrow[r] & H^2(\Gamma, \R)
		\end{tikzcd}$$
		
		Let $d<\infty$ be the dimension of the space spanned by $K$ in $\Hb^2(\Gamma, \R)$. To prove the cohomology gap~\eqref{pt:K}, we first claim that $K$ is free abelian of rank $d$. This claim implies that $K$ is
		discrete in the finite dimensional space $\Hb^2(\Gamma, \R)$ since it spans
		a space of dimension $d$ (we use here the comparison with the standard
		lattice $\Z^d$ in $\R^d$ and the fact that linear maps are continuous in
		finite dimensions). Then, discreteness implies the desired gap. To prove the claim that $K$ is free abelian of rank $d$, we can work with $H^2(\Gamma, \R)$ since $\Hb^2(\Gamma, \R)$ injects there. The universal coefficient theorem gives the following commutative diagram with exact rows.
		$$\begin{tikzcd}
		0  \arrow[r]  & Ext_\Z^1(H_1(\Gamma, \Z), \Z) \arrow[r] \arrow[d] & H^2(\Gamma, \Z) \arrow[r] \arrow[d] & Hom_{\Z}(H_2(\Gamma, \Z), \Z) \arrow[r] \arrow[d] & 0\\
		& 0 \arrow[r] & H^2(\Gamma, \R) \arrow[r] & Hom_{\Z}(H_2(\Gamma, \Z), \R) \arrow[r] & 0
		\end{tikzcd}$$

		Therefore it suffices to show that the image $B$ of $\tilde B$ in $Hom_{\Z}(H_2(\Gamma, \Z), \Z)$ is free abelian of rank $d$, which follows from \cref{lem:Specker} applied to $A=H_2(\Gamma, \Z)$. 
	\end{proof}
	
	Since finitely presented groups have finite-dimensional $H^2$, we obtain the following clean necessary and sufficient condition:
	
	\begin{corollary}
		Let $\Gamma$ be a finitely presented group. Then $\Gamma$ is uniformly $U(1)$-stable if and only if every quasi-morphism of $\Gamma$ is at bounded distance from a homomorphism.
	\end{corollary}
	
	\begin{proof}  As explained above if $QM(\Gamma) \neq 0$ (i.e., not every quasi-morphism of $\Gamma$ is at bounded distance from a homomorphism) then $\Gamma$ is not $U(1)$-stable. On the other hand, if $QM(\Gamma) = 0$, then $Ker (\Hb^2 (\Gamma, \R) \to H^2(\Gamma, \R)) = 0$, i.e., $\Hb^2(\Gamma,\R)$ injects into $H^2(\Gamma, \R)$.  If in addition $\Gamma$ is finitely presented $H^2 (\Gamma,\R)$ is of finite dimension.  Thus all the assumptions of \cref{thm:U1:qm} are satisfied and $\Gamma$ is uniformly $U(1)$-stable. 
	\end{proof}
	
	Now, back to the case in which $\Gamma$ is a lattice in a higher rank group.  For such $\Gamma$, Burger and Monod \cite{burgerMonod} showed that $\Hb^2 (\Gamma,\R)$ injects into $H^2 (\Gamma,\R)$.  Such $\Gamma$ is \say{usually} finitely presented (and then we can apply the last corollary) except when $\Gamma$ is a lattice in (rank 2) semisimple Lie group of positive characteristic.  But in this case $\Hb^2(\Gamma,\R)$ is anyway zero, by another result of Burger and Monod \cite{burgerMonod}, so \cref{thm:U1:qm} applies and we can deduce:
	
	\begin{theorem}\label{1-stable}
		If $\Gamma$ is a higher rank lattice then it is uniformly $U(1)$-stable.
	\end{theorem}
	
	The main result of this paper will be a far reaching extension of the above theorem (with some additional assumptions on the lattice $\Gamma$) to a larger family of metric groups (namely any unitary group $U(n)$ equipped with any submultiplicative matrix norm $\|\cdot\|$). In this general case, bounded cohomology theory would not suffice, so we will motivate and build a more appropriate cohomology theory in the upcoming sections.\\
	
	We conclude this section with an observation in the case of groups of Hermitian type:
	\begin{proposition}\label{hermitian}
		Let $G$ be a split Lie group of Hermitian type (for example, $Sp(2m,\R)$), with universal central extension $\tilde{G}$. Let $\Gamma$ be a cocompact lattice in $G$, and $\tilde{\Gamma}$ be its preimage in $\tilde{G}$. Then $\tilde{\Gamma}$ is not uniformly $U(1)$-stable (and hence, not uniformly $\mathfrak{U}$-stable).
	\end{proposition}
	\begin{proof}
		Consider the non-trivial $2$-cocycle $\alpha \in H^2(\Gamma,\Z)$ corresponding to the central extension $\tilde{\Gamma}$ of $\Gamma$. From \cite{guichardet}, we know that $\alpha$ is actually an element of $\Hb^2(\Gamma,\R)$ that does not vanish in $H^2(\Gamma,\R)$. In particular, $\alpha$ is not contained in the kernel of the comparison map $c:\Hb^2(\Gamma,\R) \to H^2(\Gamma,\R)$.\\
		Since the extension $\tilde{\Gamma}$ of $\Gamma$ is central (in particular, has amenable kernel), $\alpha$ has a pullback $\tilde{\alpha} \in \Hb^2(\tilde{\Gamma},\R)$ which is a non-trivial bounded $2$-cocycle. However, $\tilde{\alpha}$ is trivial in cohomology, hence $\tilde{\alpha}$ is an element of the kernel of the comparison map $c:\Hb^2(\tilde{\Gamma},\R) \to H^2(\tilde{\Gamma},\R)$. 
		Thus, from Propositions \cref{quasi1} and \cref{quasi2}, we conclude that $\tilde{\Gamma}$ is not uniformly $U(1)$-stable.
	\end{proof}
	\begin{remark}
		A non-trivial quasimorphism in the above proof of \cref{hermitian} can be explicitly described: let $j:\Gamma \to \tilde{\Gamma}$ be a section corresponding to the cocycle $\alpha \in H^2(\Gamma,\Z)$ so that as a set, $\tilde{\Gamma} = \Z \times j(\Gamma)$. Consider the map $\phi:\tilde{\Gamma} \to \R$ defined to be $\phi(m,j(\gamma)) \coloneqq m$. One can check that this is a non-trivial quasimorphism of $\tilde{\Gamma}$, and note that this map is \say{trivial} on $\Gamma$ and we do know, from \cref{1-stable}, that $\Gamma$ is indeed uniformly $U(1)$-stable if it has rank at least $2$, even if $\tilde{\Gamma}$ is not. 
	\end{remark}
	
	\section{Preliminaries and Basic Constructions}\label{sec-prelims}
	In this section, we establish the connection between uniform stability and a homomorphism lifting problem, which is then further explored in \S\ref{sec-asymgamma} for discrete groups, and in \S\ref{sec-coho} for topological groups.\\
	In the first \S\ref{ssec-definitions}, we define our central notion of uniform stability, and describe it using sequences of maps. This then allows for a reformulation of the notion of stability as a homomorphism lifting problem using the language of ultrafilters in \S\ref{ssec-ultraproducts}. Finally, in \S\ref{ssec-liftings}, we introduce the idea of defect diminishing, which is a relaxation of the lifting problem with abelian kernels. This property can naturally be related to a cohomological problem that is then studied in the subsequent \S\ref{sec-asymgamma}. 
	\subsection{Uniform Stability and Asymptotic Homomorphisms}\label{ssec-definitions}
	Let $\Gamma$ be a countable discrete group, and let $(G,d_G)$ be a metric group (that is, a group $G$ equipped with a bi-invariant metric $d_G$). We use the metric to define the (uniform) distance between maps from $\Gamma$ to $G$ as follows: for $f_1,f_2:\Gamma \to G$, the distance between $f_1$ and $f_2$, denoted $dist_{\Gamma,G}(f_1,f_2)$, as
	$$dist_{\Gamma,G}(f_1,f_2) \coloneq \sup_{x \in \Gamma} d_G\left(f_1(x),f_2(x)\right)$$ 
	This allows us to define the distance of a function $f:\Gamma \to G$ from a homomorphism as follows:
	\begin{definition}
		The \textbf{homomorphism distance} of a function $f:\Gamma \to G$, denoted $D_{\Gamma,G}(f)$, is defined as
		$$D_{\Gamma,G}(f) \coloneq \inf \{ dist_{\Gamma,G}(f,\psi) \space : \space \psi \in Hom(\Gamma,G) \}$$
		The function $f$ is said to be $\delta$-close to a homomorphism if $D_{\Gamma,G}(f) \leq \delta$.  
	\end{definition}
	There is another invariant of a function $f$ that also quantifies its distance from being a homomorphism.
	\begin{definition}
		For a function $f:\Gamma \to G$, we define its (uniform) \textbf{defect} $def_{\Gamma,G}(f)$ as
		$$def_{\Gamma,G}(f) \coloneq \sup_{x,y \in \Gamma} d_G\left( f(xy), f(x)f(y) \right)$$
		The function $f$ is said to be an \textbf{$\epsilon$-homomorphism} if $def_{\Gamma,G}(f) \leq \epsilon$.
	\end{definition}
	Note that a priori both $def_{\Gamma,G}(f)$ and $D_{\Gamma,G}(f)$ could be $\infty$. It is easy to show (by the triangle inequality) that $def_{\Gamma,G}(f) \leq 3D_{\Gamma,G}(f)$ for any function $f$, so if $f$ is close to a homomorphism, then it has small defect. Uniform stability is the question of whether the converse is true: is any function with small defect necessarily close to a homomorphism?\\
	Uniform stability is usually studied with respect to not just one metric group $(G,d_G)$ but a family $\mathcal{G}$ of metric groups. In this case, we can define 
	$$F_{\Gamma,\mathcal{G}}:[0,\infty] \to [0,\infty]$$
	$$F_{\Gamma,\mathcal{G}}(\epsilon) \coloneq \sup_{G \in \mathcal{G}} \sup_{f} \{ D_{\Gamma,G}(f) \space : \space def_{\Gamma,G}(f) \leq \epsilon\}$$
	\begin{definition}
		The group $\Gamma$ is said to be \textbf{uniformly $\mathcal{G}$-stable} if 
		$$\lim \limits_{\epsilon \to 0^{+}} F_{\Gamma,\mathcal{G}}(\epsilon) = 0$$
	\end{definition}
	A further refinement of the above definition involves a quanitification of the above convergence:
	\begin{definition}
		The group $\Gamma$ is uniformly $\mathcal{G}$-stable \emph{with a linear estimate} if $\exists \epsilon_0>0$ and $\exists M\geq 0$ such that $\forall \epsilon<\epsilon_0$ and $\forall G \in \mathcal{G}$, every $\epsilon$-homomorphism $\phi:\Gamma \to G$ is $M\epsilon$-close to a homomorphism.    
	\end{definition}
	
	It is helpful to rephrase these notions also in terms of sequences of maps, especially since we shall further refine this view in \S\ref{ssec-ultraproducts} and \S\ref{ssec-liftings}. Consider a sequence of functions $\{\phi_n:\Gamma \to G_n\}_{n \in \N}$ where $G_n \in \mathcal{G}$ for every $n \in \N$. 
	\begin{itemize}
		\item A sequence $\{\phi_n:\Gamma \to G_n\}_{n \in \N}$ is said to be a (uniform) \textbf{asymptotic homomorphism} of $\Gamma$ to $\mathcal{G}$ if $\lim \limits_{n \to \infty}def_{\Gamma,G_n}(\phi_n) = 0$. 
		\item A sequence $\{\phi_n:\Gamma \to G_n\}_{n \in \N}$ is said to be (uniformly) \textbf{asymptotically close to a homomorphism} if $\lim \limits_{n \to \infty} D_{\Gamma,G_n}(\phi_n) = 0$. 
	\end{itemize}
	It is easy to see that a group $\Gamma$ is uniformly $\mathcal{G}$-stable iff every asymptotic homomorphism of $\Gamma$ to $\mathcal{G}$ is asymptotically close to a homomorphism.\\
	Uniform $\mathcal{G}$-stability with a linear estimate too can be rephrased in terms of sequences of maps. Recall the Laundau big-O notation: for sequences $\{x_n\}_{n \in \N}$ and $\{y_n\}_{n \in \N}$ of positive real numbers, we denote $x_n=O(y_n)$ if there exists a constant $C \geq 0$ and $N \in \N$ such that for all $n \geq N$, $x_n \leq C y_n$. We denote by $x_n=o(y_n)$ if there exists a sequence $\{\epsilon_n\}_{n \in \N}$ with $\lim_{n \to \infty}\epsilon_n=0$ and such that $x_n=\epsilon_n y_n$. Firstly, note that if $\Gamma$ is uniformly $\mathcal{G}$-stable with a linear estimate, then for any asymptotic homomorphism $\{\phi_n:\Gamma \to G_n\}_{n \in \N}$ of $\Gamma$ to $\mathcal{G}$, $D_{\Gamma,G_n}(\phi_n) =  O\left(def_{\Gamma,G_n}(\phi_n)\right)$. The following lemma shows that the converse is also true:
	\begin{lemma}\label{l1}
		The group $\Gamma$ is uniformly $\mathcal{G}$-stable with a linear estimate iff for every asymptotic homomorphism $\{\phi_n:\Gamma \to G_n\}_{n \in \N}$ of $\Gamma$ to $\mathcal{G}$, 
		$$D_{\Gamma,G_n}(\phi_n) =  O\left(def_{\Gamma,G_n}(\phi_n)\right)$$
	\end{lemma}
	\begin{proof}
		Suppose $\Gamma$ is not uniformly $\mathcal{G}$-stable with a linear estimate. Then for any $M>0$ and any $\epsilon>0$, there exists a map $\phi:\Gamma \to G$ (for some $G \in \mathcal{G}$) such that $def_{\Gamma,G}(\phi) \leq \epsilon$ and $D_{\Gamma,G}(\phi) > M\epsilon$. Now consider a sequence $\{M_n\}_{n \in \N}$ with $\lim_{n \to \infty}M_n = \infty$ and a sequence $\{\epsilon_n\}_{n \in \N}$ with $\lim_{n \to \infty}\epsilon_n = 0$. For each $n \in \N$, let $\phi_n:\Gamma \to G_n$ be the map with $def_{\Gamma,G_n}(\phi_n) \leq \epsilon_n$ but $D_{\Gamma,G_n}(\phi_n) > M_n\epsilon_n$. Then $\{\phi_n\}_{n \in \N}$ is an asymptotic homomorphism of $\Gamma$ to $\mathcal{G}$ such that $D_{\Gamma,G_n}(\phi_n) $ is clearly not  $O\left(def_{\Gamma,G_n}(\phi_n)\right)$. 
	\end{proof}
	We shall now conclude this section by informally introducing the following relaxation of the hypothesis of \cref{l1}, which we shall call \emph{asymptotic defect diminishing}. The idea is to look for improvements to the asymptotic homomorphism, rather than a true homomorphism. 
	\begin{definition}
		The group $\Gamma$ is said to have the \textbf{asymptotic defect diminishing} property with respect to the family $\mathcal{G}$ if for any asymptotic homomorphism $\{\phi_n:\Gamma \to G_n\}_{n \in \N}$, there exists an asymptotic homomorphism $\{\psi_n:\Gamma \to G_n\}_{n \in \N}$ such that  
		\begin{itemize}
			\item The defect $def_{\Gamma,G_n}(\psi_n) = o\left(def_{\Gamma,G_n}(\phi_n)\right)$.
			\item The distance $dist_{\Gamma,G_n}(\phi_n,\psi_n) = O\left(def_{\Gamma,G_n}(\phi_n)\right)$.
		\end{itemize}
	\end{definition}
	The following results motivate this notion, which we shall formally study in more detail in \S\ref{ssec-ultraproducts} in an ultraproduct setting.
	\begin{lemma}\label{dim}
		Suppose $\Gamma$ has the asymptotic defect diminishing property with respect to $\mathcal{G}$. Then there exists $\epsilon_0>0$ and $M>0$ such that for $\epsilon<\epsilon_0$ and any $\epsilon$-homomorphism $\phi:\Gamma \to G$, there exists a map $\psi:\Gamma \to G$ with defect $def_{\Gamma,G}(\psi)<\frac{1}{2}def_{\Gamma,G}(\phi)$ and $dist_{\Gamma,G}(\psi)<M def_{\Gamma,G}(\phi)$.
	\end{lemma}
	\begin{proof}
		We shall prove this by contradiction, so suppose for every $\epsilon>0$ and every $M>0$, there exists an $\epsilon$-homomorphism $\phi:\Gamma \to G$ such that for any $\psi:\Gamma \to G$, either $def_{\Gamma,G}(\psi)\geq \frac{1}{2}def_{\Gamma,G}(\phi)$ or $dist_{\Gamma,G}(\psi)\geq M def_{\Gamma,G_n}(\phi_n)$.\\
		Consider a sequence $\{M_n\}_{n \in \N}$ with $\lim_{n \to \infty}M_n = \infty$ and a sequence $\{\epsilon_n\}_{n \in \N}$ with $\lim_{n \to \infty}\epsilon_n = 0$. For each $n \in \N$, let $\phi_n:\Gamma \to G_n$ (with $G_n \in \mathcal{G}$) be a map with $def_{\Gamma,G_n}(\phi_n) \leq \epsilon_n$ such that for any map $\psi:\Gamma \to G_n$, either $def_{\Gamma,G_n}(\psi_n)\geq \frac{1}{2}def_{\Gamma,G_n}(\phi_n)$ or $dist_{\Gamma,G_n}(\psi_n)\geq M_n def_{\Gamma,G_n}(\phi_n)$. Then the asymptotic homomorphism $\{\phi_n\}_{n \in \N}$ as constructed proves that that $\Gamma$ does not have the asymptotic defect diminishing property with respect to $\mathcal{G}$. 
	\end{proof}
	\begin{theorem}\label{thm1}
		Suppose $\mathcal{G}$ is such that every $G \in \mathcal{G}$ is a complete metric space (with respect to its metric $d_G$). Then group $\Gamma$ is uniformly $\mathcal{G}$-stable with a linear estimate iff $\Gamma$ has the asymptotic defect diminishing property with respect to $\mathcal{G}$
	\end{theorem}
	\begin{proof}
		Suppose $\Gamma$ is uniformly $\mathcal{G}$-stable with a linear estimate, then the implication is immediate. Conversely, suppose $\Gamma$ has the asymptotic defect diminishing property with respect to $\mathcal{G}$. From the previous lemma, this means that there exists $\epsilon_0>0$ and $M>0$ such that for any $\epsilon$-homomorphism $\phi:\Gamma \to G$ (with $\epsilon<\epsilon_0$), there exists a map $\psi:\Gamma \to G$ with defect $def_{\Gamma,G}(\psi)<\frac{1}{2}def_{\Gamma,G}(\phi)$ and $dist_{\Gamma,G}(\psi)<M def_{\Gamma,G}(\phi)$. Set $\phi^{(0)} \coloneq \phi$ and $\phi^{(1)} \coloneq \psi$. Applying \cref{dim} inductively on $\phi^{(i)}$ to get $\phi^{(i+1)}$, we obtain a sequence of maps $\{\phi^{(j)}:\Gamma \to G\}_{j \in \N}$ where for each $j \in \N$, $\phi^{(j)}$ has defect at most $\epsilon/2^{j}$ and is of distance at most $M\epsilon/2^{j}$ from $\phi^{(j-1)}$. This gives us a Cauchy sequence of maps from $\Gamma \to G$. Since $G$ is a complete, the sequence  $\{\phi^{(j)}$ has a limit $\phi^{\infty}:\Gamma \to G$ with defect $def(\phi^{\infty}) = 0$ (hence it is a homomorphism). As for its distance from $\phi$,
		$$dist_{\Gamma,G}(\phi,\phi^{\infty}) \leq \sum \limits_{j=0}^{\infty}\frac{M\epsilon}{2^j} = 2M\epsilon$$
		Hence $\Gamma$ is uniformly $\mathcal{G}$-stable with a linear estimate. 
	\end{proof}
	
	\subsection{Ultraproducts and Internal Maps}\label{ssec-ultraproducts}
	We can quantify the asymptotic rates succintly using ultraproducts. We shall now briefly review some concepts from the theory of ultraproducts and non-standard analysis that would be of relevance to our constructions (for more details, refer \cite{goldblatt} and \cite{nonstandard}).\\
	Let $\U$ be a non-principal ultrafilter on $\N$, which is fixed throughout. A subset $S \subseteq \N$ is said to be \emph{large} if $S \in \U$. 
	\begin{definition}
		The (algebraic) \textbf{ultraproduct} $\prod_{\U}X_n$ (or alternately, $\{X_n\}_{\U}$) of an indexed collection $\{X_n\}_{n \in \N}$ of sets is defined to be 
		$$\prod_{\U}X_n \coloneq \prod_{n \in \N}X_n / \sim$$
		where for $\{x_n\}_{n \in \N}, \{y_n\}_{n \in \N} \in \prod_{n \in \N}X_n$, $\{x_n\}_{n \in \N} \sim \{y_n\}_{n \in \N}$ if $\{n \space : \space x_n=y_n\} \in \U$. 
	\end{definition}
	In other words, we identify two sequences $\{x_n\}_{n \in \N}, \{y_n\}_{n \in \N} \in \prod_{n \in \N}X_n$ if they agree on a large set of indices. The image of a sequence $\{x_n\}_{n \in \N} \in \prod_{n \in \N}X_n$ under this equivalence relation shall be denoted $\{x_n\}_{\U}$. Conversely, given an element of $\prod_{\U}X_n$, we shall always regard it as $\{x_n\}_{\U}$ for some sequence $\{x_n\}_{n \in \N} \in \prod_{n \in \N}X_n$.\\
	If $X_n=X$ for every $n \in \N$, then $\prod_{\U}X$ is called the (algebraic) \emph{ultrapower} of $X$, denoted ${}^*X$. Note that $X$ can be embedded in ${}^*X$ via a diagonal embedding (for $x \in X$, $x \mapsto \{x\}_{\U} \in {}^*X$).\\
	Ultraproducts can be made to inherit algebraic structures of their building blocks. More precisely, let $\{X_n\}_{n \in \N}$, $\{Y_n\}_{n \in \N}$, and $\{Z_n\}_{n \in \N}$ be indexed families of sets with operations $*_n:X_n \times Y_n \to Z_n$ for every $n \in \N$. This naturally defines an operation $*:\prod_{\U}X_n \times \prod_{\U}Y_n \to \prod_{\U}Z_n$ by
	$$\{x_n\}_{\U}\space  * \space \{y_n\}_{\U} = \{x_n *_{n} y_n\}_{\U}$$
	We shall frequently encounter the following examples of ultraproducts:
	\begin{itemize}
		\item The ultrapower ${}^*G$ of a group $G$ is itself a group. This can be seen by noting that ${}^*G$ is the quotient of the direct product group $\prod_{n \in \N}G$ by the normal subgroup comprising elements $\{g_n\}_{n \in \N}$ with $\{n \space : \space g_n=1\} \in \U$. 
		\item The ultrapower ${}^*\R$ of $\R$ is a non-archimedean ordered field called the \emph{hyperreals}.
		\item Let $\{W_n\}_{n \in \N}$ be a family of Banach spaces. Then $\prod_{\U}W_n$ can be given the structure of a ${}^*\R$-vector space. In fact, it also comes equipped with a ${}^*\R$-valued norm. 
	\end{itemize}
	One of the standard tools of non-standard analysis is the transfer principle which relates the truth of statements concerning objects and their counterparts in the ultraproduct universe. Intuitively, our standard universe comprises all objects under normal consideration like $\R$, $\C$, etc. but maybe formally modeled as follows: define
	$$V_0(\R)=\R$$
	$$V_{n+1}(\R)=V_n(\R) \cup \mathcal{P}\left(V_n(\R)\right)$$
	where $\mathcal{P}\left(V_n(\R)\right)$ denotes the power set of $\left(V_n(\R)\right)$.
	Then
	$$V(\R)=\cup_{n \geq 0}V_n(\R)$$
	is called the superstructure over $\R$, and can be interpreted as comprising all the natural structures we study in mathematics. This shall comprise our standard universe $Univ$. \\
	We can construct a mapping ${}^*:V(\R) \to V({}^*\R)$ that takes an object in the superstructure of $\R$ (our standard universe $Univ$) to an object in the superstructure of ${}^*\R$ satisfying the following:
	\begin{itemize}
		\item (Extension Principle) The mapping $*$ maps $\R$ to ${}^*\R$.
		\item (Transfer Principle) For any first-order formula $\phi$ involving $k$ variables, and $A_1,\dots,A_k \in V(\R)$, the statement $\phi(A_1,\dots,A_k)$ is true in $V(\R)$ iff $\phi({}^*A_1,\dots,{}^*A_k)$ is true in $V({}^*\R)$.
		\item (Countable Saturation) Suppose $\{X_n\}_{n \in \N}$ is a collection of sets in $^*\left(V(\R)\right)$ such that the intersection of any finite subcollection is non-empty. Then $\cap X_n$ is non-empty.
	\end{itemize}
	It is a basic result of non-standard analysis that such a mapping $*$ exists, and can be constructed using a non-principal ultrafilter on $\N$. The image of this mapping shall be referred to as our non-standard universe${}^*Univ$, which is
	$${}^*Univ= \{ \{X_n\}_{\mathcal{U}} \lvert X_n \in Univ\}$$
	comprising ultraproducts of elements of $Univ$. Note that ${}^*\R, {}^*\C, {}^*\Gamma$ are all contained in ${}^*Univ$.\\
	Objects contained in $Univ$ are called standard, while objects contained in ${}^*Univ$ are called internal. In particular, a subset $S$ of an ultraproduct $\prod_{\U}X_n$ is an \emph{internal subset} if there exist subsets $S_n \subseteq X_n$ for every $n \in \N$ such that $S= \prod_{\U}S_n$. A function $f:\prod_{\U}X_n \to \prod_{\U}Y_n$ is said to be an \textbf{internal function} if there exists a sequence $\{f_n:X_n \to Y_n\}_{n \in \N}$ such that $f=\{f_n\}_{\U}$.\\
	Objects that are not contained in ${}^*Univ$ are called \emph{external}. Note that standard objects like $\R$ and $\C$ too are external. We can always consider objects that are neither in $Univ$ nor in ${}^*Univ$. Two important examples of non-standard external subsets are
	\begin{itemize}
		\item The set of \textbf{bounded} hyperreals, denoted ${}^*\R_b$, is the subset comprising elements $\{x_n\}_{\U} \in {}^*\R$ for which there exists $C \in \R_{\geq 0}$ such that $\lvert x_n \rvert \leq C$ for every $n \in \N$. 
		\item The set of \textbf{infinitesimal} hyperreals, denoted ${}^*\R_{inf}$, is the subset comprising elements $\{x_n\}_{\U} \in {}^*\R$ such that for \emph{every} real $\epsilon > 0$, there exists a large set $S \in \U$ such that $\lvert x_n \rvert < \epsilon$ for every $n \in S$.
		\item For $x,y \in {}^*\R$, denote by $x=O_{\U}(y)$ if $x/y \in {}^*\R_b$, and by $x=o_{\U}(y)$ if $x/y \in {}^*\R_{inf}$ (in particular,any bounded element $x \in {}^*\R_b$ is $x=O_{\U}(1)$ while any infinitesimal element $\epsilon \in {}^*\R_{inf}$ is $\epsilon = o_{\U}(1)$).
	\end{itemize}
	Note that the preimage of ${}^*\R_b$ under the map $\prod_{n \in \N}\R \to {}^*\R$ includes the subset of all bounded sequences, while the preimage of ${}^*\R_{inf}$ includes the subset of infinitesimal sequences (that is, sequences that converge to $0$).\\
	The subset ${}^*\R_b$ forms a valuation ring with ${}^*\R_{inf}$ being the unique maximal ideal, with quotient ${}^*\R_b/{}^*\R_{inf} \cong \R$. The quotient map $st:{}^*\R_b \to \R$ is known as the \emph{standard part} map or \emph{limit along the ultrafilter $\U$}.\\
	The previous construction can also be replicated for metric spaces with specified base points. Let $\{(X_n,d_n,p_n)\}_{n \in \N}$ be an indexed family of metric spaces (where the space $X_n$ comes with the metric $d_n$ and the base point $p_n$). The ultraproduct $\prod_{\U}X_n$ comes equipped with an ${}^*\R$-values metric $\prod_{\U}d_n$ and base point $\prod_{\U}p_n$. Consider the subset, denoted $\left(\prod_{\U}X_n\right)_b$, of $\prod_{\U}X_n$ comprising $\{x_n\}_{\U}$ such that $\{d_n(x_n,p_n)\}_{\U} \in {}^*\R_b$ (such elements are referred to as bounded or \emph{admissible}), and a subset, denoted $\left(\prod_{\U}X_n\right)_{inf}$, comprising $\{x_n\}_{\U}$ such that $\{d_n(x_n,p_n)\}_{\U} \in {}^*\R_{inf}$. Define an equivalence relation $\sim$ on $\left(\prod_{\U}X_n\right)_b$ by $\{x_n\}_{\U} \sim \{y_n\}_{\U}$ if $\{d_n(x_n,y_n)\}_{\U} \in {}^*\R_{inf}$. The set of equivalence classes $\left(\prod_{\U}X_n\right)_b/\sim$ is called the \emph{ultralimit} of $\{(X_n,d_n,p_n)\}_{n \in \N}$. We will return to this notion in the context of Banach spaces in \S\ref{sec-coho}.\\
	\begin{remark}
		For convenience, we shall henceforth denote by \say{for $n \in \U$} to mean \say{for $n \in S$ for some $S \in \U$}.
	\end{remark}
	Returning to our setting, consider a sequence $\{\phi_n:\Gamma \to G_n\}_{n \in \N}$ where $G_n \in \mathcal{G}$. This can be used to construct an internal map $\{\phi_n\}_{\U}: {}^*\Gamma \to \prod_{\U}G_n$, and allows us to redefine asymptotic homomorphisms and closeness to homomorphisms in the ultraproduct. 
	\begin{itemize}
		\item Given two internal maps $\phi^{(1)}:{}^*\Gamma \to \prod_{\U}G_n$ and $\phi^{(2)}:{}^*\Gamma \to \prod_{\U}G_n$, we denote by $dist(\phi^{(1)},\phi^{(2)}) \coloneq \{dist_{\Gamma,G_n}(\phi^{(1)}_n,\phi^{(2)}_n)\}_{\U}$. The maps $\phi^{(1)}$ and $\phi^{(2)}$ are said to be (internally) \textbf{asymptotically close} to each other if $dist(\phi^{(1)},\phi^{(2)}) \in {}^*\R_{inf}$. 
		\item An internal map $\psi: {}^*\Gamma \to \prod_{\U}G_n$ is said to be an \textbf{internal homomorphism} if there exists a sequence $\{\psi_n\}_{n \in \N}$ of homomorphisms such that $\psi=\{\psi_n\}_{\U}$.
		\item An internal map $\phi \coloneq \{\phi_n\}_{\U} : {}^*\Gamma \to \prod_{\U}G_n$ is said to be (internally) \textbf{asymptotically close to an internal homomorphism} if there exists an internal homomorphism $\psi = \{\psi_n:\Gamma \to G_n\}_{n \in \N}$ such that $dist(\phi,\psi) \in {}^*\R_{inf}$ (in other words, $\{D_{\Gamma,G_n}\}_{\U} \in {}^*\R_{inf}$).
		\item An internal map $\phi \coloneq \{\phi_n\}_{\U}: {}^*\Gamma \to \prod_{\U}G_n$ is called an \textbf{internal asymptotic homomorphism} if $def(\phi) \coloneq \{def(\phi_n)\}_{\U} \in {}^*\R_{inf}$. For $\epsilon \in {}^*\R_{inf}$, we shall call an internal asymptotic homomorphism $\phi$ an internal $\epsilon$-homomorphism if $def(\phi) = \epsilon$ (we similarly define internal $O_{\U}(\epsilon)$-homomorphisms and internal $o_{\U}(\epsilon)$-homomorphisms).
	\end{itemize}
	The following lemmas are variants of \cref{l1} and \cref{thm1} in the ultraproduct setting:
	\begin{lemma}\label{l5}
		The group $\Gamma$ is uniformly $\mathcal{G}$-stable iff every internal asymptotic homomorphism $\phi:{}^*\Gamma \to \prod_{\U}G_n$ for $G_n \in \mathcal{G}$ is  asymptotically close to an internal homomorphism.
	\end{lemma}
	\begin{proof}
		We shall prove both directions by contradiction. \\
		Let $\phi=\{\phi_n\}_{\U} :{}^*\Gamma \to \prod_{\U}G_n$ be an internal asymptotic homomorphism that is \emph{not} asymptotically close to any internal homomorphism. This means that $\{def(\phi_n)\}_{\U} \in {}^*\R_{inf}$, but there exists some real $c>0$ and $S \in \U$ such that $D_{\Gamma,G_n}(\phi_n) \geq c$ for $n \in S$. In particular, for this $c$ and any $\epsilon>0$, there exists a map $\phi_n:\Gamma \to G_n$ with defect $def(\phi_n) \leq \epsilon$ and $D_{\Gamma,G_n} \geq c$, thus proving that $\Gamma$ is \emph{not} uniformly $\mathcal{G}$-stable.\\ 
		Conversely, suppose $\Gamma$ is \emph{not} uniformly $\mathcal{G}$-stable. This means, by definition, that there exists $\delta>0$ such that for every $\epsilon>0$, there exists $G \in \mathcal{G}$ and $\phi:\Gamma \to G$ such that $def(\phi)\leq \epsilon$ and $dist_{\Gamma,G}(\phi) > \delta$. Let us fix this $\delta$, and consider a sequence $\{\epsilon_n\}_{n \in \N}$ with $\lim_{n \to \infty}\epsilon_n =0$. For each such $\epsilon_n$, there exists $G_n \in \mathcal{G}$ and $\phi_n:\Gamma \to G_n$ with $def(\phi_n)\leq \epsilon_n$ and $dist_{\Gamma,G_n}(\phi_n) > \delta$. Consider the internal asymptotic homomorphism $\{\phi_n\}_{\U}$. Clearly it is not asymptotically close to any internal homomorphism. 
	\end{proof}
	\begin{lemma}
		The group $\Gamma$ is $\mathcal{G}$-uniformly stable with a linear estimate iff for every internal asymptotic homomorphism $\phi:{}^*\Gamma \to \prod_{\U}G_n$, there exists an internal homomorphism $\psi:{}^*\Gamma \to \prod_{\U}G_n$ with $dist(\phi,\psi) = O_{\U}\left( def(\phi) \right)$.
	\end{lemma}
	\begin{proof}
		(The proof is similar to \cref{l1}) Suppose $\Gamma$ is not uniformly $\mathcal{G}$-stable with a linear estimate. Consider a sequence $\{M_n\}_{n \in \N}$ with $\lim_{n \to \infty}M_n = \infty$ and a sequence $\{\epsilon_n\}_{n \in \N}$ with $\lim_{n \to \infty}\epsilon_n = 0$. For each $n \in \N$, let $\phi_n:\Gamma \to G_n$ be the map with $def_{\Gamma,G_n}(\phi_n) \leq \epsilon_n$ but $D_{\Gamma,G_n}(\phi_n) > M_n\epsilon_n$. Then the ultraproduct $\phi = \{\phi_n\}_{\U}:{}^*\Gamma \to \prod_{\U}G_n$ is an internal asymptotic homomorphism with $def(\phi) = \epsilon \coloneq \{\epsilon_n\}_{\U}$ such that for any internal homomorphism $\psi:{}^*\Gamma \to \prod_{\U}G_n$, $dist(\phi,\psi)/def(\phi)$ is not in ${}^*\R_b$.\\
		Conversely, suppose $\Gamma$ is $\mathcal{G}$-uniformly stable with a linear estimate, and let $\phi:{}^*\Gamma \to \prod_{\U}G_n$ be an internal asymptotic homomorphism. There exists $\epsilon_0>0$, $M>0$ and a large subset $S_{\epsilon_0} \in \U$ such that for every $n \in S_{\epsilon_0}$, $\phi_n:\Gamma \to G_n$ has defect $def(\phi_n)\leq \epsilon_0$, and $D_{\Gamma,G_n}(\phi_n)\leq M def(\phi_n)$, allowing us to construct an internal homomorphism $\psi:{}^*\Gamma \to \prod_{\U}G_n$ that is asymptotically close to $\phi$.
	\end{proof}
	We now reformulate the notion of defect diminishing in this ultraproduct setting:
	\begin{definition}\label{def-defdim}
		The group $\Gamma$ is said to have the \textbf{defect diminishing} property with respect to the family $\mathcal{G}$ if for any internal asymptotic homomorphism $\phi:{}^*\Gamma \to \prod_{\U}G_n$, there exists an asymptotic homomorphism $\psi:{}^*\Gamma \to \prod_{\U}G_n$ such that  
		\begin{itemize}
			\item The defect $def(\psi) = o_{\U}\left(def(\phi)\right)$.
			\item The distance $dist(\phi,\psi) = O_{\U}\left(def(\phi)\right)$.
		\end{itemize}
	\end{definition}
	The following lemma reformulates \cref{thm1} in the ultraproduct setting, and the proof is on the same lines.
	\begin{theorem}
		The group $\Gamma$ is uniformly $\mathcal{G}$-stable with a linear estimate iff $\Gamma$ has the defect diminishing property with respect to $\mathcal{G}$.
	\end{theorem}
	
	\subsection{Internal Liftings and Defect Diminishing}\label{ssec-liftings}
	In this \S, we shall reinterpret an internal asymptotic homomorphism as a (true) homomorpism to a quotient group. This will allow us to describe defect diminishing as a homomorphism lifting problem. In \S\ref{ssec-abeliankernel} we will restrict to a family of metric groups of interest such that this homomorphism lifting problem has an abelian kernel, allowing us to build a cohomological theory capturing the obstruction to uniform stability.\\
	Let $\phi:{}^*\Gamma \to \prod_{\U}G_n$ be an internal asymptotic homomorphism. Consider the external subset $\left(\prod_{\U}G_n\right)_{inf}$ defined as
	$$\left(\prod_{\U}G_n\right)_{inf} \coloneq \Big\{ \{g_n\}_{\U} \space : \space \{d_n(g_n,1)\}_{\U} \in {}^*\R_{inf}\Big\}$$
	In other words, $\left(\prod_{\U}G_n\right)_{inf}$ comprises all elements that are infinitesimally close to the identity in $\prod_{\U}G_n$. It is easy to check that $\left(\prod_{\U}G_n\right)_{inf}$ is not just a subset but a normal subgroup of $\prod_{\U}G_n$. Since $\phi$ has defect $def(\phi) \in {}^*\R_{inf}$, this means that for every $x,y \in {}^*\Gamma$, $\phi(xy)^{-1}\phi(x)\phi(y) \in \left(\prod_{\U}G_n\right)_{inf}$, making $\tilde{\phi}:{}^*\Gamma \to \prod_{\U}G_n/\left(\prod_{\U}G_n\right)_{inf}$ a homomorphism. Thus,
	\begin{lemma}
		Given an internal asymptotic homomorphism $\phi:{}^*\Gamma \to \prod_{\U}G_n$, its composition, denoted $\tilde{\phi}$, with the canonical quotient homomorphism $\prod_{\U}G_n \to \prod_{\U}G_n/\left(\prod_{\U}G_n\right)_{inf}$ is a homomorphism from ${}^*\Gamma$ to the group $\prod_{\U}G_n/\left(\prod_{\U}G_n\right)_{inf}$. 
	\end{lemma}
	Note that a homomorphism from ${}^*\Gamma$ to $\prod_{\U}G_n/\left(\prod_{\U}G_n\right)_{inf}$ does not necessarily arise as the composition of an internal map ${}^*\Gamma \to \prod_{\U}G_n$ and the quotient homomorphism $\prod_{\U}G_n \to \prod_{\U}G_n/\left(\prod_{\U}G_n\right)_{inf}$. However, we will be specifically interested in homomorphisms that arise this way.\\
	\begin{definition}
		Let $F:{}^*\Gamma \to \prod_{\U}G_n/\left(\prod_{\U}G_n\right)_{inf}$ be a homomorphism. We say that $F$ has an \textbf{internal lift} $\hat{F}:{}^*\Gamma \to \prod_{\U}G_n$ if $\hat{F}$ is internal, and its composition with the canonical quotient homomorphism $\prod_{\U}G_n \to \prod_{\U}G_n/\left(\prod_{\U}G_n\right)_{inf}$ is $F$.\\
		We say that $F$ has an \textbf{internal lift homomorphism} if there exists an internal lift $\hat{F}$ of $F$ that is also a homomorphism from ${}^*\Gamma$ to $\prod_{\U}G_n$.
	\end{definition}
	Observe that for an internal asymptotic homomorphism $\phi:{}^*\Gamma \to \prod_{\U}G_n$, $\phi$ itself is an internal lift of the homomorphism $\tilde{\phi}$. Conversely,
	\begin{lemma}\label{l4}
		Suppose a homomorphism $F:{}^*\Gamma \to \prod_{\U}G_n/\left(\prod_{\U}G_n\right)_{inf}$ has an internal lift $\hat{F}:{}^*\Gamma \to \prod_{\U}G_n$. Then $\hat{F}$ is an internal asymptotic homomorphism. Furthermore, suppose $\hat{F}^{(1)}$ and $\hat{F}^{(2)}$ are two internal lifts of a homomorphism $F:{}^*\Gamma \to \prod_{\U}G_n/\left(\prod_{\U}G_n\right)_{inf}$, then $\hat{F}^{(1)}$ and $\hat{F}^{(2)}$ are asymptotically close to each other. 
	\end{lemma}
	\begin{proof}
		The map $\hat{F}:{}^*\Gamma \to \prod_{\U}G_n$ , being internal, is of the form $\hat{F}=\{\hat{F}_n\}_{\U}$ for $\hat{F}_n:\Gamma \to G_n$ for every $n \in \N$. Since $F$ is a homomorphism, for every $x=\{x_n\}_{\U},y=\{y_n\}_{\U} \in {}^*\Gamma$, $\{\hat{F}_n(x_ny_n)^{-1}\hat{F}_n(x_n)\hat{F}_n(y_n)\}_{\U} \in \left(\prod_{\U}G_n\right)_{inf}$ which means that $\{def(\hat{F}_n)\}_{\U} \in {}^*\R_{inf}$, making $\hat{F}$ an internal asymptotic homomorphism.\\
		Since both $\hat{F}^{(1)}$ and $\hat{F}^{(2)}$ are internal lifts of the homomorphism $F: {}^*\Gamma \to \prod_{\U}G_n/\left(\prod_{\U}G_n\right)_{inf}$, for every $x=\{x_n\}_{\U} \in {}^*\Gamma$, $\{d_n(\hat{F}^{(1)}_n(x_n),\hat{F}^{(2)}_n(x_n))\}_{\U} \in {}^*\R_{inf}$. 
	\end{proof}
	This motivates the following equivalent condition for uniform $\mathcal{G}$-stability in terms of internal lifts.
	\begin{lemma}\label{intasym}
		The group $\Gamma$ is uniformly $\mathcal{G}$-stable iff every homomorphism $\tilde{\phi}:{}^*\Gamma \to \prod_{\U}G_n/\left(\prod_{\U}G_n\right)_{inf}$ that has an internal lift also has an internal lift homomorphism.
	\end{lemma}
	\begin{proof}
		Let $\phi:{}^*\Gamma \to \prod_{\U}G_n$ be an internal asymptotic homomorphism, and $\tilde{\phi}:{}^*\Gamma \to \prod_{\U}G_n/\left(\prod_{\U}G_n\right)_{inf}$ be the homomorphism obtained by composing $\phi$ with the quotient map $\prod_{\U}G_n \to \prod_{\U}G_n/\left(\prod_{\U}G_n\right)_{inf}$. Let $\psi:{}^*\Gamma \to \prod_{\U}G_n$ be an internal lift homomorphism of $\tilde{\phi}$. Then by the previous \cref{l4}, $\phi$ is asymptotically close to the internal homomorphism $\psi$. By \cref{l5}, we conclude that $\Gamma$ is uniformly $\mathcal{G}$-stable. The converse follows by definition of uniform $\mathcal{G}$-stability. 
	\end{proof}
	\begin{remark}\label{rem1}
		The quotient group $\prod_{\U}G_n/\left(\prod_{\U}G_n\right)_{inf}$ is called the \emph{metric ultraproduct} of the sequence $\{G_n\}_{n \in \N}$ of groups. For (pointwise) stability of groups, it is shown in \cite{arz} that $\Gamma$ is (pointise) $\mathcal{G}$-stable if any homomorphism $\tilde{\phi}:\Gamma \to \prod_{\U}G_n/\left(\prod_{\U}G_n\right)_{inf}$ from $\Gamma$ to the metric ultraproduct, can be lifted to a homomorphism $\psi:\Gamma \to \prod_{\U}G_n$. In our setting, the uniformity requirement forces us to work with internal maps from the ultrapower ${}^*\Gamma$ as opposed to just maps from $\Gamma$. 
	\end{remark}
	We shall further refine the internal lifting property parametrizing it by the precise defect. Let $\phi = \{\phi_n\}_{\U}:{}^*\Gamma \to \prod_{\U}G_n$ be an internal asymptotic homomorphism with defect $def(\phi)=\epsilon \in {}^*\R_{inf}$. Consider the subset $B(\epsilon)$ (elements \emph{bounded} by $\epsilon$) of $\prod_{\U}G_n$ defined as follows:
	$$B(\epsilon)\coloneqq \{ \{g_n\}_{\U} \in \prod_{\U}G_n: \{d_n(g_n,1_n)\}_{\U} = O_{\U}(\epsilon)\}$$
	Note that $B(\epsilon)$ is an externally defined subset. Since the metric on $G_n$ is bi-invariant, the subset $B(\epsilon)$ is a normal subgroup of $\prod_{\U}G_n$. Let $q_{B(\epsilon)}:\prod_{\U}G_n \to \prod_{\U}G_n/B(\epsilon)$ be the canonical quotient homomorphism, and denote by $\tilde{\phi}$ the composition map $q_{B(\epsilon)}\cdot \phi:{}^*\Gamma \to \left( \prod_{\U}G_n\right)/B(\epsilon)$. The following lemma is a parametrized reformulation of \cref{l4}, and the proof is on the same lines, which we omit here:
	\begin{lemma}
		The map $\tilde{\phi}:{}^*\Gamma \to \left( \prod_{\U}G_n\right)/B(\epsilon)$ is a group homomorphism. 
	\end{lemma}
	Conversely, let $F:{}^*\Gamma \to  \prod_{\U}G_n/B(\epsilon)$. An internal map $\hat{F}:{}^*\Gamma \to  \prod_{\U}G_n$ is said to be an internal lift of $\tilde{\phi}$ if $q_{B(\epsilon)} \cdot \hat{F}=F$. We again have an analogue of \cref{l4} here too, whose proof is similar:
	\begin{lemma}
		Suppose a homomorphism $F:{}^*\Gamma \to  \prod_{\U}G_n/B(\epsilon)$ has an internal lift $\hat{F}:{}^*\Gamma \to  \prod_{\U}G_n$. Then $\hat{F}$ is an internal $O_{\U}(\epsilon)$-homomorphism. Futhermore, if $\hat{F}^{(1)}$ and $\hat{F}^{(2)}$ are two such internal lifts of $F$, then $\hat{F}^{(1)}$ and $\hat{F}^{(2)}$ are internally $O_{\U}(\epsilon)$-close to each other. 
	\end{lemma} 
	\begin{lemma}
		The group $\Gamma$ is uniformly $\mathcal{G}$-stable with a linear estimate iff for every $\epsilon \in {}^*\R_{inf}$, every homomorphism $\tilde{\phi}:{}^*\Gamma \to \prod_{\U}G_n/B(\epsilon)$ that has an internal lift also has an internal lift homomorphism.
	\end{lemma}
	\begin{proof}
		Suppose $\phi:{}^*\Gamma \to \prod_{\U}G_n$ is an internal asymptotic homomorphism with defect $\epsilon \in {}^*\R_{inf}$. Then $\tilde{\phi}::{}^*\Gamma \to \prod_{\U}G_n/B(\epsilon)$ is a homomorphism which has internal lift $\phi$. Thus, it also has an internal lift homomorphism $\psi:{}^*\Gamma \to \prod_{\U}G_n$ which is internally $O_{\U}(\epsilon)$-close to $\phi$. Hence, by \cref{intasym}, $\Gamma$ is uniformly $\mathcal{G}$-stable with a linear estimate. The converse is immediate. 
	\end{proof}
	Thus, for an infinitesimal $\epsilon \in {}^*\R_{inf}$, if a homomorphism $\tilde{\phi}:{}^*\Gamma \to \prod_{\U}G_n/B(\epsilon)$ that has an internal lift, can be internally lifted to an internal homomorphism $\psi:{}^*\Gamma \to \prod_{\U}G_n$, then $\Gamma$ is uniformly $\mathcal{G}$-stable with a linear estimate. We shall now try to obtain such a lift by a sequence of intermediate lifts.\\
	For $\epsilon \in {}^*\R_{inf}$, denote by $I(\epsilon)$ the subset of $\prod_{\U}G_n$ (elements infinitesimal with respect to $\epsilon$) defined as
	$$I(\epsilon) \coloneq \{ \{g_n\}_{\U} \in \prod_{\U}G_n: \{d_n(g_n,1_n)\}_{\U} = o_{\U}(\epsilon)\}$$
	Note that by the bi-invariance of the metric, $I(\epsilon) \subseteq B(\epsilon)$ is a normal subgroup of $\prod_{\U}G_n$. Let $q_{I(\epsilon)}:\prod_{\U}G_n \to \prod_{\U}G_n/I(\epsilon)$ be the canonical quotient homomorphism. The following lemma is similar to \cref{l4}, and the proof too is on the same lines:
	\begin{lemma}\label{l3}
		Suppose $\phi:{}^*\Gamma \to \prod_{\U}G_n$ is an internal $o_{\U}(\epsilon)$-homomorphism, then $q_{I(\epsilon)}\cdot \phi: {}^*\Gamma \to \prod_{\U}G_n/I(\epsilon)$ is a homomorphism. Conversely, suppose a homomorphism $F:{}^*\Gamma \to \prod_{\U}G_n/I(\epsilon)$ has an internal lift $\hat{F}:{}^*\Gamma \to \prod_{\U}G_n$, then $\hat{F}$ is an internal $o_{\U}(\epsilon)$-homomorphism, and any two internal lifts of $F$ are internally  $o_{\U}(\epsilon)$-close to one another. 
	\end{lemma}
	We can now reformulate the defect diminishing property in terms of internal lifts.
	\begin{lemma}
		The group $\Gamma$ has the defect diminishing property with respect to $\mathcal{G}$ iff for every $\epsilon \in {}^*\R_{inf}$ and every homomorphism $F:{}^*\Gamma \to \prod_{\U}G_n/B(\epsilon)$ that has an internal lift, $F$ has an internal lift $\hat{F}:{}^*\Gamma \to \prod_{\U}G_n$ such that $q_{I(\epsilon)}\cdot \hat{F}: {}^*\Gamma \to \prod_{\U}G_n/I(\epsilon)$ is a homomorphism. 
	\end{lemma}
	\begin{proof}
		Suppose $\Gamma$ has the defect diminishing property, then it is immediate from \cref{def-defdim} that for every $\epsilon \in {}^*\R_{inf}$ and every homomorphism $F:{}^*\Gamma \to \prod_{\U}G_n/B(\epsilon)$ that has an internal lift, $F$ has an internal lift $\hat{F}:{}^*\Gamma \to \prod_{\U}G_n$ such that $q_{I(\epsilon)}\cdot F: {}^*\Gamma \to \prod_{\U}G_n/I(\epsilon)$ is a homomorphism.\\
		Conversely, consider an internal asymptotic homomorphism $\phi:{}^*\Gamma \to \prod_{\U}G_n$ with defect $def(\phi)=\epsilon$, and the induced homomorphism $\tilde{\phi}:{}^*\Gamma \to \prod_{\U}G_n/B(\epsilon)$. By the hypothesis of the lemma (and \cref{l3}), there exists an internal $o_{\U}(\epsilon)$-homomorphism $\psi:{}^*\Gamma \to \prod_{\U}G_n$ which is $O_{\U}(\epsilon)$-close to $\phi$. This shows that $\Gamma$ has the defect diminishing property with respect to $\mathcal{G}$. 
	\end{proof}
	We now recap the results obtained so far:
	\begin{theorem}
		Let $\Gamma$ be a discrete group, and $\mathcal{G}$ be a family of metric groups such that for every group $G \in \mathcal{G}$, its metric $d_{G}$ is complete. Then the following are equivalent:
		\begin{enumerate}
			\item The group $\Gamma$ is uniformly $\mathcal{G}$-stable with a linear estimate.
			\item The group $\Gamma$ has the defect diminishing property with respect to $\mathcal{G}$.
			\item For every $\epsilon \in {}^*\R_{inf}$ and every homomorphism $F:{}^*\Gamma \to \prod_{\U}G_n/B(\epsilon)$ that has an internal lift, $F$ has an internal lift $\hat{F}:{}^*\Gamma \to \prod_{\U}G_n$ such that $q_{I(\epsilon)}\cdot \tilde{F}: {}^*\Gamma \to \prod_{\U}G_n/I(\epsilon)$ is a homomorphism.
		\end{enumerate}
	\end{theorem}
	
	\section{A Cohomological Interpretation of Stability}\label{sec-asymgamma}
	We concluded the previous section by noting that in order to prove that $\Gamma$ is uniformly $\mathcal{G}$-stable with a linear estimate, it is sufficient to show that it has the defect diminishing property with respect to $\mathcal{G}$. Furthermore, we interpreted this property in terms of internal lifts of homomorphisms to quotient groups.\\
	Recall that a matrix norm $\|\cdot\|$ on $M_n(\C)$ is said to be submultiplicative if for every $A,B \in M_n(\C)$, $\|AB\|\leq \|A\| \cdot \|B\|$. From now on, we shall work exclusively with  the family of unitary groups, each equipped with a unitarily bi-invariant \emph{submultiplicative} matrix norm, and shall denote this family by $\mathfrak{U}$.
	$$\mathfrak{U} \coloneq \Big\{ \left( U(n), \|\cdot\| \right) \space : \space n \in \N \text{ and }\|\cdot\| \text{ is submultiplicative} \Big\}$$
	In particular, these include the $p$-Schatten norms given by
	\begin{equation*} 
	\| A \|_p = \begin{cases} 
	(Tr |A|^{p})^{1/p} &1 \le p < \infty\\
	\sup_{\| \nu \| = 1} \|A\nu\|  &p = \infty\end{cases}
	\end{equation*}
	Note that for $p=\infty$, this is the operator norm (as studied in \cite{Kaz} and  \cite{BOT}), while for $p=2$, this is the \emph{Frobenius} norm or the (unnormalized) \emph{Hilbert-Schmidt} norm (as studied in \cite{DCGLT} and \cite{dogon}).\\
	
	Before we proceed further, we state and sketch the proof of the following useful transference lemma for uniform $\mathfrak{U}$-stability with a linear estimate, which we shall use in \S\ref{sec-mainproof}. The proof is on the lines of a relative version of \cite[Theorem 3.2]{BOT} (which reproves Kazhdan's result on the Ulam stability of amenable groups), which we can adapt here in the simpler setting of finite index.
	\begin{lemma}\label{finite-index-induction}
		Let $\Lambda \leq \Gamma$ be a subgroup of finite index. Then $\Gamma$ is uniformly $\mathfrak{U}$-stable with a linear estimate iff $\Lambda$ is uniformly $\mathfrak{U}$-stable with a linear estimate. 
	\end{lemma}
	\begin{proof}[Sketch]
		Suppose $\Gamma$ is uniformly $\mathfrak{U}$-stable with a linear estimate. Then the proof of \cite[Corollary 2.7]{BOT} (further explained in \cite[Lemma II.22]{gamm}) implies that $\Lambda$ too is uniformly $\mathfrak{U}$-stable with a linear estimate.\\
		Conversely, suppose $\Lambda$ is uniformly $\mathfrak{U}$-stable with a linear estimate, and let $\phi:\Gamma \to U(n)$ be an $\epsilon$-homomorphism. Since the finite index subgroup $\Lambda$ is uniformly $\mathfrak{U}$-stable with a linear estimate, we can assume that the restriction of $\phi$ to $\Lambda$ is a homomorphism, and furthermore, $\phi(g\delta)=\phi(g)\phi(\delta)$ for every $g \in \Gamma$, $\delta \in \Lambda$. Now define $\phi':\Gamma \to M(n)$ as 
		$$\phi'(g) \coloneq \frac{1}{|\Gamma:\Lambda|}\sum_{x \in \Gamma/\Lambda} \phi(gx)\phi(x)^*$$
		Note that $\phi'(g)$ is just the average over coset representatives in $\Gamma/\Lambda$, and $M(n)$ is the space of $n \times n$ matrices. Just as in the proof of \cite[Theorem 3.2]{BOT}, this $\phi'$ can be normalized to obtain $\phi_1:\Gamma \to U(n)$ such that $\phi_1$ has defect $C\epsilon^2$ (for a universal constant $C$ not depending on $n$), and we repeat the process to obtain a (true) homomorphism as the limit. 
	\end{proof}
	\begin{remark}
		In fact, it is further shown in \cite[Proposition 1.5]{francesco} that for a subgroup $\Lambda \leq \Gamma$ that is \emph{co-amenable} in $\Gamma$, if $\Lambda$ is uniformly $\mathfrak{U}$-stable with a linear estimate, then $\Gamma$ too is uniformly $\mathfrak{U}$-stable with a linear estimate. This generalizes one direction of \cref{finite-index-induction}, though the converse is not true in this level of generality.
	\end{remark}
	\begin{remark}\label{remark-comm}
		In particular, if $\Gamma_1$ and $\Gamma_2$ are commensurable, then $\Gamma_1$ is uniformly $\mathfrak{U}$-stable with a linear estimate iff $\Gamma_2$ is uniformly $\mathfrak{U}$-stable with a linear estimate. 
	\end{remark}
	Given a sequence of unitary groups $\{U(k_n)\}_{n \in \N}$, we denote its ultraproduct by $\prod_{\U}U(k_n)$,and given an element $u=\{u_n\}_{\U} \in \prod_{\U}U(k_n)$ (where for each $n \in \N$, $u_n \in U(k_n)$), we denote its distance from the identity $1\in \prod_{\U}U(k_n)$ by $\|u-1\|\coloneq \{\|u_n-I\|\}_{\U} \in {}^*\R_b$, for notational convenience.\\
	Note that for $\epsilon \in {}^*\R_{inf}$ and an ultraproduct $\prod_{\U}U(k_n)$,  the subsets $B(\epsilon) \subseteq \prod_{\U}G_n$ and $I(\epsilon) \subseteq B(\epsilon)$ can now be written as
	$$B(\epsilon) = \Big\{u \in \prod_{\U}U(k_n):\|u-I\| = O_{\U}(\epsilon)\Big\}$$
	$$I(\epsilon) = \Big\{u \in \prod_{\U}U(k_n): \|u-I\| = o_{\U}(\epsilon)\Big\}$$
	Furthermore, the submultiplicativity of the norms implies that:
	\begin{lemma}\label{abelian}
		The group $B(\epsilon)/I(\epsilon)$ is abelian.
	\end{lemma}
	\begin{proof}
		Let $a,b \in B(\epsilon)$, and consider the commutator $aba^{*}b^{*}$. We shall prove that $aba^{*}b^{*} \in I(\epsilon)$. Observe that $\|aba^*b^*-I\|=\|ab-ba\|$ since the norm is unitarily invariant. 
		$$\|ab-ba\| = \|(a-I)(b-I)-(b-I)(a-I)\|\leq 2\|a-I\|\|b-I\|$$
		This is because of the submultiplicativity of the norm. Since $a,b \in B(\epsilon)$, we conclude that $\|aba^*b^*-I\|=O_{\U}(\epsilon^2)=o_{\U}(\epsilon)$. 
	\end{proof}
	The fact that $B(\epsilon)/I(\epsilon)$ is an abelian group will allow us to rephrase the lifting property discussed in \S\ref{ssec-liftings} in terms of the vanishing of a cohomology which we shall develop in detail in \S\ref{ssec-logarithm} and \S\ref{sec-coho}.\\
	In \S\ref{ssec-abeliankernel}, we explicitly work out the \say{cocycles} corresponding to possible lifts of a homomorphism, which are then transferred to the linearized setting in \S\ref{ssec-logarithm} where a cohomological theory begins to reveal itself. Finally, in \S\ref{ssec-amenable}, we demonstrate the notions discussed in the case of discrete abelian groups, showing that the vanishing of our second cohomology implies uniform stability. 
	
	\subsection{Lifting with an Abelian Kernel}\label{ssec-abeliankernel}
	Let $\phi:{}^*\Gamma \to \prod_{\U}U(k_n)$ be an internal $\epsilon$-homomorphism that induces a homomorphism $\tilde{\phi}:{}^*\Gamma \to \prod_{\U}U(k_n)/B(\epsilon)$. Let $\psi:{}^*\Gamma \to \prod_{\U}U(k_n)$ be an internal lift of $\tilde{\phi}$.\\ 
	To an internal lift $\psi:{}^*\Gamma \to \prod_{\U}U(k_n)$ of $\tilde{\phi}$, we associate an internal map
	$$\rho_{\psi}:{}^*\Gamma \times \prod_{\U}U(k_n) \to \prod_{\U}U(k_n)$$
	\begin{equation}\label{rhopsi}
	\rho_{\psi}(g)(u) \coloneq \psi(g)\cdot u \cdot \psi(g)^{-1}
	\end{equation}
	For every $g \in {}^*\Gamma$ and $u \in B(\epsilon)$, $\rho_{\psi}(g)(u) \in B(\epsilon)$, while for $u\in I(\epsilon)$,  $\rho_{\psi}(g)(u) \in I(\epsilon)$.
	\begin{lemma}
		The internal map $\rho_{\psi}:{}^*\Gamma \times \prod_{\U}U(k_n) \to \prod_{\U}U(k_n)$ induces an action, denoted $\tilde{\rho}_{\psi}$, of ${}^*\Gamma$ on the abelian group $B(\epsilon)/I(\epsilon)$.
	\end{lemma}
	\begin{proof}
		For $g_1,g_2 \in {}^*\Gamma$ and $u \in B(\epsilon)$, we want to show that $\rho_{\psi}(g_1)(\rho_{\psi}(g_2)(u))-\rho_{\psi}(g_1g_2)(u) \in I(\epsilon)$. Note that for every $g_1,g_2 \in {}^*\Gamma$, $\psi(g_1g_2)^{-1}\psi(g_1)\psi(g_2) \in B(\epsilon)$, so
		$$\|  \psi(g_1)\psi(g_2)u\psi(g_2)^{-1}\psi(g_1)^{-1} - \psi(g_1g_2)u\psi(g_1g_2)^{-1} \| = \| \psi(g_1g_2)^{-1}\psi(g_1)\psi(g_2)u - u\psi(g_1g_2)^{-1}\psi(g_1)\psi(g_2)\|$$
		Since the elements $\psi(g_1g_2)^{-1}\psi(g_1)\psi(g_2)$ and $u$ are in $B(\epsilon)$, their commutator is in $I(\epsilon)$ (as in proof of \cref{abelian}). 
	\end{proof}
	We shall call $\tilde{\rho}_{\psi}$ the \emph{action induced from $\psi$}. Observe that we have defined $\rho_{\psi}$ using a given internal lift $\psi$. However, this induced action of ${}^*\Gamma$ on $B(\epsilon)/I(\epsilon)$ is independent of the choice of internal lift of $\tilde{\phi}$.
	\begin{lemma}
		For two internal lifts $\psi_1$ and $\psi_2$ of $\tilde{\psi}$, $\tilde{\rho}_{\psi_1}=\tilde{\rho}_{\psi_2}$.	
	\end{lemma}
	\begin{proof}
		Note that for $g \in {}^*\Gamma$, $\psi_2(g)^{-1}\psi_1(g) \in B(\epsilon)$ since they are both internal lifts of $\tilde{\phi}$. Hence for $u \in B(\epsilon)$, again as in the proof of \cref{abelian}, $\psi_2(g)^{-1}\psi_1(g) u - u \psi_2(g)^{-1}\psi_1(g) \in I(\epsilon)$, which implies that $\psi_1(g)u\psi_1(g)^{-1} - \psi_2(g)u\psi_2(g)^{-1} \in I(\epsilon)$. 
	\end{proof}
	So while the internal map $\rho_{\psi}:{}^*\Gamma \times \prod_{\U}U(k_n) \to \prod_{\U}U(k_n)$ depends on $\psi$, the induced action $\tilde{\rho}_{\psi}$ is independent of the choice of internal lift of $\tilde{\phi}$. Hence we can denote this action by $\tilde{\rho}_{\phi}$.\\
	Define the internal map 
	$$\alpha_{\psi}:{}^*\Gamma \times {}^*\Gamma \to \prod_{\U}U(k_n)$$
	$$\alpha_{\psi}(g_1,g_2) \coloneq \psi(g_1)\psi(g_2)\psi(g_1g_2)^{-1}$$
	Note that by construction $\alpha_{\psi}$ takes values in $B(\epsilon)$ (and our goal is to find some lift $\psi$ for which $\alpha_{\psi}$ takes values only in $I(\epsilon)$). Observe that 
	\begin{lemma}\label{cocycle1}
		For $g_1,g_2,g_3 \in {}^*\Gamma$, 
		$$\alpha_{\psi}(g_1,g_2) \cdot \alpha_{\psi}(g_1g_2,g_3) \cdot \alpha_{\psi}(g_1,g_2g_3)^{-1} = \psi(g_1) \cdot \alpha_{\psi}(g_2,g_3) \psi(g_1)^{-1}$$
	\end{lemma}
	Let $q_{B/I}:B(\epsilon) \to B(\epsilon)/I(\epsilon)$ be the canonical quotient homomorphism. Since $\alpha_{\psi}$ takes values in $B(\epsilon)$, let $\tilde{\alpha}_{\psi} \coloneq q_{B/I} \cdot \alpha_{\psi}: {}^*\Gamma \times {}^*\Gamma \to B(\epsilon)/I(\epsilon)$. Since $B(\epsilon)/I(\epsilon)$ is abelian, \cref{cocycle1} implies the following corollary:
	\begin{corollary}\label{cocycle2}
		For $g_1,g_2,g_3 \in {}^*\Gamma$, 
		$$\tilde{\rho}_{\phi}(g_1) \cdot \tilde{\alpha}_{\psi}(g_2,g_3)-\tilde{\alpha}_{\psi}(g_1g_2,g_3)+\tilde{\alpha}_{\psi}(g_1,g_2g_3)-\tilde{\alpha}_{\psi}(g_1,g_2) = 0$$
	\end{corollary}
	While we noted that the action $\tilde{\rho}_{\psi}$ of ${}^*\Gamma$ on $B(\epsilon)/I(\epsilon)$ does not depend on the choice of lift $\psi$, the map $ \tilde{\alpha}_{\psi}:{}^*\Gamma \times {}^*\Gamma \to B(\epsilon)/I(\epsilon)$ does depend on the choice of lift $\psi$. Our hope is to prove the existence of some choice of lift $\psi$ such that $\tilde{\alpha}_{\psi}$ is trivial. \\
	Consider $\tilde{\alpha}_{\psi_1}$ and $\tilde{\alpha}_{\psi_2}$ for two different internal lifts $\psi_1$ and $\psi_2$ of $\tilde{\phi}$. Define an internal map
	$$\beta_{\psi_1,\psi_2}:{}^*\Gamma \to \prod_{\U}G_n$$
	$$\beta_{\psi_1,\psi_2}(g) \coloneq \psi_2(g)\psi_1(g)^{-1}$$
	Since $\beta_{\psi_1,\psi_2}$ takes values in $B(\epsilon)$, denote by $\tilde{\beta}_{\psi_1,\psi_2}:{}^*\Gamma \to B(\epsilon)/I(\epsilon)$ its composition with the quotient map $B(\epsilon) \to B(\epsilon)/I(\epsilon)$. 
	Then for $g_1,g_2 \in {}^*\Gamma$, a careful computation shows that
	\begin{equation}\label{cocycle3}
	\tilde{\alpha}_{\psi_2}-\tilde{\alpha}_{\psi_1} = \tilde{\beta}_{\psi_1,\psi_2}(g_1)+\tilde{\rho}_{\phi}(g_1) \cdot \tilde{\beta}_{\psi_1,\psi_2}(g_2)-\tilde{\beta}_{\psi_1,\psi_2}(g_1g_2)
	\end{equation}
	Suppose there exists an internal map $\beta:{}^*\Gamma \to \prod_{\U}G_n$ that takes values in $B(\epsilon)$ such that for the lift $\psi$, the following equation holds for all $g_1,g_2 \in {}^*\Gamma$:
	$$\tilde{\alpha}_{\psi}(g_1,g_2)= \tilde{\beta}(g_1)+\tilde{\rho}_{\phi}(g_1) \cdot \tilde{\beta}(g_2)-\tilde{\beta}(g_1g_2)$$
	Then the internal map 
	$$\psi^{\sim}:{}^*\Gamma \to \prod_{\U}U(k_n)$$
	$$\psi^{\sim}(g) \coloneq \psi(g)\beta(g)^{-1}$$
	is also an internal lift of $\tilde{\phi}$ such that for every $g_1,g_2 \in {}^*\Gamma$,
	$$\tilde{\alpha}_{\psi^{\sim}}(g_1,g_2)=0$$
	In particular, this means that $\psi^{\sim}$ is the internal lift that we want.\\
	The above discussion hints at a cohomological theory that captures the obstruction to such lifts. The idea is as follows: any candidate internal lift $\psi$ of $\tilde{\phi}$, with defect $O_{\U}(\epsilon)$, gives us a type of $2$-cocycle of ${}^*\Gamma$ with coefficients in $B(\epsilon)/I(\epsilon)$, and if that cocycle happens to be a $1$-couboundary, then the lift $\psi$ can be corrected to obtain another lift that has defect $o_{\U}(\epsilon)$ and is still $O_{\U}(\epsilon)$-close to $\psi$ (and $\phi$), thus implying the defect diminishing property that we want. \\
	\begin{remark}\label{rem2}
		In \cite{DCGLT}, it is shown (using the idea mentioned in \cref{rem1} and defect diminishing) that $\Gamma$ is (pointwise) stable (with respect to unitary matrices equipped with the Frobenius norm) if $H^2\left(\Gamma,B(\epsilon)/I(\epsilon)\right)$ vanishes. From \cref{cocycle1}, it might be tempting to simply consider $\tilde{\alpha}_{\psi}$ as a bounded $2$-cocycle for the group ${}^*\Gamma$ with coefficients in the abelian group $B(\epsilon)/I(\epsilon)$, and interpret \cref{cocycle3} (and the ensuing discussion) as insisting that $\tilde{\alpha}_{\psi}$ is the coboundary of a bounded $1$-cochain of ${}^*\Gamma$ in $B(\epsilon)/I(\epsilon)$. But we cannot simply work with $\Hb^2\left({}^*\Gamma,B(\epsilon)/I(\epsilon)\right)$, since we need to ensure that the bounded $2$-cocycle $\tilde{\alpha}_{\psi}$ (which was induced from an \emph{internal} map $\alpha_{\psi}$) is the coboundary of a bounded $1$-cochain $\tilde{\beta}:{}^*\Gamma \to B(\epsilon)/I(\epsilon)$ that is itself also induced from an \emph{internal} map $\beta:{}^*\Gamma \to \prod_{\U}U(k_n)$. This insistence on our maps being induced from internal maps is essential in our setting of uniform stability, and leads to the definition of an internal and asymptotic bounded cohomology machinery that we construct in \S\ref{sec-coho}.
	\end{remark}
	
	\subsection{Linearization and the Lie Algebra}\label{ssec-logarithm}
	In the previous section we observed that the defect diminishing property could be interpreted as a cohomological problem based on an action of ${}^*\Gamma$ on the abelian group $B(\epsilon)/I_{\epsilon}$. However, as pointed out in \cref{rem2} the subtlety here involves the requirement that we deal only with maps that are induced from some \emph{internal} mapping to $\prod_{\U}U(k_n)$. At that level we do not have the abelianness that would allow us to properly formulate a cohomology theory. In this \S, we transfer to the Lie algebra allowing us to work with spaces of maps from the group to Banach spaces.\\
	For a matrix $A \in M_n(\C)$, consider the matrix logarithm given by
	$$\log A \coloneq \sum \limits_{j=1}^{\infty} (-1)^{j-1}\frac{(A-I)^j}{j}$$
	The series above converges if $\|A-I\|<1$ (for some submultiplicative matrix norm $\|\cdot \|$). By subadditivity and submultiplicativity of the norm $\|\cdot\|$, if $\epsilon \leq 1/2$ and $\|u-I\|\leq \epsilon$,
	$$\|\log{u}\| \leq \sum \limits_{j=1}^{\infty} \|u-I\|^j \leq 2\epsilon$$
	\begin{lemma}\label{logbound}
		For every $\epsilon<1/2$, $n \in \N$, $u \in U(n)$ and every submultiplicative norm $\|\cdot \|$ on $M_n(\C)$, if $\|u-I\| \leq \epsilon$, then $\|\log{u}\| \leq 2 \epsilon$. 
	\end{lemma}
	It is a classical result that for a unitary matrix $u \in U(n)$, its logarithm $\log{u}$ (whenever it is defined) is an \emph{anti-Hermitian} matrix. Denote by $\mathfrak{u}(n)$ the (real) vector space of anti-Hermitian matrices in $M_n(\C)$.\\
	In the other direction, we have the matrix exponential map defined as 
	$$\exp(A) =  \sum \limits_{j=0}^{\infty} \frac{A^j}{j!}$$
	which is well-defined for every $A \in M_n(\C)$. For an anti-hermitian matrix $W \in \mathfrak{u}(n)$ of the form $W=U \circ diag(i\theta_j)_{j=1}^{n} \circ U^{*}$ for $U \in U(n)$, its exponential $exp(W) = U \circ diag(e^{i\theta_j})_{j=1}^n \circ U^{*}$. The following trivial bound is sufficient for our purposes:
	\begin{lemma}
		For every $\epsilon>0$, $n \in \N$ and a submultiplicative norm $\|\cdot\|$ on $M_n(\C)$, for a matrix $W \in \mathfrak{u}(n)$ with $\|W\|<\epsilon$, $\|exp(W)-I\| \leq e^{\epsilon}-1$. 
	\end{lemma}
	Note that since $\mathfrak{u}(n)$ is finite-dimensional, it is complete for any norm, making it a finite-dimensional real Banach space. It comes with an isometric (adjoint) action of $U(n)$ given as follows: for $v \in \mathfrak{u}(n)$ and $U \in U(n)$, $Ad(U)(v) \coloneq UvU^{*} \in \mathfrak{u}(n)$.\\
	Consider the family of $\R$-vector spaces of anti-hermitian matrices $\mathfrak{u}(n)$ each equipped with a submultiplicative norm. 
	$$\Big\{\left( \mathfrak{u}(n), \|\cdot\|\right) \space :n \in \N \text{ and }\|\cdot\| \text{ is submultiplicative}\Big\}$$
	For an infinitesimal $\epsilon \in {}^*\R_{inf}$, define the internal map $$_{\epsilon}\log:\prod_{\U}U(k_n) \to \prod_{\U}\mathfrak{u}(k_n)$$ 
	$$_{\epsilon}\log{u} \coloneq \frac{1}{\epsilon} \{\log{u_{k_n}}\}_{\U}$$
	and similarly, the internal map $\exp:\prod_{\U}\mathfrak{u}(k_n) \to \prod_{\U}U(k_n)$ given by
	$$_{\epsilon}\exp{u} \coloneq \{\exp{\epsilon_{k_n} u_{k_n}}\}_{\U}$$
	
	Let us denote the ultraproduct $\prod_{\U}\mathfrak{u}(k_n)$ by $\W$ from now on. Note that $\W$ comes with a ${}^*\R$-valued norm, which we shall denote simply as $\|\cdot\|$, obtained as the ultraproduct of the respective norms of each $\mathfrak{u}(k_n)$. The bounded elements of $\W$ shall be denoted $\W_b$ while the infinitesimal elements of $\W$ are denoted by $\W_{inf}$. That is,
	\begin{equation}\label{W_b}
	\W_b \coloneq \Big\{w \in \W \space : \|w\| \in {}^*\R_b \Big\}
	\end{equation}
	\begin{equation}\label{W_inf}
	\W_{inf} \coloneq \Big\{ w \in \W \space : \|w\| \in {}^*\R_{inf} \Big\}
	\end{equation}
	The motivation behind scaling the definitions of $\log$ and $\exp$ by $1/\epsilon$ and $\epsilon$ respectively is as follows:
	\begin{proposition}
		The internal map $_{\epsilon}\log: \prod_{\U}U(k_n) \to \W$  when restricted to $B(\epsilon)$, takes values in $\W_b$, and elements in $I(\epsilon)$ are taken to $\mathcal{W}_{inf}$. This induces an isomorphism of the abelian groups $\mathcal{W}_b/\mathcal{W}_{inf}$ and $B(\epsilon)/I(\epsilon)$.
	\end{proposition}
	\begin{proof}
		It follows from \cref{logbound} that for $u \in B(\epsilon)$, $_{\epsilon}\log{u} \in \W_b$, and for $u \in I(\epsilon)$, $_{\epsilon}\log{u} \in \W_{inf}$. The map is surjective on $\W_b$ since the map $_{\epsilon}\log{(_{\epsilon}\exp{v})}=v$ for $v \in \W_b$ (and similarly for $\W_{inf}$ as well).\\ 
		From properties of the logarithm map, it follows that for $u_1,u_2 \in B(\epsilon)$,$_{\epsilon}\log{u_1u_2}-(_{\epsilon}\log{u_1}+_{\epsilon}\log{u_2}) \in \W_{inf}$. Hence $_{\epsilon}\log$ induces a surjective group homomorphism from $B(\epsilon)$ to $\W_b/\W_{inf}$ with kernel $I(\epsilon)$. 
	\end{proof}
	The ultralimit $\W_b/\W_{inf}$ shall be denoted $\tilde{\W}$. The above lemma tells us that $\tilde{\W} \cong B(\epsilon)/I(\epsilon)$. In fact, $\tilde{\W} \cong B(\epsilon)/I(\epsilon)$ is not just an abelian group but has the structure of a real Banach space. It is an example of a construction known as a \emph{Banach space ultralimit}. We shall not prove this result here, but refer to \cite{banachultra} and \cite{nonstandard} for more details:
	\begin{proposition}[\cite{banachultra}]
		The space $\tilde{W}=B(\epsilon)/I(\epsilon)$ is a real Banach space.
	\end{proposition}
	Recall that we had defined (\cref{rhopsi}) an internal map $\rho_{\psi}:{}^*\Gamma \times \prod_{\U}U(k_n) \to \prod_{\U}U(k_n)$ defined as $\rho_{\psi}(g)v=\psi(g)v \psi(g)^{-1}$ which had induced an action $\tilde{\rho}_{\phi}$ of ${}^*\Gamma$ on $B(\epsilon)/I(\epsilon)$. We can similarly define an internal map
	$$\pi_{\psi}:{}^*\Gamma \times \W \to \W$$
	through the internal adjoint action of $\prod_{\U}U(k_n)$ on $\W$ (that is, conjugation),
	\begin{equation}\label{rho}
	\pi_{\psi}v \coloneqq \psi(g)v \psi(g)^{-1}
	\end{equation}
	Again, by the submultiplicativity of the norms, for $v \in \W_b$ and $g_1,g_2 \in {}^*\Gamma$, 
	$$\pi_{\psi}(g_1g_2)v-\pi_{\psi}(g_1)\pi_{\psi}(g_2)v \in \W_{inf} $$
	Thus, the internal map $\pi_{\psi}$ as defined above induces an action of ${}^*\Gamma$ on $\W_b/\W_{inf}$. Unless there is ambiguity, we shall denote the induced action of $g \in {}^*\Gamma$ on $\tilde{v} \in \W_b/\W_{inf}$ through $\pi_{\psi}$ by $g\cdot \tilde{v}$.
	\begin{lemma}
		The internal map $_{\epsilon}\log:\prod_{\U}U(k_n) \to \W$ induces a ${}^*\Gamma$-equivariant (additive) group isomorphism between $B(\epsilon)/I(\epsilon)$ (with the action induced from $\rho_{\psi}$) and $\tilde{\W}$ (with the action induced from $\pi_{\psi}$). 
	\end{lemma}
	\begin{proof}
		We already saw that $_{\epsilon}\log$ induces a group isomorphism between $B(\epsilon)/I(\epsilon)$ and $\W_b/\W_{inf}$. Let $g \in {}^*\Gamma$ and $u \in B(\epsilon)$. Then $\rho_{\psi}(g)u=\psi(g) u \psi(g)^{-1}$, which means that $_{\epsilon}\log(\rho_{\psi}(g)u)=\psi(g) _{\epsilon}\log{u} \psi(g)^{-1} $ since the matrix logarithm is invariant with respect to conjugation by a unitary matrix. 
	\end{proof}
	In particular, $\tilde{\W}$ is a real Banach space with an isometric action of ${}^*\Gamma$ (it is a real Banach ${}^*\Gamma$-module).
	\begin{remark}\label{realbanach}
		The fact that $\W=B(\epsilon)/I(\epsilon)$ is a real Banach ${}^*\Gamma$-module is useful in reducing (pointwise) stability of $\Gamma$ with respect to the family $\mathfrak{U}$ to showing that $H^2(\Gamma,\W)=0$, and this line of study is pursued in \cite{DCGLT} and \cite{oppen}. However, in our setting of uniform stability, the Banach structure of $\W$ is not as directly relevant. 
	\end{remark}
	Corresponding to the internal map $\alpha_{\psi}:{}^*\Gamma \times {}^*\Gamma \to \prod_{\U}U(k_n)$, define the internal map
	$$\alpha:{}^*\Gamma \times {}^*\Gamma \to \W$$
	\begin{equation}\label{alpha}
	\alpha(g_1,g_2) \coloneqq _{\epsilon}\log{\alpha_{\psi}(g_1,g_2)}
	\end{equation}
	Since $\alpha_{\psi}:{}^*\Gamma \times {}^*\Gamma \to \prod_{\U}U(k_n)$ takes values only in $B(\epsilon)$, it is clear that $\alpha:{}^*\Gamma \times {}^*\Gamma \to \W$ takes values only in $\W_b$. We shall denote by $\tilde{\alpha}:{}^*\Gamma \times {}^*\Gamma \to \tilde{\W}$ the map induced obtained by composing $\alpha$ with the canonical quotient map $\W_b \to \tilde{\W}$. 
	\begin{lemma}
		The map $\alpha:{}^*\Gamma \times {}^*\Gamma \to \W$ satisfies the following condition: for any $g_1,g_2,g_3 \in {}^*\Gamma$,
		$$\pi_{\psi}(g_1) \alpha(g_2,g_3)-\alpha(g_1g_2,g_3)+\alpha(g_1,g_2g_3)-\alpha(g_1,g_2) \in \W_{inf}$$
	\end{lemma}
	\begin{proof}
		Recall that $\alpha_{\psi}$ satisfies the following property: for $g_1,g_2,g_3 \in {}^*\Gamma$,
		$$\alpha_{\psi}(g_1,g_2)\alpha_{\psi}(g_1g_2,g_3)\alpha_{\psi}(g_1,g_2g_3)^{-1}=\rho_{\psi}(g_1)\alpha_{\psi}(g_2,g_3)$$
		The conclusion then follows from the fact that the map $_{\epsilon}\log: \prod_{\U}U(k_n) \to \W$ induces a ${}^*\Gamma$-equivariant group homomorphism between $B(\epsilon)/I(\epsilon)$ and $\W_b/\W_{inf}$. 
	\end{proof}
	Thus, the induced map $\tilde{\alpha}:{}^*\Gamma \times {}^*\Gamma \to \W$ satisifes the \textbf{$2$-cocycle} condition given by: for $g_1,g_2,g_3 \in {}^*\Gamma$,
	\begin{equation}\label{2cocycle}
	g_1 \cdot  \tilde{\alpha}(g_2,g_3)-\tilde{\alpha}(g_1g_2,g_3)+\tilde{\alpha}(g_1,g_2g_3)-\tilde{\alpha}(g_1,g_2) =0
	\end{equation}
	So transfering to the internal Lie algebra through the $_{\epsilon}\log$ map, we thus have a map $\alpha:{}^*\Gamma \times {}^*\Gamma \to \W$ which takes values in $\W_b$ and satisfies the 2-cocycle condition modulo $\W_{inf}$. Recall that the defect diminishing condition was implied by the following statement: suppose $\alpha_{\psi}:{}^*\Gamma \times {}^*\Gamma \to \prod_{\U}U(k_n)$ is an internal function taking values in $B(\epsilon)$ and such that $\tilde{\alpha}_{\phi}$ satisfies the $2$-cocycle condition. Then there exists an internal $\beta:{}^*\Gamma \to \prod_{\U}U(k_n)$ taking values in $B(\epsilon)$ such that $\tilde{\alpha}_{\phi}(g_1,g_2)=\tilde{\beta}(g_1)+g_1 \cdot \tilde{\beta}(g_2)-\tilde{\beta}(g_1g_2)$
	Using the $_{\epsilon}\log$ map to transfer to $\W$, we conclude with the following proposition summarizing our work so far:
	\begin{proposition}\label{mainthm1}
		Suppose for every internal $\alpha:{}^*\Gamma \times {}^*\Gamma \to \W$ with $Im(\alpha) \subseteq \W_b$ that satisfies the $2$-cocycle condition (\cref{2cocycle}), there exists an internal $\beta:{}^*\Gamma \to \W$ taking values in $\W_b$ such that 
		$$\tilde{\alpha}(g_1,g_2)=\tilde{\beta}(g_1)+g_1 \cdot \tilde{\beta}(g_2)-\tilde{\beta}(g_1g_2)$$
		then $\Gamma$ exhibits the defect diminishing property, and is therefore uniformly $\mathfrak{U}$-stable with a linear estimate.
	\end{proposition}
	Since we shall work with internal maps from $({}^*\Gamma)^2$ (or ${}^*\Gamma$) to $\W$ that take values in $\W_b$, it is helpful to describe such a map as an ultraproduct of bounded maps. Let $\{\alpha_{n} \in \ell^{\infty}(\Gamma^2, \mathfrak{u}(k_n))\}_{k=1}^{\infty}$ be a family of maps such that there exists a constant $C>0$ such that $\|\alpha_{n}\|_{\infty} \leq C$ for every $n \in \N$. Then it is clear that the ultraproduct $\alpha = \{\alpha_{n}\}_{\U}$ has image $Im(\alpha) \subseteq \W_b$. Conversely, 
	\begin{lemma}
		Let $\alpha:{}^*\Gamma \times {}^*\Gamma \to \W$ be an internal map with $Im(\alpha) \subseteq \W_b$. Then there exists a family $\{\alpha_{n} \in \ell^{\infty}(\Gamma^2, \mathfrak{u}(k_n))\}_{k=1}^{\infty}$ such that $\alpha = \{\alpha_{n}\}_{\U}$. Conversely, if $\{\alpha_{n} \in \ell^{\infty}(\Gamma^2, \mathfrak{u}(k_n))\}_{k=1}^{\infty}$ is a family of maps such that $\{\|\alpha_{n}\|_{\infty}\}_{\U} \in {}^*\R_b$, then $\alpha = \{\alpha_{n}\}_{\U}{}^*\Gamma \times {}^*\Gamma \to \W$ is an internal map with $Im(\alpha) \subseteq \W_b$.
	\end{lemma}
	\begin{proof}
		Since $\alpha$ is internal, it is of the form $\alpha=\{f_{n}\}_{\U}$ for a family of maps $f_{n}:\Gamma \times \Gamma \to \mathfrak{u}(k_n)$. For a subset $S \in \U$, suppose $f_{n}$ is unbounded for every $n \in S$. Then for each $n \in S$, there exists $x_{n},y_{n} \in \Gamma$ such that $f_{n}(x_{n},y_{n})$ has norm at least $n$. In particular, for $x=\{x_{n}\}_{\U}$ and $y=\{y_{n}\}_{\U}$, we have $\alpha(x,y) \notin \W_b$. This is a contradiction to the hypothesis that $Im(\alpha) \subseteq \W_b$. The converse is immediate from the definition of $\W_b$. 
	\end{proof}
	Thus, an internal map $\alpha:{}^*\Gamma \times {}^*\Gamma \to \W$ with $Im(\alpha) \subseteq \W_b$ can be described as $\{\alpha_{n}\}_{\U}$ where for every $k \in \N$, $\alpha_{n}:\Gamma \times \Gamma \to \mathfrak{u}(k_n)$ is a bounded map, and such that $\{\|\alpha_{n}\|_{\infty}\}_{\U} \in {}^*\R_b$. In other words, the internal map $\alpha$ is the ultraproduct of bounded maps, and is also bounded as an ultraproduct.\\
	From now on, we shall regard $\alpha$ as an element of $\prod_{\U}\ell^{\infty}(({}^*\Gamma)^2,\mathfrak{u}(k_n)$, which we shall henceforth denote $\L^{\infty}(({}^*\Gamma)^2,\W)$  (understood as the space of internal maps from $({}^*\Gamma)^2$ to $\W$. In fact, $\alpha$ is actually an element of $\L^{\infty}(({}^*\Gamma)^2,\W)_b$ since not only is it internally bounded, but also $\{\|\alpha_{n}\|_{\infty}\}_{\U} \in {}^*\R_b$. In general, we shall use the following notation:
	\begin{itemize}
		\item For $m \in \N$, the internal space $\L^{\infty}(({}^*\Gamma)^m,\W)$ is defined as
		$$\L^{\infty}(({}^*\Gamma)^m,\W) \coloneq \Big\{ \ell^{\infty}(\Gamma^m,\mathfrak{u}(k_n))  \Big\}_{\U}$$
		\item The (external) subspace of $\L^{\infty}(({}^*\Gamma)^m,\W)$ comprising internal functions with bounded (supremum) norm will be denoted $\L^{\infty}_b(({}^*\Gamma)^m,\W)$ while the (external) subspace of $\L^{\infty}_b(({}^*\Gamma)^m,\W)$ comprising internal functions with infinitesimal (supremum) norm will be denoted $\L^{\infty}_{inf}(({}^*\Gamma)^m,\W)$. 
		\item The quotient $\L^{\infty}_b(({}^*\Gamma)^m,\W)/\L^{\infty}_{inf}(({}^*\Gamma)^m,\W)$ shall be denoted $\tilde{\L}^{\infty}(({}^*\Gamma)^m,\W)$. This space comprises bounded maps from $({}^*\Gamma)^m$ to $\tilde{\W}$ that are induced from internal maps in $\L^{\infty}(({}^*\Gamma)^m,\W)$. As in \cref{realbanach}, $\tilde{\L}^{\infty}(({}^*\Gamma)^m,\W)$ is a real Banach space.
		\item For a map $f \in \L^{\infty}_b(({}^*\Gamma)^m,\W)$, we shall denote by $\tilde{f} \in \tilde{\L}^{\infty}(({}^*\Gamma)^m,\W)$ the composition of $f$ with the canonical quotient map $\L^{\infty}_b(({}^*\Gamma)^m,\W)\to \tilde{\L}^{\infty}(({}^*\Gamma)^m,\W)$.
	\end{itemize}
	For convenience, we restate \cref{mainthm1} in this notation:
	\begin{proposition}\label{mainthm2}
		Suppose for every $\alpha \in \L^{\infty}_b(({}^*\Gamma)^2,\W)$ that satisfies the $2$-cocycle condition (\cref{2cocycle}), there exists a $\beta \in \L^{\infty}_b(({}^*\Gamma)^1,\W)$ such that 
		$$\tilde{\alpha}(g_1,g_2)=\tilde{\beta}(g_1)+g_1 \cdot \tilde{\beta}(g_2)-\tilde{\beta}(g_1g_2)$$
		then $\Gamma$ exhibits the defect diminishing property, and is therefore uniformly $\mathfrak{U}$-stable with a linear estimate.
	\end{proposition}
	
	\subsection{Uniform Stability of Amenable groups}\label{ssec-amenable}
	In this section, we shall demonstrate an application of \cref{mainthm2} to amenable groups. The first step is to recognize the duality of $\W$ in an internal way.\\
	Consider the space $\W=\prod_{\U}\mathfrak{u}(k_n)$, and let $\W^{\sharp} \coloneqq \prod_{\U}(\mathfrak{u}(n))^*$. For the Banach space $\mathfrak{u}(n)$, consider its dual space $(\mathfrak{u}(n))^*$ and let $\langle \cdot \vert \cdot \rangle$ the canonical duality (for instance, if $\mathfrak{u}(n)$ is equipped with the Schatten $p$-norm for $p>1$, then $(\mathfrak{u}(n))^*$ comes equipped with the Schatten $q$-norm, where $1/p+1/q=1$). Denote by $\W^{\sharp}$ the ultraproduct $\prod_{\U}(\mathfrak{u}(k_n))^*$, and its external subsets $\W^{\sharp}_b$ and $\W^{\sharp}_{inf}$ as in (\cref{W_b}) and (\cref{W_inf}). Note that the internal pairing
	$$\langle \cdot \vert \cdot \rangle_{\U} : \W^{\sharp} \times \W \to {}^*\R$$
	induces a pairing 
	$$\langle \cdot \vert \cdot \rangle : \tilde{\W}^{\sharp}  \times \tilde{\W} \to \R$$ 
	Equivalently, $\W_b^{\sharp}$ comprises $\lambda \in \W^{\sharp}$ such that $\lambda(v) \in {}^*\R_b$ for every $v \in \W_b$, while $\W_{inf}^{\sharp}$ comprises $\lambda \in \W^{\sharp}$ such that $\lambda(v) \in {}^*\R_{inf}$ for every $v \in \W_b$.\\
	We can use the internal map $\pi_{\psi}:{}^*\Gamma \times \W \to \W$ to define the internal map
	\begin{equation}\label{eqn-dualW}
	\pi_{\psi}^{\sharp}:{}^*\Gamma \times \W^{\sharp} \to \W^{\sharp}
	\end{equation}
	which on $\lambda \in \W^{\sharp}$ and $v \in \W$ is defined to be
	$$\pi_{\psi}^{\sharp}(g)(\lambda)(v) \coloneqq \lambda(\pi_{\psi}(g^{-1})v)$$
	\begin{lemma}
		The internal map $\pi_{\psi}^{\sharp}$ restricts to a map on $\W^{\sharp}_b$ that induces an action of ${}^*\Gamma$ on $\W^{\sharp}_b/\W^{\sharp}_{inf}$.
	\end{lemma}
	\begin{proof}
		Let $\lambda \in \W_b^{\sharp}$. For $v \in \W$, since  $\pi_{\psi}^{\sharp}(g)(\lambda)(v)=\lambda(\pi_{\psi}(g^{-1})v)$, note that $\pi_{\psi}(g^{-1})v \in \W_b$ for $v \in \W_b$. Hence $\pi_{\psi}^{\sharp}(g)(\lambda) \in \W_b^{\sharp}$ for every $g \in {}^*\Gamma$. Similarly, for $\lambda \in \W_{inf}^{\sharp}$, $\pi_{\psi}^{\sharp}(\lambda) \in \W_{inf}^{\sharp}$. That this induces an action of ${}^*\Gamma$ on $\W_b^{\sharp}/\W_{inf}^{\sharp}$ follows easily from the fact that $\pi_{\psi}$ induces an action of ${}^*\Gamma$ on $\W_b/\W_{inf}$.  
	\end{proof}
	Going one step further, we can obtain a canonical identification of $\W$ with $(\W^{\sharp})^{\sharp}$ (which has the same norm as $\W$) so that we can regard $\W$ as $(\W^{\sharp})^{\sharp}$ with $(\pi_{\psi}^{\sharp})^{\sharp} = \pi_{\psi}$. This is true for $\W$ since $\W=\prod_{\U}\mathfrak{u}(n)$, and $\mathfrak{u}(n)$ is finite-dimensional for each $n \in \N$.
	\begin{remark}
		We shall often use this reflexivity property of $\W$ to regard its dual $\W^{\sharp}$ as its predual, so that $v \in \W$ acts on $\lambda \in \W^{\sharp}$ by $v \cdot \lambda = \lambda(v)$. However, note that for the following discussion of amenability, what we actually need is not reflexivity but merely the property that $\W$ is dual.
	\end{remark}
	Let us now recall the definition of amenability for discrete groups. While there are innumerable equivalent definitions of amenability, here we shall see the definition that is most relevant to us (later on in \S\ref{ssec-subg}, we shall study amenability and amenable actions in the locally compact case). Consider the Banach space $\ell^{\infty}(\Gamma)$ with the following action of $\Gamma$: for $g,x \in \Gamma$ and $f \in \ell^{\infty}(\Gamma)$, $(g\cdot f)(x)=f(g^{-1}x)$. A \emph{mean} on $\ell^{\infty}(\Gamma)$ is a bounded linear functional $m:\ell^{\infty}(\Gamma) \to \R$ such that $\|m\| \leq 1$, $m(1)=1$ and $m(f)\geq 0$ whenever $f \geq 0$. The mean $m$ is said to be $\Gamma$-\emph{invariant} if for every $g \in \Gamma$ and $f \in \ell^{\infty}(\Gamma)$, $m(g \cdot f)=m(f)$.
	\begin{definition}
		The discrete group $\Gamma$ is said to be \textbf{amenable} if there exists a $\Gamma$-invariant mean on $\ell^{\infty}(\Gamma)$.
	\end{definition}
	While the definition of amenability asks for a $\Gamma$-invariant mean on $\ell^{\infty}(\Gamma)$, this can be easily extended to obtain a $\Gamma$-\emph{equivariant} mean on $\ell^{\infty}(\Gamma,W)$ for a dual normed $W$-module (where the action of $\Gamma$ on $\ell^{\infty}(\Gamma,W)$ is given by $(g \cdot f)(x) = g \cdot f(g^{-1}x)$. The following lemma builds on this idea to construct an internal mean on $\L^{\infty}({}^*\Gamma,\W)$.
	\begin{lemma}
		Suppose $\Gamma$ is amenable. Then there exists an internal map $m_{in}:\L^{\infty}({}^*\Gamma,\W) \to \W$ such that $m_{in}$ induces a linear map $\tilde{m}:\tilde{\L}^{\infty}({}^*\Gamma,\W) \to \tilde{\W}$ satisfying the following two conditions:
		\begin{itemize}
			\item  Suppose $\tilde{f} \in \tilde{\L}^{\infty}({}^*\Gamma,\W)$ is the constant function $\tilde{f}(g) = \tilde{v}$ for every $g \in {}^*\Gamma$, then $\tilde{m}(\tilde{f})=\tilde{v}$.
			\item For  $\tilde{f} \in \tilde{\L}^{\infty}({}^*\Gamma,\W)$, $\|\tilde{m}(\tilde{f})\|\leq \|\tilde{f}\|$.
			\item For $g \in {}^*\Gamma$ and $\tilde{f} \in \tilde{\L}^{\infty}({}^*\Gamma,\W)$, $\tilde{m} \left(g \cdot \tilde{f} \right) = g \cdot \tilde{m}(\tilde{f})$.
		\end{itemize}
	\end{lemma}
	\begin{proof}
		Consider $f=\{f_n\}_{\U} \in \L^{\infty}({}^*\Gamma,\W)$. Since $\W=(\W^{\sharp})^{\sharp}$, for each $\lambda \in \W^{\sharp}$, we get an internal map 
		$$f_{\lambda}:{}^*\Gamma \to {}^*\R$$
		$$f_{\lambda}(x) \coloneqq f(x)(\lambda)$$
		Note that $f_{\lambda}$ being internal, is of the form $\{(f_{\lambda})_n\}_{\U}$ where $(f_{\lambda})_n \in \ell^{\infty}(\Gamma)$. This allows us to construct the internal map $m_{in}^{\lambda}:\L^{\infty}({}^*\Gamma,\W) \to {}^*\R$ as
		$$m_{in}^{\lambda}(f)=\{m\left( (f_{\lambda})_n   \right)\}_{\U}$$
		and finally $m_{in}:\L^{\infty}({}^*\Gamma,\W) \to (\W^{\sharp})^{\sharp}$ as
		$$m_{in}(f)(\lambda) \coloneqq m_{in}^{\lambda}(f)$$
		It is straightforward to check that $m_{in}$ as defined induces a linear map $\tilde{m}:\tilde{\L}({}^*\Gamma,\W) \to \W_b/\W_{inf}$. As for ${}^*\Gamma$-equivariance, this follows from the observation that $(g \cdot f)_{\lambda}(x)=\pi(g)f(g^{-1}x)(\lambda)$ while $(g \cdot f_{\lambda})(x)=f(g^{-1}x)(\lambda)$. The conditions on $\tilde{m}$ follow from the definition and properties of the $\Gamma$-invariant mean $m$ on $\ell^{\infty}(\Gamma)$. 
	\end{proof}
	We shall denote the internal mean above by $m_{in}^x$ (or $\tilde{m}^x$) when the mean is understood to be taken over $x \in \Gamma$. This would be particularly useful when working with multivariate maps where we fix certain coordinates to obtain a univariate map which we can take a mean over (as in the following \cref{amenable}). Note that in this notation, it is easy to see that the ${}^*\Gamma$-equivariance of the mean constructed above translates to the simpler (invariant) form: for $\tilde{f} \in \tilde{\L}^{\infty}({}^*\Gamma,\W)$ and $g \in {}^*\Gamma$, 
	$$\tilde{m}^x\left( \tilde{f}(gx) \right) = \tilde{m}^x\left( \tilde{f}(x)\right)$$
	We shall now use this internal map $m_{in}$ and the ${}^*\Gamma$-equivariant map $\tilde{m}$ to show the following:.
	\begin{proposition}\label{amenable}
		Suppose $\Gamma$ is amenable. Then for every $\alpha \in \L^{\infty}_b(({}^*\Gamma)^2,\W)$ that satisfies the $2$-cocycle condition (\cref{2cocycle}), there exists a $\beta \in \L^{\infty}_b({}^*\Gamma,\W)$ such that 
		$$\tilde{\alpha}(g_1,g_2)=\tilde{\beta}(g_1)+g_1 \cdot \tilde{\beta}(g_2)-\tilde{\beta}(g_1g_2)$$
	\end{proposition}
	\begin{proof}
		Suppose $\alpha \in \L^{\infty}L(({}^*\Gamma)^2,\W)$ satisfies the $2$-cocyle condition: for every $g_1,g_2,x \in {}^*\Gamma$, 
		$$\tilde{\alpha}(g_1,g_2) = g_1 \cdot  \tilde{\alpha}(g_2,x)-\tilde{\alpha}(g_1g_2,x)+\tilde{\alpha}(g_1,g_2x)$$
		For a fixed $g$, the map $\alpha_g:{}^*\Gamma \to \W$ defined as $\alpha_g(x) \coloneq \alpha(g,x)$ is clearly contained in $\L^{\infty}({}^*\Gamma,\W)$. 
		Define $\beta \in \L^{\infty}({}^*\Gamma,\W)$ as
		$$\beta(g) \coloneq m_{in}\left(\alpha_g\right)$$
		In other words, $\beta(g) = m_{in}^x\left(\alpha(g,x)\right)$. Then the $2$-cocycle condition satisfied by $\alpha$ immediately implies that
		$$\tilde{\alpha}(g_1,g_2) = g_1 \cdot  \tilde{\beta}(g_2)-\tilde{\beta}(g_1g_2)+\tilde{\beta}(g_1)$$
		
	\end{proof}
	Observe that the proof of \cref{amenable} is almost exactly on the lines of the proof that $\Hb^2(\Gamma,V)=0$ for amenable $\Gamma$ and dual normed $\Gamma$-module $W$ (refer to Theorem $3.6$ in \cite{frig} for more details).\\
	In light of \cref{amenable} and \cref{mainthm2}, we conclude that: 
	\begin{corollary}\label{amenable2}
		If $\Gamma$ is a discrete amenable group, then $\Gamma$ is uniformly $\mathfrak{U}$-stable with a linear estimate.
	\end{corollary}
	Note that while on the one hand \cref{amenable2} generalizes Kazhdan's result \cite{Kaz} to a larger family (where we allow any submultiplicative matrix norm as opposed to just the operator norm as in \cite{Kaz} and \cite{BOT}), we \emph{do not} prove here the analogous result of strong Ulam stability, where the family comprises groups of unitary operators on (possibly infinite-dimensional) Hilbert spaces.
	
	\section{Asymptotic Cohomology of Groups}\label{sec-coho}
	In this section, we shall formally define the asymptotic cohomology theory of (topological) groups, and study some basic properties along the lines of the theory of bounded cohomology. Recall that our goal is to prove uniform $\mathfrak{U}$-stability for lattices in higher rank Lie groups, so this forces us to develop the cohomology theory for locally compact groups (as opposed to just discrete groups as we briefly saw in \S\ref{ssec-logarithm} and \S\ref{ssec-amenable}).\\
	The basic objects we shall deal with are defined in \S\ref{ssec-basicG}, where we describe the category of asymptotic cohomology abstractly using tools from cohomological algebra. In \S\ref{ssec-subg}, we define the asymptotic cohomology of groups and relate it to the way it was motivated in \S\ref{ssec-amenable}. Finally, in \S\ref{ssec-subg}, we use Zimmer amenability and the functorial relations of \S\ref{ssec-basicG} to obtain other complexes that compute the same cohomology, which we shall use in \S\ref{sec-induc} and \S\ref{sec-mainproof}.
	
	\subsection{Basic Definitions and Some Cohomological Algebra}\label{ssec-basicG}
	Recall, from \S\ref{ssec-ultraproducts}, that we fix a non-principal ultrafilter $\U$ on $\N$ to define ultraproducts and internal objects. For convenience, we set some notation and conventions now. Let $C$ be a category, with $C$-objects and $C$-morphisms. We shall define a category ${}^*C$ in ${}^*Univ$ as follows:
	\begin{itemize}
		\item The objects of ${}^*C$, referred to as \emph{internal} $C$-objects, shall be ultraproducts $\prod_{\U}X_n$ where $\{X_n\}_{n\in \N}$ is an indexed collection of $C$-objects.
		\item The morphisms of ${}^*C$, referred to as \emph{internal morphism}, are of the form $\phi=\{\phi_n\}_{\U}$ (also denoted $\prod_{\U}\phi_n$), where $\{\phi_n\}_{n \in \N}$ is an indexed collection of $C$-morphisms.
	\end{itemize}
	Given a category $C$, the internal $C$-objects and internal $C$-morphisms form a category ${}^*C$ in ${}^*Univ$, and these two categories have the same first-order theories, allowing us to use the transfer principle, as remarked in \S\ref{ssec-ultraproducts}.
	\begin{definition}\label{internalprop}
		Let $A$ be a property of $C$-objects (resp, $C$-morphisms). Then $\prod_{\U}X_n$ (resp, $\phi:\prod_{\U}X_n \to \prod_{\U}Y_n$) has \emph{internal $A$} if $X_n$ (resp, $\phi_n$) has $A$ for every $n \in S$ with $S \in \U$. 
	\end{definition}
	For example, let $G$ be a locally compact, second countable topological group, and consider the ultrapower group ${}^*G$ and let $\mathcal{E}=\{E_n\}_{\U}$ be an internal Banach space. An internal map $\pi:{}^*G \times \mathcal{E} \to \mathcal{E}$, where $\pi=\{\pi_n:G \times E_n \to E_n\}$, is an \emph{internal action} of ${}^*G$ on $\mathcal{E}$ (or $\mathcal{E}$ is an \emph{internal ${}^*G$-representation}) if the map $\pi_n:G \times E_n \to E_n$ is an isometric $G$-representation for every $n \in \N$, and the internal map $\pi$ is \emph{internally continuous} if $\pi_n:G \times E_n \to E_n$ is continuous for every $n \in \N$ (here $G \times E_n$ is endowed with the product topology). All of these notions simply involve passing from standard categories to their internal counterparts in ${}^*Univ$.\\
	We shall now work with the category $Ban$ whose objects are (real) Banach spaces and whose morphisms are bounded linear maps, and study ${}^*Ban$. Consider the ultraproduct $\mathcal{E}= \prod_{\U}E_{n}$ of the real Banach spaces $\{E_{n}\}_{n \in \N}$, which is an \textbf{internal Banach space}. For an element $v \in \mathcal{E}$ where $v = \{v_n\}_{\U}$, we denote by $\|v\| \coloneq \{\|v_n\|\}_{\U} \in {}^*\R$. Given two internal Banach spaces $\mathcal{E}=\{E_n\}_{\U}$ and $\mathcal{F}=\{F_n\}_{\U}$, we shall denote by $\mathcal{Hom}(\mathcal{E},\mathcal{E})$ the set of internal morphisms between $\mathcal{E}$ and $\mathcal{F}$. These are exactly of the form $\phi \coloneq \{\phi_n:E_n \to F_n\}_{\U}$ where $\phi_n:E_n \to F_n$ is a bounded linear map for every $n \in \N$ (such maps are internal morphisms). Note that $\mathcal{Hom}(\mathcal{E},\mathcal{E})$ itself is an internal Banach space when endowed with the internal operator norm (that is, $\|\phi\|\coloneq \{\|\phi_n\|_{op}\}_{\U}$).\\
	Consider the set $\mathcal{Hom}(\mathcal{E},{}^*\R)$, which is the \emph{internal dual} of $\mathcal{E}$, denoted $\mathcal{E}^{\sharp}$. Explicitly, for each Banach space $E_n$ as above, let $E_n^{\sharp}$ denote its (constinuous) dual Banach space, and let $\langle \cdot \vert \cdot \rangle$ the canonical duality, and $\mathcal{E}^{\sharp}$ denote the ultraproduct $\prod_{\U}E_n^{\sharp}$. Note that we have an internal pairing 
	$$\langle \cdot \vert \cdot \rangle_{\U} : \mathcal{E}^{\sharp} \times \mathcal{E} \to {}^*\R$$
	We shall call $\mathcal{E}$ an \textbf{internal dual Banach space} if $\mathcal{E}$ is the internal dual of some internal Banach space (which we shall denote $\mathcal{E}^{\flat}=\prod_{\U}E_n^{\flat}$). For $\phi \in \mathcal{Hom}(\mathcal{E},\mathcal{E})$, we shall denote by $\phi^{\sharp} \in \mathcal{Hom}(\mathcal{E}^{\sharp},\mathcal{E}^{\sharp})$ the internal adjoint map with respect to the internal pairing above: that is $\phi^{\sharp}$ is such that for every $v \in \mathcal{E}$ and $w \in \mathcal{E}^{\sharp}$, $\langle \phi^{\sharp}w, v \rangle_{\U} = \langle w, \phi v \rangle_{\U}$.\\
	We shall now see some general functorial aspects of (standard) real Banach spaces and extend them naturally to internal Banach spaces. Let $X$, $Y$, $E$ and $F$ be (standard) real Banach  spaces. Let $B(X \times E)$ denote the Banach space of bounded bilinear forms $X\times E\to \R$, and $L(E,F)$ denote the Banach space of bounded linear functions from $E$ to $F$. Through the canonical pairing, note that $B(X \times E)$ is naturally isometrically isomorphic to the Banach space of bounded linear maps $E \to X^{\sharp}$ (and also to the Banach space of bounded linear maps $X \to E^{\sharp}$). That is,
	\begin{equation}\label{dunford2}
	B(X \times E) \cong L(E,X^{\sharp}) \cong L(X,E^{\sharp})
	\end{equation}
	Using these identifications, we now consider two functors:
	\begin{itemize}
		\item Given a bounded linear operator  $k:X^\sharp \to Y^\sharp$, define a bounded linear operator $k_{*}:B(X \times E)\to B(Y \times E)$ by $(k_{*}\beta)(\cdot,e)=k\beta(\cdot,e))$. Note that the correspondence $k\to k_{*}$ is covariant (although $k_*$ depends on $E$, for ease of notation, we assume the relevant Banach space from context).
		\item Given a bounded linear operator $\sigma:E \to F$, define a bounded linear operator $\sigma^*:B(X \times F)\to:B(X \times E)$ by $(\sigma^{*}\beta)(x,e)=\beta(x,\sigma(e))$. In this case, the correspondence $\sigma \to \sigma^*$ is contravariant (although $\sigma^*$ depends on $X$, for ease of notation, we assume the relevant Banach space from context).
	\end{itemize}
	\begin{proposition}\label{prop_lev1}
		Let $k:X^\sharp \to Y^\sharp$  and $\sigma: E\to F$, and consider the covariant $k_{*}:B(X \times E)\to B(Y \times E)$ and contravariant $\sigma^*:B(X \times F)\to:B(X \times E)$ as defined above. Then 
		\begin{itemize}
			\item $k_{*}\sigma^{*}= \sigma^{*}k_{*}$  
			\item $\|k_{*}\sigma^{*}\|\leq \|k_{*}\|\cdot\|\sigma^{*}\|$
		\end{itemize}
	\end{proposition}
	These notions extend easily to internal Banach spaces as well. Let $\mathcal{X}=\{X_n\}_{\U}$, $\mathcal{Y}=\{Y_n\}_{\U}$ and $\mathcal{E}=\{E_n\}_{\U}$, $\mathcal{F}=\{F_n\}_{\U}$ be internal Banach spaces. Denote by $\mathcal{B}(\mathcal{X}\times\mathcal{E})$ the internal Banach space $\mathcal{B}(\mathcal{X}\times \mathcal{E}) \coloneq \prod_{\U}B(X_n \times E_n)$ and by $\L(\mathcal{X},\mathcal{E}) \coloneq \prod_{\U}L(X_n,E_n)$. Then 
	$$\mathcal{B}(\mathcal{X} \times \mathcal{E}) \cong \L(\mathcal{E},\mathcal{X}^{\sharp}) \cong \L(\mathcal{X},\mathcal{E}^{\sharp})$$
	where the isomorphisms are internally isometric. Let $k:\mathcal{X}^{\sharp} \to \mathcal{Y}^{\sharp}$ (where $k=\{k_n\}_{\U}$), and $\sigma:\mathcal{E} \to \mathcal{F}$ (where $\sigma=\{\sigma_n\}_{\U}$) be internal morphisms. Then we have the covariant internal morphism $k_{*}: \mathcal{B}(\mathcal{X}\times \mathcal{E}) \to \mathcal{B}(\mathcal{Y}\times\mathcal{E})$ defined as $k_*=\{(k_n)_*\}_{\U}$), and the contravariant internal morphism $\sigma^*: \mathcal{B}(\mathcal{X}\times \mathcal{F}) \to \mathcal{B}(\mathcal{X}\times \mathcal{E})$ defined by $\sigma^*=\{\sigma_n^*\}_{\U}$ such that $k_*\sigma^*=\sigma^*k_*$ and $\|k_*\sigma^*\| \leq \|k_*\| \cdot \|\sigma^*\|$.\\
	Let $\mathcal{E}$ be an internal Banach space. Just as in as in (\cref{W_b}) and (\cref{W_inf}), the internal norm $\|\cdot\|:\mathcal{E} \to {}^*\R$ allows us to define special \emph{external} subsets as follows: 
	$$\mathcal{E}_b \coloneq \{v \in \mathcal{E}:\space \|v\| \in {}^*\R_b\}$$
	$$ \mathcal{E}_{inf} \coloneq \{v \in \mathcal{E} :\space \|v\| \in {}^*\R_{inf}\}$$
	and denote the ultralimit $\mathcal{E}_b/\mathcal{E}_{inf}$ by $\tilde{\mathcal{E}}$, which is a real Banach space \cite{banachultra}). The correspondence $\mathcal{E} \mapsto \tilde{\mathcal{E}}$ is \emph{not} a functor from ${}^*Ban$ to $Ban$ as such, because an internal morphism $\phi \in \mathcal{Hom}(\mathcal{E},\mathcal{F})$ need not induce a morphism $\tilde{\phi}:\tilde{\mathcal{E}} \to \tilde{\mathcal{F}}$ in general. However, if we restrict ourselves to bounded objects and morphisms in ${}^*Ban$, then we do get a functorial correspondence. That is, consider the external subsets $\mathcal{Hom}_b(\mathcal{E},\mathcal{F})$ and $\mathcal{Hom}_{inf}(\mathcal{E},\mathcal{F})$. Then
	\begin{proposition}\label{hom}
		Any $\phi \in \mathcal{Hom}_b(\mathcal{E},\mathcal{F})$, induces a map $\tilde{\phi} \in Hom(\tilde{\mathcal{E}},\tilde{\mathcal{F}})$.
	\end{proposition}
	\begin{proof}
		Note that $\phi \coloneq \{\phi_n:E_n \to F_n\}_{\U}$, where each  $\phi_n:E_n \to F_n$ is a bounded linear map for every $n \in \N$. Since $\phi \in \mathcal{Hom}_b(\mathcal{E},\mathcal{F})$, this means that $\phi(\mathcal{E}_b) \subseteq \mathcal{F}_b$, and $\phi(\mathcal{E}_{inf}) \subseteq \mathcal{F}_{inf}$, thus inducing a bounded linear map $\tilde{\phi}:\tilde{\mathcal{E}} \to \tilde{\mathcal{F}}$ between the real Banach spaces $\tilde{\mathcal{E}}$ and $\tilde{\mathcal{F}}$. 
	\end{proof}
	For example, if $\phi:\mathcal{E} \to \mathcal{F}$ is an internal isometry, then $\phi \in \mathcal{Hom}_b(\mathcal{E},\mathcal{F})$ induces $\tilde{\phi}:\tilde{\mathcal{E}} \to \tilde{\mathcal{F}}$.\\
	Observe that $\tilde{\mathcal{Hom}}(\mathcal{E},\mathcal{F})$ is a subspace of $Hom(\tilde{\mathcal{E}},\tilde{\mathcal{F}})$. The latter comprises the Banach space of all bounded linear maps between real Banach spaces $\tilde{\mathcal{E}}$ and $\tilde{\mathcal{F}}$, while the former is a Banach subspace comprising those bounded linear maps that were induced from internal morphisms from $\mathcal{E}$ to $\mathcal{F}$.\\
	\begin{proposition}\label{pairing}
		Let $\mathcal{E}$ be an internal Banach space with dual $\mathcal{E}^{\sharp}$. Then the internal pairing $\langle \cdot \rangle_{\U} \in \mathcal{B}_b(\mathcal{E}^{\sharp}, \mathcal{E})$. Furthermore, for $\phi \in \mathcal{Hom}_b(\mathcal{E},\mathcal{E})$, its internal adjoint $\phi^{\sharp} \in \mathcal{Hom}_b(\mathcal{E}^{\sharp},\mathcal{E}^{\sharp})$.
	\end{proposition}
	Thus, many of our functorial results about internal Banach spaces in ${}^*Ban$ pass through when we restrict to bounded elements, and induce corresponding results in $Ban$.\\  
	Our main structure of interest is an asymptotic variant of internal ${}^*G$-representations using the external subsets $\mathcal{E}_b$ and $\mathcal{E}_{inf}$:
	\begin{definition}
		Let $G$ be a locally compact, second countable topological group, and $\mathcal{E}$ be an internal Banach space. An internal isometry $\pi:{}^*G \times \mathcal{E} \to \mathcal{E}$ be an internal isometry such that it induces an action $\tilde{\pi}:{}^*G \times \tilde{\mathcal{E}} \to \tilde{\mathcal{E}}$ of ${}^*G$ on the real Banach space $\tilde{\mathcal{E}}$. The internal Banach space $\mathcal{E}$, equipped with such an internal map $\pi$, is called an \textbf{asymptotic Banach ${}^*G$-module} with asymptotic ${}^*G$-representation $\pi$. 
	\end{definition}
	We shall denote an asymptotic Banach ${}^*G$-module either by  $(\pi,\mathcal{E})$, or just $\mathcal{E}$ if the map $\pi$ can implicitly be assumed in context.\\
	Observe that the real Banach space $\tilde{\mathcal{E}}$ is a (true) representation of ${}^*G$ through $\tilde{\pi}$. An element $v \in \mathcal{E}_b$ is said to be \emph{asymptotically fixed} by $g \in {}^*G$ if its image $\tilde{v} \in \tilde{\mathcal{E}}$ is (truly) fixed by $g$. The set of asymptotically ${}^*G$-fixed elements of $\mathcal{E}$ shall be denoted $\mathcal{E}_b^{\sim {}^*G}$. More generally, for an internal subgroup $\mathcal{N} \leq {}^*G$, the set of asymptotically $\mathcal{N}$-fixed elements of $\mathcal{E}$ shall be denoted $\mathcal{E}_b^{\sim \mathcal{N}}$ \\
	For an asymptotic Banach ${}^*G$-module $\mathcal{E}$, consider its dual  internal Banach space $\mathcal{E}^{\sharp}=\L(\mathcal{E},{}^*\R)$. In this case, the internal map $\pi:{}^*G \times\mathcal{E} \to \mathcal{E}$ defines an internal map 
	$$\pi^{\sharp}:{}^*G \times \mathcal{E}^{\sharp} \to \mathcal{E}^{\sharp}$$
	$$\pi^{\sharp}(g)(\lambda)(v) = \lambda\left( \pi(g)^{-1}v \right)$$
	in the usual way, and it is easy to check that $\mathcal{E}^{\sharp}$ an asymptotic Banach ${}^*G$-module $(\pi^{\sharp},\mathcal{E}^{\sharp})$. Essentially, we are simply defining the internal adjoint of $\pi(g)$ with respect to the pairing of $\mathcal{E}$ and $\mathcal{E}^{\sharp}$, and the functorial results of \cref{pairing} and \cref{hom} ensure that the map $\pi^{\sharp}$ defined this way is an asymptotic ${}^*G$-representation of ${}^*G$ on $\mathcal{E}^{\sharp}$. An asymptotic Banach ${}^*G$-module $(\pi,\mathcal{E})$ is called a \textbf{dual asymptotic Banach ${}^*G$-module} if $(\pi,\mathcal{E})$ is the dual of an asymptotic Banach ${}^*G$-module $(\pi^{\flat},\mathcal{E}^{\flat})$.
	
	More generally, let $(\pi,\mathcal{E})$ and $(\rho,\mathcal{F})$ be asymptotic Banach ${}^*G$-modules. Then $\mathcal{L}(\mathcal{E},\mathcal{F})$ and $\mathcal{B}(\mathcal{E},\mathcal{F})$ can also be regarded as asymptotic Banach ${}^*G$-modules in a natural way. 
	\begin{definition}
		Let $(\pi,\mathcal{E})$ and $(\rho,\mathcal{F})$ be asymptotic Banach ${}^*G$-modules. An \textbf{asymptotic ${}^*G$-morphism} from $\mathcal{E}$ to $\mathcal{F}$ is a map $\phi \in \mathcal{L}_b(\mathcal{E},\mathcal{F})$ such that the induced $\tilde{\phi}:\tilde{\mathcal{E}} \to \tilde{\mathcal{F}}$ is ${}^*G$-equivariant (that is, $\tilde{\phi}$ is a morphism of ${}^*G$-representations $\tilde{\mathcal{E}}$ and $\tilde{\mathcal{F}}$).
	\end{definition}
	In other words, an asymptotic ${}^*G$-morphism is an element of $\L_b(\mathcal{E},\mathcal{F})$ whose image in $\tilde{\L}(\mathcal{E},\mathcal{F})$ is a ${}^*G$-fixed point, and the set of such elements is denoted $\mathcal{Hom}_{G}(\mathcal{E},\mathcal{F})$. The following proposition tells us that the functors $k \mapsto k_*$ and $\sigma \mapsto \sigma^*$, defined for internal Banach spaces, respect the asymptotic ${}^*G$-representations.
	\begin{proposition}\label{ksigma2}
		Let $\mathcal{E}, \mathcal{F}, \mathcal{X}$ and $\mathcal{Y}$ be asymptotic Banach ${}^*G$-modules, and $k:\mathcal{X}^{\sharp} \to \mathcal{Y}^{\sharp}$ and $\sigma:\mathcal{E} \to \mathcal{F}$ be asymptotic ${}^*G$-morphisms. Then $k_*:\mathcal{B}(\mathcal{X},\mathcal{E}) \to \mathcal{B}(\mathcal{Y},\mathcal{E})$ and $\sigma^*:\mathcal{B}(\mathcal{X},\mathcal{F}) \to \mathcal{B}(\mathcal{X},\mathcal{E})$ are also asymptotic ${}^*G$-morphisms. 
	\end{proposition}
	
	\begin{definition}
		An asymptotic ${}^*G$-cochain complex $(\mathcal{E}^{\bullet},d^{\bullet})$ is a $\Zplus$-indexed sequence
		$$\begin{tikzcd}
		0 \arrow[r,"d^0"] & \mathcal{E}^0  \arrow[r, "d^{1}"]&  \mathcal{E}^1 \arrow[r, "d^2"] & \mathcal{E}^2 \arrow[r, "d^3"]  & \mathcal{E}^3 \arrow[r, "d^4"] & \dots
		\end{tikzcd}$$
		where for every $n \geq 0$, $\mathcal{E}^n$ is an asymptotic Banach ${}^*G$-module, $d^n$ is an asymptotic ${}^*G$-morphism, and $d^{n+1}d^n(\mathcal{E}^n_b) \subseteq \mathcal{E}^{n+2}_{inf}$.
	\end{definition}
	We shall also denote $(\mathcal{E}^{\bullet},d^{\bullet})$ by just $\mathcal{E}^{\bullet}$ if the maps $d^n$ are assumed from context, and also assume that $\mathcal{E}^{-1}=0$ to make statements with indices easier. Observe that the condition $d^{n+1}d^n(\mathcal{E}^{n-1}_b) \subseteq \mathcal{E}^{n+1}_{inf}$ is a relaxation to infinitesimals of the usual condition on differential maps. In fact, since the maps $d^n$ are asymptotic ${}^*G$-morphisms, the asymptotic ${}^*G$-cochain complex induces a (true) cochain complex 
	$$\begin{tikzcd}
	0 \arrow[r,"\tilde{d}^0"] & (\tilde{\mathcal{E}}^0)^{{}^*G}  \arrow[r, "\tilde{d}^{1}"]&  (\tilde{\mathcal{E}}^1)^{{}^*G}  \arrow[r, "\tilde{d}^{2}"] & (\tilde{\mathcal{E}}^2)^{{}^*G} \arrow[r, "\tilde{d}^{3}"]  &(\tilde{\mathcal{E}}^3)^{{}^*G}  \arrow[r, "\tilde{d}^{4}"] & \dots
	\end{tikzcd}$$
	allowing us to define the following:
	\begin{definition}
		The \textbf{asymptotic cohomology} of the asymptotic ${}^*G$-cochain complex $(\mathcal{E}^{\bullet},d^{\bullet})$, denoted $\Ha^{\bullet}(\mathcal{E}^{\bullet},d^{\bullet})$ or $\Ha^{\bullet}(\mathcal{E}^{\bullet})$, is defined to be
		$$\Ha^m(\mathcal{E}^{\bullet}) \coloneq ker(\tilde{d}^{m+1})/Im(\tilde{d}^m)$$ 
		An element $\alpha \in \mathcal{E}_b$ such that $\tilde{\alpha} \in ker(\tilde{d}^{m+1})$ is called an \textbf{asymptotic $m$-cocycle}, and will be called an \textbf{asymptotic $m$-coboundary} if $\tilde{\alpha} \in Im(\tilde{d}^m)$.
	\end{definition}
	To understand the correspondence $\mathcal{E}^{\bullet} \mapsto \Ha^{\bullet}(\mathcal{E}^{\bullet})$, we now define morphisms and homotopies of asymptotic ${}^*G$-cochain complexes.
	\begin{definition}
		Let $(\mathcal{E}^{\bullet},d_\mathcal{E}^{\bullet})$ and $(\mathcal{F}^{\bullet},d_\mathcal{F}^{\bullet})$ be asymptotic ${}^*G$-cochain complexes. An asymptotic ${}^*G$-morphism $\alpha^{\bullet}:\mathcal{E}^{\bullet} \to \mathcal{F}^{\bullet}$ between $(\mathcal{E}^{\bullet},d_\mathcal{E}^{\bullet})$ and $(\mathcal{F}^{\bullet},d_\mathcal{F}^{\bullet})$ is a family of asymptotic ${}^*G$-morphisms $\alpha^n:\mathcal{E}^n \to \mathcal{F}^n$ for every $n \in \Zplus$ such that for every $n \in \Zplus$,
		$$(d^{n+1}_\mathcal{F} \alpha^n - \alpha^{n+1} d^n_\mathcal{E})(\mathcal{E}^n_b) \subseteq \mathcal{F}^{n+1}_{inf}$$
	\end{definition}
	Again, this simply means that we have a ${}^*G$-morphism of the cochain complexes 
	$$\begin{tikzcd}
	0 \arrow[r,"\tilde{d}^0"] & (\tilde{\mathcal{E}}^0)^{{}^*G}  \arrow[d, "\tilde{\alpha}^0"] \arrow[r, "\tilde{d}^{1}"]&  (\tilde{\mathcal{E}}^1)^{{}^*G} \arrow[d, "\tilde{\alpha}^1"] \arrow[r, "\tilde{d}^{2}"] & (\tilde{\mathcal{E}}^2)^{{}^*G} \arrow[d, "\tilde{\alpha}^2"] \arrow[r, "\tilde{d}^{3}"]  &(\tilde{\mathcal{E}}^3)^{{}^*G} \arrow[d, "\tilde{\alpha}^3"] \arrow[r, "\tilde{d}^{4}"] & \dots\\
	0 \arrow[r,"\tilde{d}^0"] & (\tilde{\mathcal{F}}^0)^{{}^*G}  \arrow[r, "\tilde{d}^{1}"]&  (\tilde{\mathcal{F}}^1)^{{}^*G}  \arrow[r, "\tilde{d}^{2}"] & (\tilde{\mathcal{F}}^2)^{{}^*G} \arrow[r, "\tilde{d}^{3}"]  &(\tilde{\mathcal{F}}^3)^{{}^*G}  \arrow[r, "\tilde{d}^{4}"] & \dots
	\end{tikzcd}$$
	Let $\alpha^{\bullet}$ be an asymptotic ${}^*G$-morphism between the asymptotic ${}^*G$-cochain complexes $\mathcal{X}^{\bullet}$ and $\mathcal{Y}^{\bullet}$. While this clearly gives us an induced map $\Ha^{\bullet}(\mathcal{X}^{\bullet}) \to \Ha^{\bullet}(\mathcal{Y}^{\bullet})$, we now describe when two such asymptotic ${}^*G$-morphisms $\alpha^{\bullet}$ and $\beta^{\bullet}$ correspond to the same induced maps of cohomologies.
	\begin{definition}
		Let $\alpha^{\bullet},\beta^{\bullet}:\mathcal{X}^{\bullet} \to \mathcal{Y}^{\bullet}$ be two asymptotic ${}^*G$-morphisms between the asymptotic ${}^*G$-cochain complex $\mathcal{X}^{\bullet}$ and $\mathcal{Y}^{\bullet}$. Then $\alpha^{\bullet}$ is said to be asymptotically ${}^*G$-homotopic to $\beta^{\bullet}$ if there exists a family of asymptotic ${}^*G$-morphisms $\sigma^n:\mathcal{E}^n \to \mathcal{F}^{n-1}$ such that
		$$\tilde{d}^n\tilde{\sigma}^n+\tilde{\sigma}^{n+1}\tilde{d}^{n+1}=\tilde{\alpha}^n-\tilde{\beta}^n$$
		for every $n \in \Zplus$.
	\end{definition}
	The following lemma lists results that follow from standard cohomological techniques:
	\begin{proposition}\label{homotopy}
		Let $\mathcal{X}^{\bullet}$ and $\mathcal{Y}^{\bullet}$ be asymptotic ${}^*G$-cochain complexes. 
		\begin{itemize}
			\item Suppose $\alpha^{\bullet},\beta^{\bullet}:\mathcal{X}^{\bullet} \to \mathcal{Y}^{\bullet}$ are asymptotic ${}^*G$-morphisms such that $\alpha^{\bullet}$ is asymptotically ${}^*G$-homotopic to $\beta^{\bullet}$. Then they induce the same map $\Ha^{\bullet}(\mathcal{X}^{\bullet}) \to \Ha^{\bullet}(\mathcal{Y}^{\bullet})$ at the level of cohomology.
			\item Suppose $\alpha^{\bullet}:\mathcal{X}^{\bullet} \to \mathcal{Y}^{\bullet}$ and $\beta^{\bullet}:\mathcal{Y}^{\bullet} \to \mathcal{X}^{\bullet}$ are asymptotic ${}^*G$-morphisms such that $\alpha^{\bullet}\circ \beta^{\bullet}:\mathcal{Y}^{\bullet} \to \mathcal{Y}^{\bullet}$ is asymptotically ${}^*G$-homotopic to the identity on $\mathcal{Y}^{\bullet}$, and $\beta^{\bullet} \circ \alpha^{\bullet}:\mathcal{X}^{\bullet} \to \mathcal{X}^{\bullet}$ is asymptotically ${}^*G$-homotopic to the identity on $\mathcal{X}^{\bullet}$. Then $\Ha^{\bullet}(\mathcal{X}^{\bullet})$ is isomorphic to $\Ha^{\bullet}(\mathcal{Y}^{\bullet})$. In this case, the maps $\alpha^{\bullet}$ and $\beta^{\bullet}$ are called asymptotic ${}^*G$-homotopy equivalences between $\mathcal{X}^{\bullet}$ and $\mathcal{Y}^{\bullet}$. 
		\end{itemize}
	\end{proposition}
	
	Given an asymptotic ${}^*G$-cochain complex $(\mathcal{X}^{\bullet},d^{\bullet})$ and another asymptotic Banach  ${}^*G$-module $\mathcal{E}$, we can use the functors $d^n \to d^n_*$ to construct an asymptotic ${}^*G$-cochain complex 
	$$\begin{tikzcd}
	0 \arrow[r,"d^0"] & \mathcal{B}(\mathcal{X}^0 \times \mathcal{E})  \arrow[r, "d^{1}"]&  \mathcal{B}(\mathcal{X}^1 \times \mathcal{E}) \arrow[r, "d^2"] & \mathcal{B}(\mathcal{X}^2 \times \mathcal{E}) \arrow[r, "d^3"]  & \mathcal{B}(\mathcal{X}^3 \times \mathcal{E}) \arrow[r, "d^4"] & \dots
	\end{tikzcd}$$
	which we shall denote $\mathcal{B}(\mathcal{X}^{\bullet},\mathcal{E})$.  Combining the functorial properties of $k \to k_*$ and $\sigma \to \sigma^*$ and the fact that they respect the structure of the asymptotic ${}^*G$-representations, we get the following:
	\begin{proposition}\label{bilinear}
		Let $k^{\bullet}$ be an asymptotic ${}^*G$-morphism between the asymptotic ${}^*G$-cochain complexes $(\mathcal{X}^{\bullet})^{\sharp}$ and $(\mathcal{Y}^{\bullet})^{\sharp}$, and $\mathcal{E}$ be an asymptotic Banach ${}^*G$-module. Then $k^{\bullet}_*$ is an asymptotic ${}^*G$-morphism between the asymptotic ${}^*G$-cochain complexes $\mathcal{B}(\mathcal{X}^{\bullet} \times \mathcal{E})$ and $\mathcal{B}(\mathcal{Y}^{\bullet} \times \mathcal{E})$.
	\end{proposition}
	
	Note that while we cannot conclude anything about $\Ha^{\bullet}(\mathcal{B}(\mathcal{X}^{\bullet},\mathcal{E}))$ from $\Ha^{\bullet}(\mathcal{X}^{\bullet})$, we can still use \cref{homotopy} to show that:
	\begin{lemma}\label{homotopy2}
		Let $k^{\bullet}:(\mathcal{X}^{\bullet})^{\sharp} \to (\mathcal{Y}^{\bullet})^{\sharp}$ and $j^{\bullet}:(\mathcal{Y}^{\bullet})^{\sharp} \to (\mathcal{X}^{\bullet})^{\sharp}$ be asymptotic ${}^*G$-homotopy equivalences (as in \cref{homotopy}) between the asymptotic ${}^*G$-cochain complexes $(\mathcal{X}^{\bullet})^{\sharp}$ and $(\mathcal{Y}^{\bullet})^{\sharp}$. Let $\mathcal{E}$ be an asymptotic Banach ${}^*G$-module. Then $k_*^{\bullet}:\mathcal{B}(\mathcal{X}\times \mathcal{E}) \to \mathcal{B}(\mathcal{Y}\times \mathcal{E})$ and $j_*^{\bullet}:\mathcal{B}(\mathcal{Y}\times \mathcal{E}) \to \mathcal{B}(\mathcal{X}\times \mathcal{E})$ be asymptotic ${}^*G$-homotopy equivalences between the asymptotic ${}^*G$-cochain complexes $\mathcal{B}(\mathcal{X}^{\bullet},\mathcal{E})$ and $\mathcal{B}(\mathcal{Y}^{\bullet},\mathcal{E})$. 
	\end{lemma}

	\subsection{The $L^{\infty}$-cohomology and $\Ha^{\bullet}(G,\V)$}\label{ssec-coho}
	We now come to our definition of the asymptotic cohomology of the group $G$ with coefficients in a dual asymptotic Banach ${}^*G$-module $(\pi,\V)$, where $\V=\{V_n\}_{\U}$. We shall first define it explicitly, and then relate it to the notions discussed in \S\ref{ssec-basicG} to derive further results.\\
	Our objects of interest shall be ultraproducts of $L^{\infty}$-spaces. Since these are not spaces of functions but spaces of equivalence classes of functions upto null sets, we now quickly review some notions regarding its essential image. For $m \geq 0$ and dual Banach space $V$ (with predual $V^{\flat}$, and equipped with the weak-$*$ topology), $L_{w*}^{\infty}(G^m,V)$ denotes the Banach space of (equivalence classes of) essentially bounded weak-$*$ measurable maps from $G^m$ to $V$ with (essential) supremum norm $\|\cdot\|$.\\
	For a weak-$*$ measurable function $f:G^m \to V$, the set $f^{-1}(O_u)$ is measurable in $G^m$. Define the \emph{essential image} $Im(f)$ to be
	$$Im(f)=\{u\in V\;|\; \mu(f^{-1}(O_u))>0 \text{ for every weak-$*$ neighborhood } O_u\mbox{ of } v\}$$
	where $\mu$ is the Haar measure on $G^m$. We  extend this notion to function classes in $L_{w*}^{\infty}(G^m,V)$ as well. Let $f \in L_{w*}^{\infty}(G^m,V)$.
	\begin{itemize}
		\item For any functions $f_1,f_2$ in the class of $f$, $Im(f_1)=Im(f_2)$. That is, essential image is independent of the representative of $f$, allowing us to define $Im(f)$ as the essential image of any representative in its class.
		\item Let $f_0$ be a representative of $f$, then $Im(f)\subseteq \overline{range(f_0)}$ (for $u\in Im(f)$, any $O_u$ intersects with $range(f_0)$).
		\item If $u\not\in Im(f)$ then there is a representative $f_0$ such that $u\not\in\overline{range(f_0)}$ (in this case, there is a neighborhood  $O_u$ such that $f^{-1} (O_u)$ has measure $0$, and so, one may redefine $f$ on $f^{-1}(O_u)$ to avoid $O_u$).
		\item $Im(f)=\bigcap \{\overline{range(g)}\;|\; g\mbox{ is a representative of} f \}$.
		\item There is a representative $f_0$ of $f$ such that $range(f_0)\subseteq Im(f)$ (this is because $Im(f)$, being a norm bounded closed subset of $V$, is second countable, and so 
		$f^{-1}(Im(f))$ is measurable and of the full measure).
	\end{itemize}     
	Formally, we shall call $f \in L_{w*}^{\infty}(G^m,V)$ essentially constant if there exists $v \in V$ with $Im(f)=\{v\}$.\\
	Consider the internal Banach space 
	$$\L^{\infty}\left( ({}^*G)^m,\V \right) \coloneq \prod_{\U}L_{w*}^{\infty}(G^m,V_n)$$
	For $f=\{f_n\}_{\U} \in \L^{\infty}\left( ({}^*G)^m,\V \right)$, we denote by $\|f\|$ the hyperreal $\{\|f_n\|\}_{\U} \in {}^*\R$ and by $Im(f)$ the subset $\{Im(f_n)\}_{\U} \subseteq \V$. Note that the external subset $\L^{\infty}_{b}(({}^*G)^m,\V)$ is the subset of function classes $f=\{f_n\}_{\U}$ such that $Im(f) \subseteq \V_b$, while while $\L^{\infty}_{inf}(({}^*G)^m,\V)$ is the subset of function classes $f=\{f_n\}_{\U}$ such that $Im(f) \subseteq \V_{inf}$. Observe that $Im(f)$ is an internal subset of $\V$, while $\V_b$ and $\V_{inf}$ are external. 
	\begin{claim}
		For $f=\{f_n\}_{\U} \in \L^{\infty}_{b}(({}^*G)^m,\V)$, $Im(f) \subseteq \V_b^{\sim {}^*G}$ iff there exists an internal function $f'=\{f'_n\}_{\U}$ with $f'_n$ in the class of $f_n$ for $n \in \U$ such that $range(f')=\{range(f_n')\}_{\U} \subseteq \V_b^{\sim {}^*G}$.  
	\end{claim}
	\begin{proof}
		Suppose $Im(f) \subseteq \V_b^{\sim {}^*G}$. Then since there exists an internal representative $f'$ with $range(f') \subseteq Im(f)$, we have $range(f') \subseteq \V_b^{\sim {}^*G}$ as well for this $f'$.\\
		Conversely, suppose there exists an internal representative function $f'$ with $range(f') \subseteq \V_b^{\sim {}^*G}$. Let $\epsilon = \{\epsilon_n\}_{\U}$ where $\epsilon_n = ess.sup\left \|\pi_n(g)v-v\| \; | \; v \in range(f'_n), g \in G \right)$. Note that $\epsilon \in {}^*\R_{inf}$ since $range(f') \subseteq \V_b^{\sim {}^*G}$. This allows us to express the internal subset $range(f')$ as a subset of the internal set $E=\{E_n\}_{\U}$ where
		$$E_n=\{ v \in V_n \; | \; \|v\|\leq \|f'_n\| \text{ and } \forall g \in G, \|\pi_n(g)v - v \| \leq \epsilon_n \}$$
		That is, $range(f') \subseteq E \subseteq \V_b^{\sim {}^*G}$ with $E$ being internal. Also note that for $n \in \U$, $E_n$ is closed in the weak-$*$ topology. Hence $Im(f) \subseteq E \subseteq \V_b^{\sim {}^*G}$.
	\end{proof}
	The internal map $\pi:{}^*G \times \V \to \V$ can be used to define the internal map $\tau^m:{}^*G \times \L^{\infty}(({}^*G)^m,\V) \to \L^{\infty}(({}^*G)^m,\V)$ for each $m \geq 0$ defined as
	\begin{equation}\label{tau1}
	(\tau^m(g)(f))(g_1,g_2,\dots,g_m) \coloneqq \pi(g)f(g^{-1}g_1,\dots,g^{-1}g_m)
	\end{equation}
	for every $g \in {}^*G$ and every $g_1,\dots,g_m \in {}^*G$. This internal map $\tau^m:{}^*G \times \L^{\infty}(({}^*G)^m,\V) \to \L^{\infty}(({}^*G)^m,\V)$ induces an action $\tilde{\tau}^m$ of ${}^*G$ on the ultralimit space $\tilde{\L}^{\infty}(({}^*G)^m,\V)$. In particular, this means that $(\tau^m,\L^{\infty}(({}^*G)^m,\V))$ is an asymptotic Banach ${}^*G$-module.\\
	For simplicity, we shall denote the induced action of $g \in {}^*G$ on $\tilde{f} \in \tilde{\L}^{\infty}(({}^*G)^m,\V)$ simply by $g \cdot \tilde{f}$. We shall denote by $\L^{\infty}(({}^*G)^m,\V)^{\sim {}^*G}$ the set of elements $f \in \L^{\infty}(({}^*G)^m,\V)$ such that $\tilde{f} \in \tilde{\L}^{\infty}(({}^*G)^m,\V)^{{}^*G}$. Such an element $f$ shall be referred to as \emph{asymptotically ${}^*G$-equivariant}.
	\begin{remark}
		Later we shall also deal with (truly) ${}^*G$-\emph{invariant}maps $f \in \L^{\infty}({}^*G^m,\V)$ with an internal right action of $G$ on the domain. In this case, we mean that $f(x)=f(xg)$ for \emph{every} $g \in {}^*G$, and for $x=\{x_n\}_{\U} \subseteq \{X_n\}_{\U} \subseteq {}^*G$ where $X_n \subseteq G$ is co-null in $G$. We shall refer to this as $f(x)=f(xg)$ for \emph{every} $g \in {}^*G$, and \emph{almost every} $x \in {}^*G$, for convenience.
	\end{remark}
	For each $m \geq 0$, define the internal map
	$$d^m: \L^{\infty}(({}^*G)^m,\V) \to \L^{\infty}(({}^*G)^{m+1},\V)$$
	$$d^mf (g_0,\dots,g_{m}) \coloneq \sum \limits_{j=0}^m (-1)^j f(g_0,\dots,\hat{g_j},\dots,g_m)$$
	It is clear that $d^m \circ d^{m-1} = 0$ for every $m \geq 1$. Note that $d^m$ induces a map 
	$$\tilde{d}^m:\tilde{\L}^{\infty}(({}^*G)^m,\V) \to \tilde{\L}^{\infty}(({}^*G)^{m+1},\V)$$ 
	Furthermore, when we restrict to $\L^{\infty}(({}^*G)^m,\V)^{\sim {}^*G}$,
	\begin{lemma}
		For $f \in \L^{\infty}(({}^*G)^m,\V)^{\sim {}^*G}$, $d^mf \in \L^{\infty}(({}^*G)^{m+1},\V)^{\sim {}^*G}$. That is, $d^{\bullet}$ is an asymptotic ${}^*G$-morphism of asymptotic Banach ${}^*G$-modules.
	\end{lemma}
	\begin{proof}
		Consider 
		$$\pi(g)(d^mf(g_0,\dots,g_m)) =  \sum \limits_{j=0}^m (-1)^j \pi(g) f(g_0,\dots,\hat{g_j},\dots,g_m)$$
		Note that since $f \in \L^{\infty}(({}^*G)^m,\V)^{\sim {}^*G}$, $\pi(g)f(g_0,\dots,\hat{g_m},\dots,g_m) - f(gg_0,\dots,gg_m) \in \V_{inf}$, thus implying the conclusion. 
	\end{proof}
	Since the maps $\tilde{d}^m:\tilde{\L}^{\infty}(({}^*G)^m,\V) \to \tilde{\L}^{\infty}(({}^*G)^{m+1},\V)$ are ${}^*G$-equivariant for every $m \geq 1$, with $\tilde{d}^m \circ \tilde{d}^{m-1}=0$, we have an asymptotic ${}^*G$-cochain complex
	$$\begin{tikzcd}
	0 \arrow[r,"d^0"] & \L^{\infty}({}^*G,\V)  \arrow[r, "d^{1}"]&  \L^{\infty}(({}^*G)^2,\V) \arrow[r, "d^2"] & \L^{\infty}(({}^*G)^3,\V) \arrow[r, "d^3"]  & \L^{\infty}(({}^*G)^4,\V) \arrow[r, "d^4"] & \dots
	\end{tikzcd}$$
	and the induced ${}^*G$-cochain complex
	$$\begin{tikzcd}
	0 \arrow[r, "\tilde{d}^{0}"] &  \tilde{\L}^{\infty}({}^*G,\V)^{{}^*G} \arrow[r, "\tilde{d}^1"] & \tilde{\L}^{\infty}(({}^*G)^2,\V)^{{}^*G} \arrow[r, "\tilde{d}^2"]  & \tilde{\L}^{\infty}(({}^*G)^3,\V)^{{}^*G} \arrow[r, "\tilde{d}^3"] & \dots
	\end{tikzcd}$$
	\begin{definition}\label{asympt}
		The \textbf{asymptotic cohomology group of $G$ with coefficients in a dual asymptotic Banach ${}^*G$-module $\V$}, denoted $\Ha^{\bullet}(G,\V)$, is defined to be the asymptotic cohomology of the asymptotic ${}^*G$-cochain complex
		$$\begin{tikzcd}
		0 \arrow[r,"d^0"] & \L^{\infty}({}^*G,\V)  \arrow[r, "d^{1}"]&  \L^{\infty}(({}^*G)^2,\V) \arrow[r, "d^2"] & \L^{\infty}(({}^*G)^3,\V) \arrow[r, "d^3"]  & \L^{\infty}(({}^*G)^4,\V) \arrow[r, "d^4"] & \dots
		\end{tikzcd}$$
	\end{definition}
	Let us illustrate the above definitions in the case of a discrete group $\Gamma$ and the coefficients being the asymptotic Banach ${}^*\Gamma$-module $(\pi_{\Gamma},\W)$ where $\W=\prod_{\U}\mathfrak{u}_{k_n}$ and $\pi_{\Gamma}=\pi_{\psi}$ is as defined in (\cref{rho}), and relate the construction to \cref{mainthm2}. In this case, the asymptotic bounded cohomology $\Ha^{\bullet}(\Gamma,\W)$ is the cohomology of the complex given by
	$$\begin{tikzcd}
	0 \arrow[r, "\tilde{d}^{0}"] &  \tilde{\L}^{\infty}({}^*\Gamma,\W)^{{}^*\Gamma} \arrow[r, "\tilde{d}^1"] & \tilde{\L}^{\infty}(({}^*\Gamma)^2,\W)^{{}^*\Gamma} \arrow[r, "\tilde{d}^2"]  & \tilde{\L}^{\infty}(({}^*\Gamma)^3,\W)^{{}^*\Gamma} \arrow[r, "\tilde{d}^3"] & \dots
	\end{tikzcd}$$
	We shall now construct the same cohomology in another equivalent way which relates to \cref{mainthm2}. For $m \geq 0$, define an internal map
	$$\delta^m: \L^{\infty}(({}^*\Gamma)^m,\W) \to \L^{\infty}(({}^*\Gamma)^{m+1},\W)$$
	$$\delta^m(f)(g_1,\dots,g_{m+1}) \coloneqq \pi_{\Gamma}(g_1)f(g_2,\dots,g_{m+1}) + \sum_{j=1}^{m}(-1)^j f(g_1,\dots,g_jg_{j+1},\dots,g_{m+1}) + (-1)^{m+1}f(g_1,\dots,g_m)$$
	Again, $\delta^m$ restricts to a map from $\L^{\infty}_{b}(({}^*\Gamma)^m,\W)$ to $\L^{\infty}_{b}(({}^*\Gamma)^{m+1},\W)$ (which we shall continue to call $\delta^m$), which maps $\L^{\infty}_{inf}(({}^*\Gamma)^m,\W)$ to $\L^{\infty}_{inf}(({}^*\Gamma)^{m+1},\W)$. Since $\pi_{\Gamma}$ induces an asymptotic action of ${}^*\Gamma$ on $\W$, the induced map
	$$\tilde{\delta}^m: \tilde{\L}^{\infty}(({}^*\Gamma)^m,\W) \to \tilde{\L}^{\infty}(({}^*\Gamma)^{m+1},\W)$$
	$$\tilde{\delta}^m(\tilde{f})(g_1,\dots,g_{m+1}) = g_1\tilde{f}(g_2,\dots,g_{m+1}) + \sum_{j=1}^{m}(-1)^j \tilde{f}(g_1,\dots,g_jg_{j+1},\dots,g_{m+1}) + (-1)^{m+1}\tilde{f}(g_1,\dots,g_m)$$
	is exactly the coboundary map on $\tilde{\ell}(({}^*\Gamma)^m,\W)$. 
	Essentially, we now work with the ${}^*\Gamma$-complex $\mathcal{C}^{\bullet}$
	$$\begin{tikzcd}
	0 \arrow[r] & \tilde{\W}  \arrow[r, "\tilde{\delta}^0"]&  \tilde{\L}^{\infty}({}^*\Gamma,\W) \arrow[r, "\tilde{\delta}^1"] & \tilde{\L}^{\infty}(({}^*\Gamma)^2,\W) \arrow[r, "\tilde{\delta}^2"]  & \tilde{\L}^{\infty}(({}^*\Gamma)^3,\W) \arrow[r, "\tilde{\delta}^3"] & \dots
	\end{tikzcd}$$
	Since $\tilde{\delta}^{m}\cdot \tilde{\delta}^{m-1}=0$, for $m \geq 1$, denote the $m$-th cohomology group of this complex by $H^m(\mathcal{C}^{\bullet})$. In this notation, it is immediate that \cref{mainthm2} can be restated as
	\begin{theorem}
		Suppose $H^2(\mathcal{C}^{\bullet})=0$, then $\Gamma$ is uniformly $\mathfrak{U}$-stable with a linear estimate.
	\end{theorem}
	We now conclude this subsection by showing that $H^{\bullet}(\mathcal{C}^{\bullet}) \cong \Ha^{\bullet}(\Gamma,\W)$. 
	\begin{theorem}\label{inhomo1}
		For every $m \geq 1$, $H^{m}(\mathcal{C}^{\bullet}) \cong \Ha^m(\Gamma,\W)$. In particular, suppose $\Ha^2(\Gamma,\W)=0$, then $\Gamma$ is uniformly $\mathfrak{U}$-stable with a linear estimate.
	\end{theorem}
	\begin{proof}
		Consider the internal map 
		$$h^m:\L^{\infty}(({}^*\Gamma)^{m+1},\W) \to \L^{\infty}(({}^*\Gamma)^m,\W)$$
		$$h^mf(g_1,\dots,g_m) \coloneqq f(1,g_1,g_1g_2,\dots,g_1g_2\cdots g_m)$$
		This is simply a reparametrization, and its restriction to $\L^{\infty}_{b}(({}^*\Gamma)^m,\W)^{\sim {}^*\Gamma}$ induces an isomorphism
		$$\tilde{h}^m:\tilde{\L}^{\infty}(({}^*\Gamma)^{m+1},\W)^{{}^*\Gamma} \to \tilde{\L}^{\infty}(({}^*\Gamma)^{m},\W)$$
		such that the homogenous coboundary map $\tilde{d}^m$ translates to the coboundary map $\tilde{\delta}^m$ thus making $H_{a}^m(\Gamma,\W)$ canonically isomorphic to $H^m(\mathcal{C}^{\bullet})$. 
	\end{proof}
	\begin{remark}\label{inhomo}
		This complex $\mathcal{C}^{\bullet}$ can be thought of as the \emph{bar resolution} in our context, where we work with an \emph{inhomogenous} differential map, and the proof of \cref{inhomo1} goes through to show that in general, $\Ha^{\bullet}(G,\V)$ can be computed using an inhomogenous cochain complex just as with $\Gamma$. We shall work with this bar resolution for the rest of this subsection.
	\end{remark}
	We conclude this subsection with some results on $\Ha(G,\V)$ when the map $\pi:{}^*G \times \V \to \V$ is an internal trivial action of ${}^*G$ on $\V$ (that is, for every $g \in {}^*G$ and every $v \in \V$, $\pi(g)v = v$). Such an asymptotic Banach ${}^*G$-module shall be referred to as a trivial Banach ${}^*G$-module.\\
	We first recall the following facts about the \emph{vanishing moduli} of $G$ which are implicit in \cite{matsu} and clarified further in \cite[Definition 2.5]{nariman} and \cite[Lemma 4.12]{binate}:
	\begin{theorem}[\cite{matsu}\cite{nariman}\cite{binate}]\label{moduli}
		Let $V$ be trivial dual Banach $G$-module.
		\begin{itemize}
			\item Suppose $\Hb^2(G,V)=0$. Then there exists a constant $C_1>0$ such that for every (inhomogenous) $2$-cocycle $\alpha' \in \Linf_{w*}(G^2,V)$, there exists an (inhomogenous) $1$-cochain $\beta \in \Linf_{w*}(G,V)$ such that $\alpha'=\delta^1\beta$ and $\|\beta\|\leq C_1\|\alpha\|$. 
			\item Suppose $\Hb^3(G,V)$ is Hausdorff. Then there exists a constant $C_2>0$ such that for every (inhomogenous) $2$-cochain $\alpha \in \Linf_{w*}(G^2,V)$, there exists an (inhomogenous) $2$-cocycle $\alpha' \in \Linf_{w*}(G^2,V)$ such that $\|\alpha-\alpha'\|\leq C_2 \|\delta^2\alpha\|$.
		\end{itemize}
	\end{theorem}
	We can now use \cref{moduli} for $V=\R$ to show that:
	\begin{proposition}\label{trivialR}
		Suppose $\Hb^2(G,\R)=0$ and $\Hb^3(G,\R)$ is Hausdorff. Then $\Ha^2(G,{}^*\R)=0$
	\end{proposition}
	\begin{proof}
		Let $\alpha=\{\alpha_n\}_{\U} \in \L^{\infty}_b(({}^*G)^2,{}^*\R)$ with $\delta^2\alpha \in \L^{\infty}_{inf}(({}^*G)^2,{}^*\R)$. From \cref{moduli}, there exist constants $C_1,C_2>0$ such that for $n \in \U$, there exists $\beta_n \in \Linf(G,\R)$ such that $\|\alpha_n-\delta^1\beta_n\|\leq C_2\|\delta^2\alpha_n\|$ and $\|\beta_n\|\leq C_1\|\delta^1\beta_n\|$. Setting $\beta = \{\beta_n\}_{\U} \in \L^{\infty}_{b}({}^*G,{}^*\R)$, we see that $\alpha-\delta^1\beta \in \L^{\infty}_{inf}(({}^*G)^2,{}^*\R)$, to conclude that $\Ha^2(G,{}^*\R)=0$
	\end{proof}
	To extend this lemma to show vanishing of $\Ha^2(G,\V)$ for more general trivial Banach ${}^*G$-modules $\V$, we would need the constants $C_1$ and $C_2$ in \cref{moduli} to work uniformly across all trivial Banach $G$-modules. 
	\begin{lemma}\label{moduli2}
		Suppose $\Hb^2(G,V)=0$ and $\Hb^3(G,V)$ is Hausdorff for every trivial dual Banach $G$-module $V$.
		\begin{itemize}
			\item There exists a constant $C_1>0$ such that for every trivial dual Banach $G$-module $W$ and a (inhomogenous) $2$-cocycle $\alpha' \in \Linf_{w*}(G^2,W)$, there exists an (inhomogenous) $1$-cochain $\beta \in \Linf_{w*}(G,W)$ such that $\alpha'=\delta^1\beta$ and $\|\beta\|\leq C_1\|\alpha\|$. 
			\item There exists a constant $C_2>0$ such that for any trivial for every trivial dual Banach $G$-module $W$ and (inhomogenous) $2$-cochain $\alpha \in \Linf_{w*}(G^2,W)$, there exists an (inhomogenous) $2$-cocycle $\alpha' \in \Linf_{w*}(G^2,W)$ such that $\|\alpha-\alpha'\|\leq C_2 \|\delta^2\alpha\|$.
		\end{itemize}
	\end{lemma}
	\begin{proof}
		Suppose, for the sake of contradiction, that there exists a sequence $\{W_n\}_{n \in \N}$ of dual Banach spaces (all with trivial actions of $G$), and $2$-cocyles $\{\alpha_n \in \Linf_{w*}(G^2,W_n)\}_{n \in \N}$ with $\|\alpha_n\|=1$ for every $n \geq 1$, such that for every sequence $\{\beta_n \in \Linf_{w*}(G,W_n)\}_{n \in \N}$ with $\delta^1\beta_n=\alpha_n$ for every $n \geq 1$, we have $\|\beta_n\|\geq n$. Consider the $\ell^{\infty}$ direct sum $W \coloneq \oplus_n W_n$ (which is a dual Banach space) and the $2$-cocycle $\alpha \coloneq \oplus_n \alpha_n$ with $\|\alpha\|=1$. Then since $\Hb^2(G,W)=0$, that there exists $\beta \in \Linf_{w*}(G,W)$ with $\delta^1\beta = \alpha$, and let $\|\beta\|= C_1$. Note that the projections of $\beta$ on $W_n$ for $n > C_1$ would give us $\beta_n \in \Linf_{w*}(G,W_n)$ with $\delta^1\beta_n=\alpha_n$ and $\|\beta_n\| \leq C_1$, contradicting our assumption. The proof of the second item is similar.
	\end{proof}
	We now use the universality of the constants $C_1$ and $C_2$ in \cref{moduli2} to show the vanishing of $\Ha^2(G,\V)$ for a trivial Banach ${}^*G$-module, just like in \cref{trivialR}.
	\begin{proposition}\label{trivialRV2}
		Suppose $\Hb^2(G,V)=0$ and $\Hb^3(G,V)$ is Hausdorff for every trivial dual Banach $G$-module $V$. Then $\Ha^2(G,\V)=0$ for every trivial asymptotic Banach ${}^*G$-module $\V$. 
	\end{proposition}
	While the assumption that $\Hb^2(G,V)=0$ and $\Hb^3(G,V)$ is Hausdorff for \emph{every} trivial dual Banach $G$-module $V$, a priori, seems stronger than the assumption that $\Hb^2(G,\R)=0$ and $\Hb^3(G,\R)$ is Hausdorff, we now show that the former is actually implied by the latter. Recall that a Banach space $V$ is said to be \emph{injective} if for any Banach space embedding $V \subset W$, $V$ has a complement in $W$.
	\begin{proposition}\label{complement}
		Suppose $\Hb^2(G,\R)=0$ and $\Hb^3(G,\R)$ is Hausdorff. Then the image $\delta^2\Linf(G^2)$ in $\Linf(G^3)$ is injective.
	\end{proposition}
	\begin{proof}
		Note that the image $\delta^1 \Linf(G)$ in $\Linf(G^2)$ is closed, since $\Hb^2(G,\R)=0$. In particular, this image is isomorphic to $\Linf(G)/\R$, which is injective since $\Linf(G)$ and $\R$ are injective. Next, since $\Hb^3(G,\R)$ is Hausdorff, the image $\delta^2\Linf(G^2)$ is closed, and isomorphic to $L^2(G)/\delta^1\Linf(G)$, which is injective since $\Linf(G^2)$ and $\delta^1\Linf(G)$ are injective. 
	\end{proof}
	\begin{lemma}\label{trivialRV}
		Suppose $\Hb^2(G,\R)=0$ and $\Hb^3(G,\R)$ is Hausdorff. Then for any trivial dual Banach $G$-module $V$, $\Hb^2(G,V)=0$ and $\Hb^3(G,V)$ is Hausdorff.
	\end{lemma}
	\begin{proof}
		Note (\cref{dunford} and \cref{dunford2}) that $\Linf_{w*}(G^j,V) \cong L\left( V^{\flat}, \Linf(G^j) \right)$. In this identification, the differential is just given by applying the scalar differential
		$\delta$ to the image $\Linf(G^j)$. It follows that the complementing map given by \cref{complement} provides a complementing map for the image $\delta^2\Linf_{w*}(G^2,V)$ of $\Linf_{w*}(G^2,V)$ in $\Linf_{w*}(G^3,V)$. In particular this image is closed and hence
		$\Hb^3(G,V)$ is Hausdorff. Similarly, we can show that $\Hb^2(G,V)=0$ as well.
	\end{proof}
	Combining \cref{trivialRV} and \cref{trivialRV2}, we conclude that
	\begin{corollary}\label{trivialRV3}
		Suppose $\Hb^2(G,\R)=0$ and $\Hb^3(G,\R)$ is Hausdorff. Then $\Ha^2(G,\V)=0$ for any trivial dual Banach ${}^*G$-module $\V$.
	\end{corollary}
	This property of vanishing $\Hb^2(G,\R)$ and Hausdorffness of $\Hb^3(G,\R)$ will be very useful to us in \cref{ssec-semisimple}, and in \cref{ssec-2half} we shall study this property (called the \say{2\textonehalf-property}) in more detail.
	
	\subsection{Amenable Actions and Cohomology of Subgroups}\label{ssec-subg}
	Recall that we had presented the asymptotic cohomology of $G$ with coefficients in the dual asymptotic Banach ${}^*G$-module $\V$ as asymptotic cohomology of the complex
	$$\begin{tikzcd}
	0 \arrow[r,"d^0"] & \L^{\infty}({}^*G,\V)  \arrow[r, "d^{1}"]&  \L^{\infty}(({}^*G)^2,\V) \arrow[r, "d^2"] & \L^{\infty}(({}^*G)^3,\V) \arrow[r, "d^3"]  & \L^{\infty}(({}^*G)^4,\V) \arrow[r, "d^4"] & \dots
	\end{tikzcd}$$
	In this subsection, we shall relate the functorial descriptions in \S\ref{ssec-basicG} with \cref{asympt} to obtain other complexes that can also be used to compute the same cohomology.\\
	The first step is to interpret $\Ha^{\bullet}(G,\V)$ in a way that fits with \cref{bilinear}. For this, we recall the following classical fact that follows from the Dunford-Pettis theorem: for a (standard) dual Banach space $E$,
	\begin{equation}\label{dunford}
	B(L^1(G^m) \times E) \cong L\left(L^1(G^m),E^{\sharp}\right) \cong L^{\infty}_{w*}(G^m,E^{\sharp})
	\end{equation}
	In fact, this goes through even when we consider the analogous asymptotic Banach ${}^*G$-modules. Denoting the ultrapower ${}^*L^1(G^m)$ by $\L^1(G^m)$, we note that $\L^{\infty}(({}^*G)^m)$ is a dual asymptotic ${}^*G$-module with predual $\L^1(({}^*G)^m)$. In fact, it is more than just an asymptotic ${}^*G$-representation: ${}^*G$ actually has a (true) internal action on $\L^1(({}^*G)^m)$ and its dual $\L^{\infty}(({}^*G)^m)$.\\
	This allows us to construct $\Ha^{\bullet}(G,\V)$ starting from the $G$-cochain complex
	$$\begin{tikzcd}
	0 \arrow[r,"d^0"] & L^{\infty}(G)  \arrow[r, "d^{1}"]&  L^{\infty}(G^2) \arrow[r, "d^2"] & L^{\infty}(G^3) \arrow[r, "d^3"]  & L^{\infty}(G^4) \arrow[r, "d^4"] & \dots
	\end{tikzcd}$$
	extending it internally to the asymptotic ${}^*G$-cochain complex (which, in this case, turns out to be a (true) internal cochain complex)
	$$\begin{tikzcd}
	0 \arrow[r,"d^0"] & \L^{\infty}({}^*G)  \arrow[r, "d^{1}"]&  \L^{\infty}(({}^*G)^2) \arrow[r, "d^2"] & \L^{\infty}(({}^*G)^3) \arrow[r, "d^3"]  & \L^{\infty}(({}^*G)^4) \arrow[r, "d^4"] & \dots
	\end{tikzcd}$$
	and finally using the covariant functor $d \mapsto d_*$ for the coefficient module $\V^{\flat}$ (the predual of $\V$) to get
	$$\begin{tikzcd}
	0 \arrow[r,"d^0"] & \mathcal{B}(\L^1({}^*G) \times \V^{\flat})  \arrow[r, "d^{1}"]&  \mathcal{B}(\L^1(({}^*G)^2) \times \V^{\flat}) \arrow[r, "d^2"] & \mathcal{B}(\L^1(({}^*G)^3) \times \V^{\flat}) \arrow[r, "d^3"]  & \dots
	\end{tikzcd}$$
	which, in turn, is seen to be the same as
	$$\begin{tikzcd}
	0 \arrow[r,"d^0"] & \L^{\infty}({}^*G,\V)  \arrow[r, "d^{1}"]&  \L^{\infty}(({}^*G)^2,\V) \arrow[r, "d^2"] & \L^{\infty}(({}^*G)^3,\V) \arrow[r, "d^3"]  & \L^{\infty}(({}^*G)^4,\V) \arrow[r, "d^4"] & \dots
	\end{tikzcd}$$
	The advantage of this reformulation of $\Ha^{\bullet}(G,\V)$ is that we can use \cref{homotopy2} to construct other asymptotic ${}^*G$-cochain complexes that compute the same cohomology $\Ha^{\bullet}(G,\V)$. In view of this approach, we review some definitions and facts from \cite{bookMonod}:
	\begin{definition}
		Let $S$ be a regular $G$-space. A \textbf{conditional expectation} $\mathfrak{m}:L^{\infty}(G \times S) \to L^{\infty}(S)$ is a measurable norm one linear map such that
		\begin{itemize}
			\item $\mathfrak{m}(1_{G \times S}) = 1_{S}$
			\item For every $f \in L^{\infty}(G \times S)$ and every measurable set $A \subset S$, $\mathfrak{m}(f \cdot 1_{G \times A}) = \mathfrak{m}(f) \cdot 1_A$.
		\end{itemize}
		The $G$-action on $S$ is said to be \textbf{Zimmer amenable} if there exists a $G$-equivariant conditional expectation $\mathfrak{m}:L^{\infty}(G \times S) \to L^{\infty}(S)$.
	\end{definition}
	What we shall use is the following consequence of Zimmer amenability:
	\begin{proposition}
		Let $S$ be a regular $G$-space with a Zimmer-amenable action of $G$. Then there exists a $G$-homotopy equivalence between the $G$-cochain complexes
		$$\begin{tikzcd}
		0 \arrow[r,"d^0"] & L^{\infty}(G)  \arrow[r, "d^{1}"]&  L^{\infty}(G^2) \arrow[r, "d^2"] & L^{\infty}(G^3) \arrow[r, "d^3"]  & L^{\infty}(G^4) \arrow[r, "d^4"] & \dots
		\end{tikzcd}$$
		and
		$$\begin{tikzcd}
		0 \arrow[r,"d^0"] & L^{\infty}(S)  \arrow[r, "d^{1}"]&  L^{\infty}(S^2) \arrow[r, "d^2"] & L^{\infty}(S^3) \arrow[r, "d^3"]  & L^{\infty}(S^4) \arrow[r, "d^4"] & \dots
		\end{tikzcd}$$
	\end{proposition}
	Extending this internally, if $S$ a regular $G$-space with a Zimmer-amenable action of $G$, then this gives us an asymptotic ${}^*G$-homotopy equivalences between the asymptotic ${}^*G$-cochain complexes
	$$\begin{tikzcd}
	0 \arrow[r,"d^0"] & \L^{\infty}({}^*G)  \arrow[r, "d^{1}"]&  \L^{\infty}(({}^*G)^2) \arrow[r, "d^2"] & \L^{\infty}(({}^*G)^3) \arrow[r, "d^3"]  & \L^{\infty}(({}^*G)^4) \arrow[r, "d^4"] & \dots
	\end{tikzcd}$$
	and
	$$\begin{tikzcd}
	0 \arrow[r,"d^0"] & \L^{\infty}({}^*S)  \arrow[r, "d^{1}"]&  \L^{\infty}(({}^*S)^2) \arrow[r, "d^2"] & \L^{\infty}(({}^*S)^3) \arrow[r, "d^3"]  & \L^{\infty}(({}^*S)^4) \arrow[r, "d^4"] & \dots
	\end{tikzcd}$$
	Now, since $(L^1(S^m))^{\sharp} = L^{\infty}(S^m)$ and $B(L^1(S^m) \times E) \cong L\left(L^1(S^m),E^{\sharp}\right) \cong L^{\infty}_{w*}(S^m,E^{\sharp})$, applying \cref{homotopy2} with $\mathcal{X}^{\bullet} = \L^1(({}^*G)^{\bullet})$, $\mathcal{Y}^{\bullet}=\L^1(({}^*S)^{\bullet})$ and $\mathcal{E}=\V^{\flat}$, we get
	\begin{theorem}\label{zimmer}
		Let $S$ be a regular $G$-space with a Zimmer-amenable action of $G$. Then $\Ha^{\bullet}(G,\V)$ can be computed as the asymptotic cohomology of the asymptotic ${}^*G$-cochain complex
		$$\begin{tikzcd}
		0 \arrow[r,"d^0"] & \L^{\infty}({}^*S,\V)  \arrow[r, "d^{1}"]&  \L^{\infty}(({}^*S)^2,\V) \arrow[r, "d^2"] & \L^{\infty}(({}^*S)^3,\V) \arrow[r, "d^3"]  & \L^{\infty}(({}^*S)^4,\V) \arrow[r, "d^4"] & \dots
		\end{tikzcd}$$
	\end{theorem}
	We now state some observations that follow immediately from \cref{zimmer}, which we shall use later in \cref{Qinv1} in \S\ref{sec-mainproof}. Let $S$ and $T$ be two Zimmer-amenable regular $G$-spaces, so that by \cref{zimmer}, we have an asymptotic ${}^*G$-homotopy equivalence between the asymptotic ${}^*G$-cochain complexes $\L^{\infty}(({}^*S)^{\bullet},\V)$ and $\L^{\infty}(({}^*T)^{\bullet},\V)$ given by $k^{\bullet}:\L^{\infty}(({}^*S)^{\bullet},\V) \to \L^{\infty}(({}^*T)^{\bullet},\V)$ and $j^{\bullet}:\L^{\infty}(({}^*T)^{\bullet},\V) \to \L^{\infty}(({}^*S)^{\bullet},\V)$. 
	\begin{itemize}
		\item Let $\omega \in \L^{\infty}_b(({}^*T)^{m+1},\V)^{\sim {}^*G}$ be an asymptotic $m$-cocycle such that $Im(\omega) \subseteq \V_b^{\sim {}^*G}$. Then $Im(j^m\omega) \subseteq  \V_b^{\sim {}^*G}$. This follows from the construction of $j^m$ and \cref{prop_lev1}. 
		\item Let $\omega \in \L^{\infty}_b(({}^*S)^{m+1},\V)^{\sim {}^*G}$ be an asymptotic $m$-cocycle such that $Im(k^m\omega) \subseteq  \V_b^{\sim {}^*G}$. Then, setting $\omega_1 = j^m k^m \omega \in \L^{\infty}_b(({}^*S)^{m+1},\V)^{\sim {}^*G}$, note that $Im(\omega_1) \subseteq \V_b^{\sim {}^*G}$. Furthermore, since $k^{\bullet}$ and $j^{\bullet}$ are asymptotic ${}^*G$-homotopy equivalences, $\omega$ and $\omega_1$ are asymptotically cohomologous, that is, there exists $\alpha \in \L^{\infty}_b(({}^*S)^{m},\V)^{\sim {}^*G}$ such that
		\begin{equation}\label{invariants}
		\tilde{\omega}-\tilde{\omega}_1 = \tilde{d}^m\tilde{\alpha}
		\end{equation}
	\end{itemize}
	\cref{zimmer} has the following two immediate corollaries which we shall use in \S\ref{sec-induc} and \S\ref{sec-mainproof}. The first of these is an asymptotic analogue of the classical result that amenable groups have vanishing bounded cohomology, and follows from the fact that if $G$ is amenable, then the trivial space is Zimmer-amenable for $G$.
	\begin{corollary}\label{corr-amentriv}
		Let $G$ be an amenable group and $\V$ be a dual asymptotic Banach ${}^*G$-module. Then for every $n \geq 1$, $\Ha^n(G,\V)=0$.  
	\end{corollary}
	The second corollary uses the fact that for a lattice $\Gamma$ in a locally compact group $G$, $G$ is Zimmer-amenable as a regular $\Gamma$-space. This serves as the starting point for an induction procedure to go from $\Gamma$ to $G$ which we describe in \S\ref{sec-induc}.
	\begin{corollary}\label{corr-GammaG}
		Let $\Gamma \leq G$ be a lattice in a locally compact group $G$, and let $\W$ be a dual asymptotic Banach ${}^*\Gamma$-module. Then $\Ha^{\bullet}(\Gamma,\W)$ can be computed as the asymptotic cohomology of the asymptotic ${}^*\Gamma$-cochain complex
		$$\begin{tikzcd}
		0 \arrow[r,"d^0"] & \L^{\infty}({}^*G,\W)  \arrow[r, "d^{1}"]&  \L^{\infty}(({}^*G)^2,\W) \arrow[r, "d^2"] & \L^{\infty}(({}^*G)^3,\W) \arrow[r, "d^3"]  & \L^{\infty}(({}^*G)^4,\W) \arrow[r, "d^4"] & \dots
		\end{tikzcd}$$
	\end{corollary}
	The next corollary uses the fact that for a closed subgroup $Q\leq G$ (and in particular, when $Q=G$), and a closed amenable subgroup $P\leq G$, the space $G/P$ is a regular $Q$-space that is Zimmer-amenable for the $Q$-action.
	\begin{corollary}\label{Qcoho}
		Let $P \leq G$ be a closed amenable subgroup of $G$, $Q$ be a closed subgroup of $G$, and $\V$ be a dual asymptotic Banach ${}^*G$-module. Then $\Ha^{\bullet}(Q,\V)$ can be computed as the asymptotic cohomology of the asymptotic ${}^*Q$-cochain complex
		$$\begin{tikzcd}
		0 \arrow[r,"d^0"] & \L^{\infty}(({}^*(G/P)),\V)  \arrow[r, "d^{1}"]&  \L^{\infty}(({}^*(G/P))^2,\V) \arrow[r, "d^2"] & \L^{\infty}(({}^*(G/P))^3,\V) \arrow[r, "d^3"]  &\dots
		\end{tikzcd}$$
	\end{corollary}
	The second corollary uses the fact that for a closed subgroup $Q\leq G$, the space $G/P$ is a regular $Q$-space that is Zimmer-amenable for the $Q$-action.

	\section{The Induction Module}\label{sec-induc}
	In the previous section, we noted (in \cref{corr-GammaG}) that for a lattice $\Gamma$ in a Lie group $G$, the cohomology $\Ha^{\bullet}(\Gamma,\W)$ could be computed as the asymptotic cohomology of the asymptotic ${}^*\Gamma$-cochain complex
	$$\begin{tikzcd}
	0 \arrow[r,"d^0"] & \L^{\infty}({}^*G,\V)  \arrow[r, "d^{1}"]&  \L^{\infty}(({}^*G)^2,\V) \arrow[r, "d^2"] & \L^{\infty}(({}^*G)^3,\V) \arrow[r, "d^3"]  & \L^{\infty}(({}^*G)^4,\V) \arrow[r, "d^4"] & \dots
	\end{tikzcd}$$
	We shall now apply an induction procedure to go from ${}^*\Gamma$-equivariance to  ${}^*G$-equivariance, and shall construct an asymptotic Banach ${}^*G$-module $\V$ such that $\Ha^{\bullet}(\Gamma,\W) \cong \Ha^{\bullet}(G,\V)$.\\
	In \S\ref{ssec-trueAction}, we begin by studying useful properties of an intermediary structure $\L^{\infty}_b({}^*G,\W)^{\sim {}^*\Gamma}$ that shall arise in the induction procedure worked out in \S\ref{ssec-shapiro}. This structure is, upto infinitesimals, equal to the induced module $\L^{\infty}({}^*D,\W)$ that we shall use, but has the additional  useful feature of being equipped with a (true) internal action of ${}^*G$. A $1$-cohomology argument is used to pass between asymptotically equivariant maps in $\L^{\infty}({}^*D,\W)$ and (truly) equivariant maps in $\L^{\infty}_b({}^*G,\W)^{\sim {}^*\Gamma}$, which is a result we shall often use, especially in \S\ref{ssec-invariants}. Finally, the induction procedure is described in \S\ref{ssec-shapiro}.
	
	\subsection{The ${}^*G$-action on $\L^{\infty}_b({}^*G,\W)^{\sim {}^*\Gamma}$}\label{ssec-trueAction}
	Recall that the internal Banach space $\W=\prod_{\U}\mathfrak{u}(k_n)$ came with an internal map $\pi_{\Gamma}:{}^*\Gamma \times \W \to \W$ such that this map induced an action of ${}^*\Gamma$ on $\tilde{\W}=\W_b/\W_{inf}$. Let $S$ be an Zimmer-amenable regular $G$-space. For $m \geq 0$, consider the internal Banach space
	$$\L^{\infty}(({}^*S)^m,\W)$$
	equipped with the following internal ${}^*G$-action: for $g \in {}^*G$, $f \in \L^{\infty}(({}^*S)^m,\W)$,
	$$(g \cdot f)(x_1,\dots,x_m) = f(g^{-1}x_1,\dots,g^{-1}x_m)$$
	for $x_1,\dots,x_m \in {}^*S$. Clearly $\L^{\infty}_b(({}^*S)^m,\W)$ is invariant with respect to this ${}^*G$-action, and induces a ${}^*G$-action on $\tilde{L}^{\infty}(({}^*S)^m,\W)$. \\
	Consider the subsets $\L^{\infty}_b(({}^*S)^m,\W)^{{}^*G}$ of bounded  ${}^*G$-fixed points of this action, and the subset $\L^{\infty}_b(({}^*S)^m,\W)^{\sim {}^*G}$ of bounded asymptotically ${}^*G$-fixed points. Clearly, 
	$$\L^{\infty}_b(({}^*S)^m,\W)^{{}^*G} + \L^{\infty}_{inf}(({}^*S)^m,\W) \subseteq \L^{\infty}_b(({}^*S)^m,\W)^{\sim {}^*G}$$
	We shall now show that the containment goes through in the other direction too. That is,
	\begin{lemma}\label{correction}
		For every $f \in \L^{\infty}_b(({}^*S)^m,\W)^{\sim {}^*G}$, there exists $\beta \in \L^{\infty}_{inf}(({}^*S)^m,\W)$ such that $f-\beta \in \L^{\infty}_b(({}^*S)^m,\W)^{{}^*G}$. Moreover, the map $f \mapsto \beta$ is induced from an internal map from $\L^{\infty}(({}^*S)^m,\W)$ to itself.
	\end{lemma}
	\begin{proof}
		Consider the internal map $\alpha: {}^*G \to  \L^{\infty}(({}^*S)^m,\W)$ defined as $\alpha(g) \coloneq g\cdot f - f$. Note that since $f \in \L^{\infty}_b(({}^*S)^m,\W)^{\sim {}^*G}$, $Im(\alpha) \subseteq \L^{\infty}_{inf}(({}^*S)^m,\W)$ (hence $\|\alpha\| \in {}^*\R_{inf}$). Let $\alpha=\{\alpha_n\}_{\U}$ where for every $n \in \N$, $\alpha_n:G \to L^{\infty}_{w*}(S^m,\mathfrak{u}(k_n))$. Observe that each such $\alpha_n$ is a bounded function that satisfies the (inhomogenous) $1$-cocycle condition. That is, $\alpha_n \in \Hb^1\left( G, L^{\infty}_{w*}(S^m,\mathfrak{u}(k_n)) \right)$. From \cite{bookMonod}, we know that $L^{\infty}_{w*}(S^m,\mathfrak{u}(k_n))$ is a relatively injective Banach $G$-module, and hence $\Hb^1\left( G, L^{\infty}_{w*}(S^m,\mathfrak{u}(k_n)) \right)=0$. In fact, there exists a constant $C$ (independent of $n$) such that for every $n \in \N$, there exists $\beta_n \in L^{\infty}_{w*}(S^m,\mathfrak{u}(k_n))$ with $\alpha_n(g)=g\cdot \beta_n - \beta_n$, with $\|\beta_n\|\leq C\|\alpha_n\|$. Set $\beta=\{\beta_n\}_{\U}$ so that for $g \in {}^*G$, $\alpha(g)=g \cdot \beta - \beta$ implying that $g \cdot (f-\beta)=f-\beta$. Since $\|\beta\|\leq C\|\alpha\|$, $\beta \in \L^{\infty}_{inf}(({}^*S)^m,\W)$. Note that the correspondence $f \mapsto \beta$ is internal by construction.
	\end{proof}
	A special case of particular interest to us is the internal Banach space $\L^{\infty}({}^*G,\W)$. This space comes equipped with an internal ${}^*G$-action and an asymptotic ${}^*\Gamma$-action: 
	\begin{itemize}
		\item For $g \in {}^*G$ and $f \in \L^{\infty}({}^*G,\W)$, define $(g \cdot f)(x)=f(xg)$ for $x \in {}^*G$. This makes $\L^{\infty}({}^*G,\W)$ an internal ${}^*G$-representation. 
		\item Consider the internal map $\pi'_{\Gamma}:{}^*\Gamma \times \L^{\infty}({}^*G,\W) \to \L^{\infty}({}^*G,\W)$ defined as follows: for $\gamma \in {}^*\Gamma$ and $f \in \L^{\infty}({}^*G,\W)$, define 
		$$(\pi'_{\Gamma}(\gamma)f)(x)=\pi_{\Gamma}(\gamma)f(\gamma^{-1}x)$$
		where $x \in {}^*G$. This makes $\L^{\infty}({}^*G,\W)$ into an asymptotic Banach ${}^*\Gamma$-module.
	\end{itemize}
	\begin{remark}
		Note that ${}^*G$ acts internally on $\L^{\infty}({}^*G,\W)$, while ${}^*\Gamma$ does not act, through $\pi'_{\Gamma}$, on $\L^{\infty}({}^*G,\W)$, but only induces an action on the quotient $\tilde{\L}^{\infty}({}^*G,\W)$.
	\end{remark}
	Let us restrict to the subspace $\L^{\infty}_{b}({}^*G,\W)^{\sim {}^*\Gamma}$ comprising functions that are asymptotically ${}^*\Gamma$-equivariant. Note that $\L^{\infty}_{b}({}^*G,\W)^{\sim {}^*\Gamma}$ is \emph{not} an internal space.
	\begin{lemma}\label{inducG}
		The subset $\L^{\infty}_{b}({}^*G,\W)^{\sim {}^*\Gamma}$ is invariant under the internal action of ${}^*G$. 
	\end{lemma}
	Consider the subspaces $\left(\L^{\infty}_{b}(G,\W)^{\sim{}^*\Gamma}\right)^{{}^*G}$ and $\left(\L^{\infty}_{b}(G,\W)^{\sim{}^*\Gamma}\right)^{\sim {}^*G}$. Observe that the proof of \cref{correction} goes through when restricted to asymptotically ${}^*\Gamma$-equivariant elements, giving us:
	\begin{corollary}\label{VG}
		For $v \in \left( \L_b^{\infty}({}^*G,\W)^{\sim {}^*\Gamma}\right)^{\sim {}^*G}$, there exists $w \in  \left( \L_b^{\infty}({}^*G,\W)^{\sim {}^*\Gamma}\right)^{{}^*G}$ such that $v-w \in  \L_{inf}^{\infty}({}^*G,\W)$. Moreover, the map $v \mapsto w$ is induced from an internal map from $\L^{\infty}({}^*G,\W)$ to itself.
	\end{corollary}
	An element $f \in \left(\L^{\infty}_{b}(G,\W)^{\sim{}^*\Gamma}\right)^{{}^*G}$ satisfies the following two conditions:
	\begin{itemize}
		\item For every $g\in {}^*G$ and almost every $x \in {}^*G$, $f(xg)=f(x)$. In other words, $f$ is an essentially constant function, where the constant value is an element of $\W_b$.
		\item For $\gamma \in {}^*\Gamma$ and almost every $x \in {}^*G$, $f(\gamma x) - \gamma f(x) \in \W_{inf}$ (that is, the constant value of $f$ is an element of $\W_b)^{\sim {}^*\Gamma}$).
	\end{itemize}
	Thus, $f \in \left(\L^{\infty}_{b}(G,\W)^{\sim{}^*\Gamma}\right)^{{}^*G}$ can be represented as an element of $(\W_b)^{\sim {}^*\Gamma}$. 
	\begin{lemma}\label{Wb}
		The internal map $e:\W \to \L^{\infty}(G,\W)$ defined as $e(w)(g) \coloneq w$ for $w \in \W_b$ and $g \in {}^*G$ restricts to a bijection between $\W_b^{\sim {}^*\Gamma}$ and $\left(\L^{\infty}_{b}(G,\W)^{\sim{}^*\Gamma}\right)^{{}^*G}$. 
	\end{lemma}
	\cref{VG} and \cref{Wb} together imply that, upto infinitesimals, $\left(\L_b^{\infty}({}^*G,\W)^{\sim {}^*\Gamma}\right)^{\sim {}^*G}$ can be identified with $\W_b^{\sim {}^*\Gamma}$, and this bijection is (trivially) ${}^*G$-equivariant and induced from an internal map between $\L^{\infty}({}^*G,\W)$ and $\W$.\\
	We combine the results above to obtain a corollary that will be useful later in \S\ref{sec-mainproof}. 
	\begin{proposition}\label{almostGequi2}
		Let $Q$ be a closed subgroup of $G$. Let $f \in  \L^{\infty}_b\left( ({}^*Q)^m, \L^{\infty}({}^*G,\W)\right)^{\sim {}^*Q}$ be such that $Im(f) \subseteq (\L^{\infty}_b({}^*G,\W)^{\sim {}^*\Gamma})^{\sim {}^*G}$. Then there exists $\beta \in \L_{inf}^{\infty}\left( ({}^*Q)^m, \L^{\infty}({}^*G,\W) \right)$ such that $f-\beta$ is ${}^*Q$-fixed, and $Im(f-\beta) \subseteq \left( \L^{\infty}_b({}^*G,\W)^{\sim {}^*\Gamma} \right)^{{}^*G}=\W_b^{{}^*\Gamma}$. The map $f \mapsto \beta$ is induced from an internal map from $\L^{\infty}\left( ({}^*Q)^m, \L^{\infty}({}^*G,\W)\right)$ to itself.
	\end{proposition}
	\begin{proof}
		Since $Im(f) \subseteq (\L^{\infty}_b({}^*G,\W)^{\sim {}^*\Gamma})^{\sim {}^*G}$, \cref{VG} gives us $f' \in \L^{\infty}_b\left( ({}^*Q)^m, \L^{\infty}({}^*G,\W)\right)^{\sim {}^*Q}$ with $f-f' \in \L^{\infty}_{inf}\left( ({}^*Q)^m, \L^{\infty}({}^*G,\W)\right)$ and $Im(f') \subseteq (\L^{\infty}_b({}^*G,\W)^{\sim {}^*\Gamma})^{{}^*G} = \W_b^{\sim {}^*\Gamma}$. The conclusion then follows from applying \cref{correction} to $f'$. 
	\end{proof}
	Recall \cref{correction} where an element in $\L^{\infty}_b(({}^*S)^m,\W)^{\sim {}^*G}$ was shown to be corrected to get a truly ${}^*G$-fixed element in $\L^{\infty}_b(({}^*S)^m,\W)^{\sim {}^*G}$. We shall now see a similar result with coefficients being $\L^{\infty}({}^*G,\W)$ instead, which we shall use in \S\ref{ssec-shapiro}. For $m \geq 0$, consider the internal space 
	$$\L^{\infty}\left(({}^*G)^m,\L^{\infty}({}^*G,\W) \right)$$
	The internal action of ${}^*G$ on $\L^{\infty}({}^*G,\W)$ can be extended to the following internal action of ${}^*G$ on this space in the natural way: for $g,h \in {}^*G$, $F \in \L^{\infty}\left(({}^*G)^m,\L^{\infty}({}^*G,\W) \right)$,
	$$(g\cdot F)(g_1,g_2,\dots,g_m)(x)=F(g^{-1}g_1,\dots,g^{-1}g_m)(xg)$$
	For convenience, let us denote by
	$$\L^{\infty}\left( ({}^*G)^m, \L_b^{\infty}({}^*G,\W)^{\sim {}^*\Gamma} \right)$$
	the space of internal maps in $\L^{\infty}\left( ({}^*G)^m, \L_{in}^{\infty}({}^*G,\W) \right)$  whose image is contained in $\L_b^{\infty}({}^*G,\W)^{\sim {}^*\Gamma}$. From \cref{inducG}, this space is invariant under the internal action of ${}^*G$, so denote by
	$$\L^{\infty}\left( ({}^*G)^m, \L_b^{\infty}({}^*G,\W)^{\sim {}^*\Gamma} \right)^{{}^*G}$$
	the subspace of ${}^*G$-equivariant maps, and by 
	$$\L^{\infty}\left( ({}^*G)^m, \L_b^{\infty}({}^*G,\W)^{\sim {}^*\Gamma} \right)^{\sim {}^*G}$$
	the subspace of asymptotically ${}^*G$-equivariant maps. That is, $f \in \L^{\infty}\left( ({}^*G)^m, \L_b^{\infty}({}^*G,\W)^{\sim {}^*\Gamma} \right)^{\sim {}^*G}$ if for every $g \in {}^*G$, $g \cdot f - f \in \L_{inf}^{\infty}\left( ({}^*G)^m, \L^{\infty}({}^*G,\W) \right)$. \\
	Note that 
	$$\L^{\infty}\left( ({}^*G)^m, \L_b^{\infty}({}^*G,\W)^{\sim {}^*\Gamma} \right)^{{}^*G} + \L_{inf}^{\infty}\left( ({}^*G)^m, \L^{\infty}({}^*G,\W) \right) \subseteq \L^{\infty}\left( ({}^*G)^m, \L_b^{\infty}({}^*G,\W)^{\sim {}^*\Gamma} \right)^{\sim {}^*G}$$
	In other words, a perturbation of a ${}^*G$-equivariant function by an infinitesimal function is clearly an asymptotically ${}^*G$-equivariant function. We now show that the converse is true. That is, any asymptotically ${}^*G$-equivariant map is infinitesimally close to a ${}^*G$-equivariant map.\\
	\begin{proposition}\label{almostGequi}
		For $f \in \L^{\infty}\left( ({}^*G)^m, \L_b^{\infty}({}^*G,\W)^{\sim {}^*\Gamma} \right)^{\sim {}^*G}$, there exists $\beta \in \L_{inf}^{\infty}\left( ({}^*G)^m, \L^{\infty}({}^*G,\W) \right)$ such that $f-\beta \in \L^{\infty}\left( ({}^*G)^m, \L_b^{\infty}({}^*G,\W)^{\sim {}^*\Gamma} \right)^{{}^*G}$ (that is, $f-\beta$ is (truly) ${}^*G$-equivariant). The map $f \mapsto \beta$ is induced from an internal map from $\L^{\infty}\left( ({}^*G)^m, \L^{\infty}({}^*G,\W) \right)$ to itself.
	\end{proposition}
	\begin{proof}
		The argument is exactly as in the proof of \cref{correction}, except that now we use the fact that $L^{\infty}_{w*}(G^m, L^{\infty}_{w*}(G,\mathfrak{u}(k_n)))$ is relatively injective as a $G$-module. 
	\end{proof}
	In conclusion,
	\begin{equation*}
	\L^{\infty}\left( ({}^*G)^m, \L_b^{\infty}({}^*G,\W)^{\sim {}^*\Gamma} \right)^{{}^*G} + \L_{inf}^{\infty}\left( ({}^*G)^m, \L^{\infty}({}^*G,\W) \right) = \L^{\infty}\left( ({}^*G)^m, \L_b^{\infty}({}^*G,\W)^{\sim {}^*\Gamma} \right)^{\sim {}^*G}
	\end{equation*}

	\subsection{$\L^{\infty}({}^*D,\W)$ and the Eckmann-Shapiro Induction}\label{ssec-shapiro}
	Consider the space $\L_b^{\infty}({}^*G,\W)^{\sim {}^*\Gamma}$. We shall now show that, upto infinitesimals, this space can be identified with the internal Banach space $\V \coloneq \L_b^{\infty}({}^*D,\W)$ where $D$ is a Borelian left fundamental domain of $\Gamma$ in $G$. This will then be used to serve as the coefficients to define the asymptotic cohomology $\Ha^{\bullet}(G,\V)$ of $G$.\\
	Consider the internal map
	$$\theta: \L^{\infty}({}^*G,\W) \to \L^{\infty}({}^*D,\W)$$
	given by restriction of a function to ${}^*D$. That is, for $f=\{f_n\}_{\U} \in \L^{\infty}({}^*G,\W)$,
	\begin{equation}\label{theta}
	\theta f= \{f_n \vert_D\}_{\U}
	\end{equation}
	In the other direction, consider the internal map
	$$\zeta: \L^{\infty}({}^*D,\W) \to \L^{\infty}({}^*G,\W)$$
	\begin{equation}\label{zeta}
	\zeta f(g) \coloneqq \pi_{\Gamma}(\gamma)f(z)
	\end{equation}
	where $g=\gamma z$ for $\gamma \in {}^*\Gamma$ and $z \in D$. Observe that $\theta \cdot \zeta$ is the identity on $\L^{\infty}({}^*D,\W)$. As for $\zeta \cdot \theta$,
	\begin{lemma}
		For any $f \in \L_b^{\infty}({}^*G,\W)^{\sim {}^*\Gamma}$, $(\zeta \cdot \theta)(f) - f \in \L_{inf}^{\infty}({}^*G,\W)$.
	\end{lemma}
	\begin{proof}
		Since $f \in \L_b^{\infty}({}^*G,\W)^{\sim {}^*\Gamma}$, we note that for $\gamma \in {}^*\Gamma$ and $z \in {}^*D$, $f(\gamma z) - \pi_{\Gamma}(\gamma)f(z) \in \W_{inf}$. 
	\end{proof}
	Furthermore, since $\theta$ maps $\L_{inf}^{\infty}({}^*G,\W)$ to $ \L_{inf}^{\infty}({}^*D,\W)$, and $\zeta$ maps $ \L_{inf}^{\infty}({}^*D,\W)$ to $ \L_{inf}^{\infty}({}^*G,\W)$,
	\begin{lemma}
		The internal maps $\theta$ and $\zeta$ induce bijections
		$$\tilde{\theta}: \L_b^{\infty}({}^*G,\W)^{\sim {}^*\Gamma}/\L_{inf}^{\infty}({}^*G,\W) \to \tilde{\L}^{\infty}({}^*D,\W)$$
		$$\tilde{\zeta}:\tilde{\L}^{\infty}({}^*D,\W) \to \L_b^{\infty}({}^*G,\W)^{\sim {}^*\Gamma}/\L_{inf}^{\infty}({}^*G,\W)$$
		with $\tilde{\zeta}=\tilde{\theta}^{-1}$. 
	\end{lemma}
	Henceforth we shall restrict $\theta$ to $\L_b^{\infty}({}^*G,\W)^{\sim {}^*\Gamma}$. We shall now define an internal map
	$$\pi_G: {}^*G \times \L^{\infty}({}^*D,\W) \to \L^{\infty}({}^*D,\W)$$
	\begin{equation}
	\pi_G(g)(f)(z) \coloneqq \pi_{\Gamma}(\gamma) f(x)\label{piG}
	\end{equation}
	where $zg = \gamma x$ for $\gamma \in {}^*\Gamma$ and $x \in {}^*D$. Note that ${}^*G \times {}^*D \to {}^*D$ given by $(g,z) \mapsto x$, where $zg=\gamma x$, defines an internal right action of ${}^*G$ on ${}^*D$. So the above map $\pi_G$ in \cref{piG} can be denoted as
	$$\pi_G(g)(f)(z) \coloneqq \pi_{\Gamma}(\gamma) f(zg)$$
	This map $\pi_G$ induces an action of ${}^*G$ on $\tilde{\L}^{\infty}({}^*D,\W)$, which we shall denote $\tilde{\pi_G}$. In particular, note that this gives the internal Banach space $\L^{\infty}({}^*D,\W)$ the structure of an asymptotic Banach ${}^*G$-module, with the asymptotic ${}^*G$-representation being $\pi_G$ as defined above.
	\begin{lemma}
		The maps $\tilde{\theta}$ and $\tilde{\zeta}$ are ${}^*G$-equivariant. 
	\end{lemma}
	\begin{proof}
		Note that while we have an internal action of ${}^*G$ on $\L^{\infty}({}^*G,\W)$ that is invariant on $\L_b^{\infty}({}^*G,\W)^{\sim {}^*\Gamma}$, the internal map $\pi_G$ is not an action of ${}^*G$ on $\L^{\infty}({}^*D,\W)$. Nevertheless, let $f \in \L_b^{\infty}({}^*G,\W)^{\sim {}^*\Gamma}$ and $g \in {}^*G$. Then $\pi_G(g) (\theta f)(z)=\pi_{\Gamma}(\gamma) (\theta f)(x)=\pi_{\Gamma}(\gamma) f(x)$ where $zg=\gamma x$. Also, $(\theta gf)(z)=f(zg)=f(\gamma x)$. Since $f \in \L_b^{\infty}({}^*G,\W)^{\sim {}^*\Gamma}$, $f(\gamma x) -\pi_{\Gamma}(\gamma)f(x) \in \W_{inf}$. Thus, $\theta gf - \pi_G(g) \theta f \in \L_{inf}^{\infty}({}^*D,\W)$. A similar argument hold for $\zeta$ as well.     
	\end{proof}
	\begin{remark}\label{WbVG}
		Consider the space $\L^{\infty}({}^*D,\W)^{\sim {}^*G}$ of asymptotic ${}^*G$-fixed elements. The restrictions of the maps $\zeta$ and $\theta$ to $\L^{\infty}({}^*D,\W)^{\sim {}^*G}$ and $\left(\L^{\infty}({}^*G,\W)^{\sim {}^*\Gamma}\right)^{\sim {}^*G}$ allows us to identify these two spaces upto infinitesimals. Furthermore, from \cref{VG} and \cref{Wb}, we see that, upto infinitesimals, $\L^{\infty}({}^*D,\W)^{\sim {}^*G}$ can be identified with $\W_b^{\sim {}^*\Gamma}$.
	\end{remark}
	One of the advantages of defining and working with $\L^{\infty}({}^*D,\W)$ (instead of $\L^{\infty}({}^*G,\W)^{\sim {}^*\Gamma}$) is that, not only is it an asymptotic Banach ${}^*G$-module, but is also easily seen to be dual, which we explicitly describe below. Consider the internal Banach space $\L^1({}^*D,\W^{\flat})$ constructed as
	$$\L^1({}^*D,\W^{\flat}) \coloneqq \prod_{\U}L^1\left(D,(\mathfrak{u}(k_n))^{\flat}\right)$$
	where $L^1\left(D,(\mathfrak{u}(k_n))^{\flat}\right)$ is the Bochner-Lebesgue space of Bochner-integrable functions from $D$ to $(\mathfrak{u}(k_n))^{\flat}$. Note that $\Linf(D,\mathfrak{u}(k_n))$ is the dual space of $L^1(D,(\mathfrak{u}(k_n))^{\flat})$, so in particular, an element of $\L^{\infty}({}^*D,\W)$ is an internal linear map from $\L^1({}^*D,\W^{\flat})$ to ${}^*\R$. We now have an explicit dual pairing  can be used to construct the predual asymptotic ${}^*G$-action given $\pi_G$. For $f=\{f_n\}_{\U} \in \L^{\infty}({}^*D,\W)$ and $\eta=\{\eta_n\}_{\U} \in  \L^1({}^*D,\W^{\flat})$, define $\langle f, \eta \rangle$ as
	$$\langle f, \eta \rangle = \left\{ \int_D \langle f_n(x),\eta_n(x) \rangle dx \right\}_{\U}$$
	This defines an internal pairing
	$$\langle,\rangle_{\U} : \L^{\infty}({}^*D,\W) \times \L^1({}^*D,\W^{\flat}) \to {}^*\R$$
	that induces a pairing between the $\R$-spaces $\tilde{\L}^{\infty}({}^*D,\W)$ and $\tilde{\L}^1({}^*D,\W^{\flat})$. The space $\L^1({}^*D,\W^{\flat})$ comes with an internal map 
	$$\pi_G^{\flat}: {}^*G \times \L^1({}^*D,\W^{\flat}) \to \L^1({}^*D,\W^{\flat})$$
	such that:
	\begin{itemize}
		\item The map $\tilde{\pi}_G$ is contragredient to $\tilde{\pi}_G^{\flat}$, that is, for $f \in \L^{\infty}({}^*D,\W)$, $\eta \in \L^{1}({}^*D,\W^{\flat})$ and $g \in {}^*G$, 
		$$\langle \pi_G(g)f, \eta \rangle - \langle f, \pi_G^{\flat}(g)\eta \rangle \in {}^*\R_{inf}$$
		\item The internal map $\pi_G^{\flat}$ induces an action of ${}^*G$, denoted $\tilde{\pi}_G^{\flat}$, on the quotient $\tilde{\L}^1({}^*D,\W^{\flat})$.
	\end{itemize}
	Thus, 
	\begin{proposition}
		The internal Banach space $(\pi_G,\L^{\infty}({}^*D,\W))$ is a dual asymptotic Banach ${}^*G$-module with predual $\L^{1}({}^*D,\W^{\flat})$. 
	\end{proposition}
	We now have the dual asymptotic Banach ${}^*G$-module $\V=\L^{\infty}({}^*D,\W)$ and an internal map 
	$$\pi_G: {}^*G \times \V \to \V$$
	that induces an action of ${}^*G$ on $\tilde{\V}$. This allows us to define the asymptotic cohomology of $G$ with coefficients in $\V$ as in \S\ref{ssec-basicG}, as the cohomology of the complex
	$$\begin{tikzcd}
	0 \arrow[r] & \tilde{\V}^{{}^*G}  \arrow[r, "\tilde{d}^{-1}"]&  \tilde{\L}^{\infty}({}^*G,\V)^{{}^*G} \arrow[r, "\tilde{d}^0"] & \tilde{\L}^{\infty}(({}^*G)^2,\V)^{{}^*G}  \arrow[r, "\tilde{d}^1"]  & \dots
	\end{tikzcd}$$ 
	\begin{theorem}\label{induction}
		For every $m \geq 0$, $\Ha^m(G,\V) \cong \Ha^m(\Gamma,\W)$.
	\end{theorem}
	The first step towards proving \cref{induction} involves a bijection between ${}^*\Gamma$-invariants in $\tilde{\L}^{\infty}(({}^*G)^m,\W)$, and ${}^*G$-invariants in $\tilde{\L}^{\infty}\left(({}^*G)^m,\V \right)$. After that, we shall use the internal maps $\theta: \L^{\infty}({}^*G,\W) \to \L^{\infty}({}^*D,\W)$ and $\zeta: \L^{\infty}({}^*D,\W) \to \L^{\infty}({}^*G,\W)$ (defined in (\cref{theta}) and (\cref{zeta})) to pass between $\L_{b}^{\infty}({}^*G,\W)^{\sim {}^*\Gamma}$ and $\V$.\\
	Let $\tilde{\alpha} \in \tilde{\L}^{\infty}(({}^*G)^m,\W)^{{}^*\Gamma}$, and let $\alpha \in \L^{\infty}(({}^*G)^m,\W)$ be an internal map that induces $\tilde{\alpha}$ (note that $\alpha \in \L^{\infty}_{b}(({}^*G)^m,\W)^{\sim {}^*\Gamma}$). Define the internal map
	$$A_{\alpha}: ({}^*G)^m \to \L^{\infty}({}^*G,\W)$$
	$$A_{\alpha}(g_1,\dots,g_m) (x) \coloneq \alpha(xg_1,\dots,xg_m)$$
	Firstly it is clear that $A_{\alpha} \in \L^{\infty}\left(({}^*G)^m,\L^{\infty}({}^*G,\W) \right)$. Furthermore,
	\begin{proposition}
		For every $\tilde{\alpha} \in \tilde{\L}^{\infty}(({}^*G)^m,\W)^{{}^*\Gamma}$, the map $A_{\alpha} \in \L^{\infty}\left(({}^*G)^m,\L^{\infty}({}^*G,\W) \right)$ takes values in $\L^{\infty}_{b}({}^*G,\W)^{\sim {}^*\Gamma}$ and is ${}^*G$-equivariant. 
	\end{proposition}
	\begin{proof}
		For $g,x \in {}^*G$, 
		$$A_{\alpha}(gg_1,\dots,gg_m)(x)= \alpha(xgg_1,\dots,xgg_m) = A_{\alpha}(g_1,\dots,g_m)(xg)$$
		This proves that $A_{\alpha}$ is ${}^*G$-equivariant. Next, for $\gamma \in {}^*\Gamma$, 
		$$A_{\alpha}(g_1,\dots,g_m)(\gamma x)=\alpha(\gamma xg_1,\dots,\gamma xg_m)$$
		Since $\tilde{\alpha} \in  \tilde{\L}^{\infty}(({}^*G)^m,\W)^{{}^*\Gamma}$, this means that 
		$$\alpha(\gamma xg_1,\dots,\gamma xg_m)-\pi_{\Gamma}(\gamma)\alpha( xg_1,\dots, xg_m) \in \W_{inf}$$
		Hence $A_{\alpha}(g_1,\dots,g_m)(\gamma x)-\pi_{\Gamma}(\gamma) \left( A_{\alpha}(g_1,\dots,g_m)(\gamma x) \right) \in \W_{inf}$. This shows that $A_{\alpha}(g_1,\dots,g_m) \in \L^{\infty}_{b}({}^*G,\W)^{\sim {}^*\Gamma}$.  
	\end{proof}
	Thus, given a ${}^*\Gamma$-equivariant map $\tilde{\alpha} \in \tilde{\L}^{\infty}(({}^*G)^m,\W)$, we obtain an internal ${}^*G$-equivariant map $A_{\alpha} \in \L^{\infty}\left(({}^*G)^m,\L^{\infty}({}^*G,\W) \right)$ that takes values in $\L^{\infty}_{b}({}^*G,\W)^{\sim {}^*\Gamma}$.\\
	Conversely, suppose we have an internal ${}^*G$-equivariant map $A \in \L^{\infty}\left(({}^*G)^m,\L^{\infty}({}^*G,\W) \right)$ that takes values in $\L^{\infty}_{b}({}^*G,\W)^{\sim {}^*\Gamma}$. Define the internal map
	$$\alpha_A \in \L^{\infty}(({}^*G)^m,\W)$$
	$$\alpha_A(g_1,\dots,g_m) \coloneqq A(x^{-1}g_1,\dots,x^{-1}g_m)(x)$$
	for $x \in {}^*G$. Note that since $A$ is ${}^*G$-equivariant, the map $x \mapsto A(x^{-1}g_1,\dots,x^{-1}g_m)(x)$ is essentially constant in $\W$, making the above well-defined.
	\begin{lemma}\label{updown}
		Given an internal ${}^*G$-equivariant map $A \in \L^{\infty}\left(({}^*G)^m,\L^{\infty}({}^*G,\W) \right)$ that takes values in $\L^{\infty}_{b}({}^*G,\W)^{\sim {}^*\Gamma}$, the internal map $\alpha_A$ as defined above induces the map $\tilde{\alpha}_A$ that is ${}^*\Gamma$-equivariant. That is, $\tilde{\alpha}_A \in \tilde{\L}^{\infty}(({}^*G)^m,\W)^{{}^*\Gamma}$.
	\end{lemma}
	\begin{proof}
		Let $\gamma \in {}^*\Gamma$, then since $A$ is ${}^*G$-equivariant,
		$$\alpha_A(\gamma g_1,\dots,\gamma g_m) = A(x^{-1}\gamma g_1,\dots,x^{-1}\gamma g_m)(x)= A( x^{-1}g_1,\dots,x^{-1}g_m)(\gamma x)$$
		Since $A$ takes values in $\L^{\infty}_{b}({}^*G,\W)^{\sim {}^*\Gamma}$, we know that
		$$A( x^{-1}g_1,\dots,x^{-1}g_m)(\gamma x) - \pi_{\Gamma}(\gamma) \left( A( x^{-1}g_1,\dots,x^{-1}g_m)(x) \right) \in \W_{inf}$$
		Thus, $\alpha_A(\gamma g_1,\dots,\gamma g_m)-\pi_{\Gamma}(\gamma)\alpha_A( g_1,\dots, g_m) \in \W_{inf}$, proving that $\tilde{\alpha}_A \in \tilde{\L}^{\infty}(({}^*G)^m,\W)^{{}^*\Gamma}$. 
	\end{proof}
	Furthermore, the correspondences 
	\begin{align}\label{alphaA}
	\alpha & \mapsto  A_{\alpha} \\
	A & \mapsto  \alpha_A
	\end{align}
	are clearly inverses of each other, thus giving a bijection between internal maps $\alpha \in  \L^{\infty}_{b}(({}^*G)^m,\W)$ that are asymptotically ${}^*\Gamma$-equivariant, and internal ${}^*G$-equivariant maps $A \in \L^{\infty}\left(({}^*G)^m,\L^{\infty}({}^*G,\W) \right)$ that take values in $\L^{\infty}_{b}({}^*G,\W)^{\sim {}^*\Gamma}$.
	\begin{proof}[Proof of \cref{induction}]
		For $m \in \Zplus$, it is sufficient to show a bijection between the quotients $\tilde{\L}^{\infty}(({}^*G)^{m+1},\W)^{{}^*\Gamma}$ and $\tilde{\L}^{\infty}(({}^*G)^{m+1},\V)^{{}^*G}$ that commutes with the differentials $\tilde{d}$.\\
		Let $\alpha \in  \L^{\infty}_{b}(({}^*G)^{m+1},\W)^{\sim {}^*\Gamma}$ and let
		$$A_{\alpha} \in \L^{\infty}\left(({}^*G)^{m+1},\L^{\infty}({}^*G,\W) \right)$$
		be its lift (as in (\cref{alphaA})) that is asymptotically ${}^*G$-equivariant and takes values in $\L^{\infty}_{b}({}^*G,\W)^{\sim {}^*\Gamma}$. Composing $A_{\alpha}$ with $\theta$, we get 
		$$\theta \cdot A_{\alpha} \in \L^{\infty}_b\left( ({}^*G)^{m+1} ,\V \right)^{\sim {}^*G}$$
		In the other direction, for an internal map $A' \in \L^{\infty}_b\left( ({}^*G)^{m+1} ,\V \right)^{\sim {}^*G}$, we first compose $\zeta$ (defined in (\cref{zeta})) with $A'$ to get $\zeta \cdot A' \in \L^{\infty}\left(({}^*G)^{m+1},\L^{\infty}_{b}({}^*G,\W)^{\sim {}^*\Gamma} \right)$, where $\zeta \cdot A'$ is asymptotically ${}^*G$-equivariant. Let $A \in  \L^{\infty}\left( ({}^*G)^{m+1}, \L_b^{\infty}({}^*G,\W)^{\sim {}^*\Gamma} \right)^{ {}^*G}$ be the ${}^*G$-equivariant map infinitesimally close to $\zeta \cdot A'$ (as guaranteed by \cref{almostGequi}).\\
		Now we can descend to $\Gamma$ by defining the internal map
		$$\alpha_A \in \L^{\infty}(({}^*G)^{m+1},\W)$$	$$\alpha_A(g_0,\dots,g_{m}) \coloneqq A(x^{-1}g_0,\dots,x^{-1}g_m)(x)$$
		Since $A$ is ${}^*G$-equivariant, from \cref{updown} we conclude that $\alpha_A$ is asymptotically ${}^*\Gamma$-equivariant. Thus, we have a bijection between the quotients $\tilde{\L}^{\infty}(({}^*G)^{m+1},\W)^{{}^*\Gamma}$ and $\tilde{\L}^{\infty}(({}^*G)^{m+1},\V)^{{}^*G}$ that commutes with the differentials $\tilde{d}.$
	\end{proof}
	
	\subsection{Internal Contraction and Fixed Points}\label{ssec-contraction}
	Recall that we have the dual asymptotic Banach ${}^*G$-module $\L^{\infty}({}^*D,\W)$ with the asymptotic ${}^*G$-action of ${}^*G$ given by $\pi_G$ as in (\cref{piG}). The map $\pi_G$ can be studied in terms of two internal maps: an internal right ${}^*G$-action ${}^*G \times {}^*D \to {}^*D$ of ${}^*G$ on ${}^*D$ given by $(g,z)\mapsto zg$, and another internal twisting map ${}^*G \times {}^*D \to {}^*\Gamma$ given by $(g,z) \mapsto \gamma$ (here $\gamma \in {}^*\Gamma$ and $x \in {}^*D$ are such that $zg=\gamma x$). The latter map is then composed with $\pi_{\Gamma}$ to give $\pi_G(g)(f)(z) \coloneqq \pi_{\Gamma}(\gamma) f(x)$ as in (\cref{piG}).\\
	We now see a continuity property of the asymptotic ${}^*G$-action $\pi_G$ on $\L^{\infty}({}^*D,\W)$. Observe that the internal map $(g,z) \mapsto \pi_{\Gamma}(\gamma)$ is internally measurable (and internally locally constant), and as for the internal (true) action of ${}^*G$ on $\L^{\infty}({}^*D,\W)$ given by $(g \cdot f)(z) \coloneq f(x)$, this is internally continuous, but with respect to the internal $L^2$-norm defined on $f=\{f_n\}_{\U}$ as follows:
	$$\|f\|_2^2 \coloneq \left\{  \int_D \|f_n(x)\|^2 dx\right\}_{\U}$$
	\begin{lemma}\label{2cont}
		The dual asymptotic Banach ${}^*G$-module $\L^{\infty}({}^*D,\W)$ is internally continuous with respect to the internal $L^2$-norm on $\L^{\infty}({}^*D,\W)$. 
	\end{lemma}
	We shall refer to internal continuity with respect to the $L^2$-norm on $\L^{\infty}({}^*D,\W)$ as \textbf{internal $L^2$-continuity}. The main consequence of this property that we shall use is the following: let $g^{(1)},g^{(2)},\dots$ be a sequence of elements of ${}^*G$ (with $g^{(m)}=\{g_n{(m)}\}_{\U}$) such that the internal limit $\lim \limits_{m \to \infty}g^{(m)}=\{\lim\limits_{m \to \infty} g_n^{(m)} \}_{\U}=g \in {}^*G$. Then for $f \in \L^{\infty}({}^*D,\W)$, $\lim \limits_{m \to \infty} f(g^{(m)})=f(g)$ (where the limit is now  with respect to the internal $L^2$-norm). A particular instance of an internal limit we shall use is given by \emph{contraction} of elements. 
	\begin{definition}
		Let $g,h \in G$. We say that $h$ is contracted by $g$ (or $g$ contracts $h$) if $\lim \limits_{m \to \infty} g^{-m} hg^{m} = 1$.  Let $g=\{g_n\}_{\U}$ and $h=\{h_n\}_{\U}$ be elements of ${}^*G$. We say $h$ \textbf{is internally contracted by $g$} (or $g$ internally contracts $h$) if for every $n \in \U$,
		$$\lim \limits_{m \to \infty} g_{n}^{-m} h_n g_n^{m} = 1$$
		In other words, $g$ internally contracts $h$ if the sequence $g^{-m}hg^m$ has internal limit being the identity in ${}^*G$.
	\end{definition}
	\begin{definition}
		An internal subgroup of the ultrapower ${}^*G$ is called an \textbf{internally amenable subgroup} if it is of the form $\prod_{\U}H_n$ where $H_n \leq G$ is a closed amenable subgroup for every $n \in \N$. 
	\end{definition}
	\begin{remark}\label{contractamen}
		Suppose $g,h \in G$ such that $g$ contracts $h$. Then the closed subgroup $<g,h>$ generated by $g$ and $h$ is an amenable subgroup of $G$ (refer Propositions $6.5$ and $6.17$ in \cite{caprace}). In particular, this implies that if $g,h \in {}^*G$ is such that $g$ internally contracts $h$, then the internal subgroup $<g,h>$ generated by $g$ and $h$ in ${}^*G$ is an internally amenable subgroup of ${}^*G$.
	\end{remark} 
	To use the full power of internal contraction and the internal $L^2$-continuity of the asymptotic ${}^*G$-action $\pi_G$, our first step is to \say{correct} the asymptotic ${}^*G$-action $\pi_G$ to get a true action when restricted to an internally amenable subgroup. Note that the imperfection is in the map $\pi_G'$, so consider the map $\alpha:{}^*G \times {}^*D \to \prod_{\U}U(k_n)$ with $\alpha(g,z) \coloneq \pi_{\Gamma}(\gamma)$. Note that $\alpha$ is a measurable map, and for every $g_1,g_2 \in {}^*\Gamma$ and $z \in {}^*D$,
	\begin{equation}\label{twista}
	\|\alpha(g_1,z)\alpha(g_2,zg_1)-\alpha(g_1g_2,z) \| \in {}^*\R_{inf}
	\end{equation}
	Observe that this looks similar to the very classical question of Ulam stability, but with a twist provided by the action of $G$ on $D$. We consider the following definitions which are analogues of uniform stability in the context of such twists:
	\begin{definition}
		Let $G$ be a locally compact, second countable group, and $X$ be a non-singular $G$-space. A measurable map $\psi:G \times X \to U(n)$ is called a \emph{twisted} homomorphism of $G$ (with respect to $D$) if for every $g_1,g_2 \in G$ and almost every $z \in X$, 
		$$\psi(g_1g_2,z)=\psi(g_1,z)\psi(g_2,zg_1)$$ 
		For $\epsilon>0$ (the defect), a measurable  map $\phi:{}G \times X \to U(n)$ is called a twisted $\epsilon$-homomorphism of $G$ (with respect to $D$) if for almost every $g_1,g_2 \in G$ and almost every $z \in X$ $$\|\phi(g_1,z)\phi(g_2,zg_1)-\phi(g_1g_2,z) \|  \leq \epsilon$$
	\end{definition}
	The following claim essentially retraces the arguments used in \S\ref{sec-asymgamma} (where we use the logarithm map to obtain an asymptotic cocycle, and use the vanishing of cohomology to diminish defect). 
	\begin{claim}\label{twisted}
		There exists $\epsilon>0$ small enough, and a constant $C$ such that for an amenable subgroup $H$ of $G$ and any twisted $\epsilon$-homomorphism $\alpha:G \times D \to U(n)$, there exists a measurable map $\alpha_{H}:H \times D \to U(n)$ of $H$ such that for every $h \in H$ and almost every $z \in D$, $\|\alpha(h,z)-\alpha_{H}(h,z)\| \leq C\epsilon$, and for  every $h_1,h_2 \in H$ and $z \in D$, $$\alpha_{H}(h_1,z)\alpha_{H}(h_2,zh_1)=\alpha_{H}(h_1 h_2,z)$$
	\end{claim}
	\begin{proof}
		As in \S\ref{sec-asymgamma}, consider a sequence $\{\alpha_n:H \times D \to U(k_n)\}_{n \in \N}$ and its ultraproduct $\alpha$, which is an asymptotic twisted homomorphism with defect $\epsilon \in {}^*\R_{inf}$. Define the internal map
		$$\omega: {}^*H \times {}^*H \times {}^*D \to \prod_{\U}\mathfrak{u}(k_n)$$
		$$\omega(h_1,h_2,z) = {}_{\epsilon} \log\left(\alpha(h_1,z)\alpha(h_2,zh_1)\alpha(h_1h_2,z)^{-1}\right)$$
		As $\mathfrak{u}(k_n)=\W$, note that we can regard $\omega$ as an element of $\L^{\infty}_b({}^*H \times {}^*H \times {}^*D, \W)$, or equivalently, $\L^{\infty}_b(({}^*H)^2,\L^{\infty}({}^*D,\W))$. Equipping $\L^{\infty}({}^*D,\W)$ with an asymptotic action $\rho_{\alpha}:{}^*H \times \L^{\infty}({}^*D,\W) \to \L^{\infty}({}^*D,\W)$ of ${}^*H$ given by $\rho_{\alpha}(h)(f)(z) \coloneq \alpha(h,z)f(zh)\alpha(h,z)^{-1}$, making it a dual asymptotic ${}^*H$-module. One can check that the map $\omega$ satisfies the condition to be an (inhomogenous) asymptotic $2$-cocycle in $\Ha^2(H,\L^{\infty}({}^*D,\W))$. Now since $H$ is amenable and $\L^{\infty}({}^*D,\W)$ is a dual asymptotic Banach ${}^*H$-module, we know (from \cref{corr-amentriv}) that $\Ha^2(H,\L^{\infty}({}^*D,\W))=0$, implying that there exists $\beta \in \L^{\infty}_b({}^*H,\L^{\infty}({}^*D,\W))$ such that $\alpha \cdot _{\epsilon}\exp{\beta}$ is an $o_{\U}(\epsilon)$-twisted homomorphism (note that $\alpha \cdot _{\epsilon}\exp{\beta}$ is measurable). Repeating this process (as in defect diminishing), we obtain a measurable map $\alpha_{H}':H \times D \to U(n)$ such that for almost every $h_1,h_2 \in H$ and almost every $z \in D$, we have$\alpha_{H}'(h_1,z)\alpha_{H}'(h_2,zh_1)=\alpha_{H}'(h_1h_2,z)$. From Theorem B$9$ (p.200) in \cite{zimmerbook}, we conclude that there exists $\alpha_{H}$ as desired. 
	\end{proof}
	\begin{remark}\label{twista2}
		The same argument also goes through for internally amenable subgroups $\mathcal{H}$ by applying \cref{twisted} internally.
	\end{remark}
	\begin{lemma}\label{piH}
		Let $\mathcal{H} \leq {}^*G$ be an internally amenable subgroup of ${}^*G$. Then there exists an internal map $\pi_\mathcal{H}:\mathcal{H} \times \L^{\infty}({}^*D,\W) \to \L^{\infty}({}^*D,\W)$ such that
		\begin{itemize}
			\item For every $h_1,h_2 \in H$, $\pi_{\mathcal{H}}(h_1h_2)=\pi_{H}(h_1)\pi_H(h_2)$. In other words, $\pi_{\mathcal{H}}$ is a (true) internal action of $\mathcal{H}$ on $\L^{\infty}({}^*D,\W)$.
			\item For every $f \in \L_{b}^{\infty}({}^*D,\W)$ and $h \in \mathcal{H}$, $\pi_{\mathcal{H}}(h)f-\pi_G(h)f \in \L_{inf}^{\infty}({}^*D,\W)$.
			\item The internal action $\pi_{\mathcal{H}}$ of $\mathcal{H}$ is internally $2$-continuous.
		\end{itemize}
	\end{lemma}
	\begin{proof}
		Consider the internal map $\alpha: {}^*G \times {}^*D \to \prod_{\U}U(k_n)$
		$$\alpha(g,z) \coloneq \pi_{\Gamma}(\gamma)$$
		where $\gamma x = zg$ for $\gamma \in {}^*\Gamma$ and $x \in {}^*D$. Since $\alpha$ is an asymptotic twisted homomorphism of $G$, from \cref{twista2}, there exists an internally measurable map $\alpha_{\mathcal{H}}: \mathcal{H} \times {}^*D \to \prod_{\U}U(k_n)$ as in \cref{twisted}.\\
		Define $\pi_{\mathcal{H}}:\mathcal{H} \times \L^{\infty}({}^*D,\W) \to \L^{\infty}({}^*D,\W)$ by
		$$\pi_{\mathcal{H}}(g)(f)(z) \coloneq \alpha_{\mathcal{H}}(g,z) f(x)$$
		whereas always, $x \in {}^*D$ such that $zg=\gamma x$ for $\gamma \in {}^*\Gamma$. Then $\pi_{\mathcal{H}}$ is a (true) internal action of $\mathcal{H}$ on $\L^{\infty}({}^*D,\W)$, and for every $f \in \L_{b}^{\infty}({}^*D,\W)$ and $h \in \mathcal{H}$, $\pi_{\mathcal{H}}(h)f-\pi_G(h)f \in \L_{inf}^{\infty}({}^*D,\W)$. Since $\pi_{\mathcal{H}}$ is an internal action of $\mathcal{H}$ that is internally measurable, it is also internally $2$-continuous. 
	\end{proof}
	\begin{lemma}\label{contract1}
		Let $x,y \in {}^*G$ such that $x$ internally contracts $y$. Suppose $f \in \L^{\infty}_b({}^*D,\W)$ is such that $\pi_G(x)f-f \in \L^{\infty}_{inf}({}^*D,\W)$ (in other words, $x$ fixes $\tilde{f}$). Then $\pi_G(y)f-f \in \L^{\infty}_{inf}({}^*D,\W)$, that is, $y$ too fixes $\tilde{f}$. 
	\end{lemma}
	\begin{proof}
		Let $\mathcal{H}$ be the closure of the subgroup generated by $x$ and $y$. By \cref{contractamen}, $\mathcal{H}$ is an internally amenable subgroup of ${}^*G$. By \cref{piH}, consider the correction $\pi_{\mathcal{H}}$ whose restriction to $\mathcal{H}$ is a (true) internally $2$-continuous action of $\mathcal{H}$ on $\L^{\infty}({}^*D,\W)$. Now $\pi_{\mathcal{H}}(x)f-f \in \L_{inf}^{\infty}({}^*D,\W)$, and let $f_{\mathcal{H}} \in \L_{b}^{\infty}({}^*D,\W)$ such that $f-f_{\mathcal{H}} \in \L_{inf}^{\infty}({}^*D,\W)$ and $\pi_{\mathcal{H}}(x)f_{\mathcal{H}}=f_{\mathcal{H}}$.\\
		Now we have 
		$$\|\pi_{\mathcal{H}}(y)f_H-f_H\|_2=\|\pi_{\mathcal{H}}(y)\pi_{\mathcal{H}}(x^{-m})f_{\mathcal{H}}-\pi_{\mathcal{H}}(x^{-m})f_{\mathcal{H}} \|_2$$
		Since $\pi_{\mathcal{H}}$ too acts unitarily,
		$$\|\pi_{\mathcal{H}}(y)\pi_{\mathcal{H}}(x^{-m})f_{\mathcal{H}}-\pi_{\mathcal{H}}(x^{-m})f_{\mathcal{H}}\|_2 = \|\pi_{\mathcal{H}}(x^{m})\pi_{\mathcal{H}}(y)\pi_H(x^{-m})f_{\mathcal{H}}-f_{\mathcal{H}}\|_2$$
		Since $\pi_{\mathcal{H}}$ is a (true) internal action of ${\mathcal{H}}$, $\pi_{\mathcal{H}}(x^{m})\pi_{\mathcal{H}}(y)\pi_{\mathcal{H}}(x^{-m}) = \pi_{\mathcal{H}}(x^myx^{-m})$, and so
		$$\|\pi_{\mathcal{H}}(x^{m})\pi_{\mathcal{H}}(y)\pi_{\mathcal{H}}(x^{-m})f_{\mathcal{H}}-f_{\mathcal{H}}\|_2 = \|\pi_{\mathcal{H}}(x^myx^{-m})f_{\mathcal{H}} - f_{\mathcal{H}}\|_2$$
		Since $\pi_{\mathcal{H}}$ is internally $2$-continuous and $y$ is internally contracted by $x$, we conclude that $\pi_{\mathcal{H}}(y)f_{\mathcal{H}} = f_{\mathcal{H}}$, implying that $\pi_G(y)f-f \in \L_{inf}^{\infty}({}^*D,\W)$. 
	\end{proof}
	We can apply the above results in a slightly more general setting, which we shall develop into an internal \emph{Mautner's Lemma} in Subection \S\ref{ssec-invariants}. 
	\begin{definition}\label{bdgendef}
		Let $T\leq G$ be a closed subgroup and $M>0$. The group $G$ is said to be \emph{$M$-boundedly generated by $T$-contracted elements} if any element $g \in G$, there exist $s_1,\dots,s_m \in G$ with $m \leq M$, such that $g=s_1s_2\dots s_m$, and each $s_i\in G$ is contracted by some element of $T$. More generally, for a family $\mathfrak{T}$ of closed subgroups of $G$, the group $G$ is said to be \textbf{boundedly generated by $\mathfrak{T}$-contracted elements} if there exists $M>0$ such that $G$ is $M$-boundedly generated by $T$-contracted elements for every $T\in \mathfrak{T}$.
	\end{definition}
	\begin{remark}\label{bdgen}
		Suppose $G$ is boundedly generated by $\mathfrak{T}$-contracted elements for a family $\mathfrak{T}$ of subgroups. Let $\mathcal{T}=\prod_{\U}T_n$ be an internal subgroup of ${}^*G$, with $T_n \in \mathfrak{T}$ for every $n \in \N$. Observe that there exists $M>0$ such that any $g \in {}^*G$ can be expressed as $g=s_1s_2\dots s_m$ such that each $s_i \in {}^*G$ is internally contracted by some element of $\mathcal{T}$. 
	\end{remark}
	With this definition, \cref{contract1} implies the following corollary:
	\begin{corollary}\label{contract2}
		Suppose $G$ is boundedly generated by $\mathfrak{T}$-contracted elements, and let $\mathcal{T}=\prod_{\U}T_n$ be an internal subgroup of ${}^*G$, with $T_n \in \mathfrak{T}$ for every $n \in \N$. Then if $f \in \L^{\infty}_b({}^*D,\W)^{\sim \mathcal{T}}$, then $f \in \L^{\infty}_b({}^*D,\W)^{\sim {}^*G}$ (that is, if $f$ is asymptotically fixed by all elements of $\mathcal{T}$, then it is asymptotically ${}^*G$-fixed).
	\end{corollary}
	\begin{proof}
		Let $f \in \L^{\infty}_b({}^*D,\W)^{\sim \mathcal{T}}$ and $g \in {}^*G$. Let $g=s_1s_2\dots s_m$ as in \cref{bdgen}. Since $f$ is asymptotically fixed by every element of $\mathcal{T}$, \cref{contract1} implies that it is asymptotically fixed by each $s_i$. In particular, $f$ is asymptotically fixed by $g$. Thus, $f \in \L^{\infty}_b({}^*D,\W)^{\sim {}^*G}$. 
	\end{proof}
	
	\section{Vanishing of $\Ha^2(\Gamma,\W)$}\label{sec-mainproof}
	In this section, we combine the results of the previous sections to prove our main result about uniform stability of lattices in semisimple groups. In \S\ref{ssec-invariants}, we build further tools specialized to $G$ being semisimple groups, which in turn shall be used in \S\ref{ssec-mainproof}. We begin with some structure properties of semisimple groups with regard to the notions of bounded generation by contracting elements, as discussed in \S\ref{ssec-contraction}, and use it prove an asymptotic version of the Mautner property and an asymptotic double ergodicity theorem with coefficients. In \S\ref{ssec-mainproof}, we study the Property-$G(\mathcal{Q}_1,\mathcal{Q}_2)$, which is an assumption on $G$ that would allow us to build on the results of \S\ref{ssec-invariants} to prove our main theorem, namely that $\Ha^2(G,\V)=0$. We then conclude the section by studying Property-$G(\mathcal{Q}_1,\mathcal{Q}_2)$ in more detail in \S\ref{ssec-2half}, and list out a large class of semisimple groups that satisfy it, thus making our main result applicable to them. 
	
	\subsection{Asymptotic Mautner Property and Ergodicity}\label{ssec-invariants}
	In this subsection and henceforth, we shall work with $G$ being a semisimple group of the form $G = \prod^k_{i = 1} {\GG_i}(K_i)$ where for $1\leq i \leq k$, $K_i$ is a local field, and $\GG_i$ is a connected, simply connected, almost $K_i$-simple group. Our goal is to use the results developed in the previous sections to prove that $\Ha^2(\Gamma,\W)=\Ha^2(G,\V)=0$ (recall that $\W$ is the ultraproduct $\prod_{\U}\mathfrak{u}(k_n)$ which is an asymptotic Banach ${}^*\Gamma$-module obtained from the asymptotic homomorphism we start with, while $\V=\L^{\infty}({}^*D,\W)$ is the asymptotic Banach ${}^*G$-module obtained by the induction procedure in \cref{ssec-shapiro}).\\
	Recall \cref{bdgendef} of bounded generation by contracted elements. We shall now see that $G$ is boundedly generated by $\mathfrak{T}$-contracted elements, for the subgroup family $\mathfrak{T}$ being the maximal tori of the radicals of proper parabolic subgroups of $G$. 
	\begin{proposition}\label{rapinchuk}
		Let $G=\GG(K)$ be a connected, simply connected, almost $K$-simple $K$-isotropic group over local field $K$, and let $\mathfrak{T}$ be the family of maximal tori of the solvable radicals of proper parabolic subgroups of $G$. Then $G$ is boundedly generated by $\mathfrak{T}$-contracted elements. 
	\end{proposition}
	\begin{proof}
		Let $T$ be a maximal torus contained in the solvable radical of a proper parabolic subgroup $Q$ of $G$. It is sufficient to show that for some $M>0$, $G$ is $M$-boundedly generated by $T$-contracted elements. This is because the same bound $M$ works for conjugates of $T$, and upto conjugation, the group $G$ has only finitely many parabolic subgroups.\\
		Let $Q=L\cdot R_{u}(Q)$ be the Levi decomposition of $Q$, where $L$ is the Levi subgroup and $R_{u}(Q)$ is the unipotent radical of $Q$. Then every element of $R_{u}(Q)$ is contracted by some element of $T$. In fact, the same holds in the case of the opposite parabolic $Q^{-}$, that is, every element of $R_{u}(Q^{-})$ too is contracted by some element of $T$. Hence it is sufficient to show that $G$ is boundedly generated by $\Delta \coloneqq R_{u}(Q)(K) \cup R_{u}(Q^{-})(K)$, that is, to show that there exists $\ell>0$ such that $G=\Delta^{\ell}$.\\
		Firstly, note that $G$ is indeed generated by $\Delta$ (\cite[1.5.4]{margulisbook}), and so, there exists $k>0$ so that the product map ${\Delta}^k \to G$ is a dominant $K$-morphism. From \cite[Proposition 3.3]{rapinchuk}, we conclude that $\Omega \coloneqq \Delta^{\ell}$ contains an open neighborhood of the identity. Let $T$ be a maximal $K$-split torus of $G$ contained in $Q \cap Q^{-}$. For any $K$-defined unipotent subgroup $U \subset G$ which is normalized by $T$, we have $$U(K)=\bigcup_{t \in T(K)} t (U(K)\cap \Omega) t^{-1}$$
		Since $T(K)$ normalizes $\Delta$ (and hence also $\Omega$), we see that $U(K) \subseteq \Omega$. In particular, for a $K$-defined minimal parabolic subgroup $P \subset Q$ containing $T$, we conclude that $R_u(P)(K) \subset \Omega$.\\
		Let $D=T\cdot R_u(P)$. Then there exists a compact subset $E \subset G$ so that
		$$G=E \cdot D(K)=E \cdot T(K)\cdot R_u(P)(K)$$
		Since $\Omega$ is open and generates $G$, there exists $s>0$ such that $E \subset \Omega^s$. Hence it is sufficient to show that there exists $t >0$ such that $T(K) \subset \Omega^t$.\\
		There exists $N>0$ such that for any root $\alpha$ of $T$ with coroot $\alpha^{\vee}$ (recall that for a root $\alpha:T(K) \to K$, its dual $\alpha^{\vee}:K \to T(K)$ is such that $\alpha^{\vee}(\alpha)=2$), $\alpha^{\vee}(t)$ for $t \in K^{\times}$ can be written as a product of $N$ elements of $\Delta$ (for instance, $N=6$ when $G=SL_n$). In particular, this means that every element of $T(K)$ can be written as a product of $Nn$ elements of $\Delta$ (where $n$ is the rank of $G$). This concludes the proof.
	\end{proof}
	In fact, \cref{rapinchuk} can be immediately extended to semisimple groups as well, as long as we ensure that the projection of the tori to each factor is non-trivial.
	\begin{corollary}\label{rapinchuk2}
		Let $G$ be a semisimple group of the form $G = \prod^k_{i = 1} {\GG_i}(K_i)$ where for $1\leq i \leq k$, $K_i$ is a local field, and $\GG_i$ is a connected, simply connected, almost $K_i$-simple $K_i$-isotropic group, and let $\mathfrak{T}$ be the family of maximal tori of the radicals of parabolic subgroups of $G$ of the form $Q=\prod_{i=1}^k Q_i$, where each $Q_i$ is a proper parabolic subgroup of ${\GG_i}(K_i)$. Then $G$ is boundedly generated by $\mathfrak{T}$-contracted elements. 
	\end{corollary}
	Combining \cref{contract2} with \cref{rapinchuk2}, we immediately get the following:
	\begin{corollary}[Asymptotic Mautner Property]\label{Mautner}
		Let $G$ and a parabolic subgroup $Q \leq G$ be as in \cref{rapinchuk2}, and let $N \leq G$ be the radical of $Q$. Then $\V_b^{\sim {}^*N} = \V_b^{\sim {}^*G}$. 
	\end{corollary}
	\cref{contract2} and \cref{rapinchuk2} can also be used to get an internal analogue of ergodicity (and double ergodicity) with coefficients as in \cite{bookMonod}. We recall the following classical definitions for comparison:
	\begin{definition}
		For a locally compact second countable group $G$ and a regular $G$-space $S$, the $G$-action on $S$ is said to be \emph{ergodic} if any measurable $G$-invariant function $f:S \to \R$ is essentially constant. The $G$-action on $S$ is said to be \emph{doubly ergodic} if the diagonal $G$-action on $S \times S$ is ergodic.
	\end{definition}
	
	There is no difference if $\R$ above is replaced by any dual separable Banach space: the $G$-action on $S$ is ergodic (resp., double ergodic) iff for any dual separable Banach space $E$, any weak-* measurable $G$-\emph{invariant} map $f:S \to E$ (resp, $f:S \times S \to E$) is essentially constant. Indeed the underlying Borel structures are isomorphic. Even if $E$ is not separable, it suffices that it admits a weak-* measurable (for instance dual linear) injection into a separable dual, since we can then apply ergodicity in the latter space. Here is an illustration of this phenomenon:
	
	\begin{claim}\label{invconstant}
		Let $f\colon G \to \Linf_{w*}(D,\mathfrak{u}(k_n))$ be a $G$-invariant weak-* measurable map. Then $f$ is essentially constant.
	\end{claim}
	\begin{proof}
		The $G$-action on itself is ergodic, so that this follows from the previous discussion after noticing that $\Linf_{w*}(D,\mathfrak{u}(k_n))$ admits a dual linear injection into $L^2(D,\mathfrak{u}(k_n))$, which is dual and separable. This injection is the dual of the natural map from $L^2(D,\mathfrak{u}(k_n))$ to $L^1(D,\mathfrak{u}(k_n))$, which is defined because $D$ has finite measure.
	\end{proof}
	Note that \cref{invconstant} immediately implies the following (wherein no $G$-structure is considered on $\V$):
	\begin{claim}\label{invconstantint}
		Let $F \in \L_b^{\infty}\left( {}^*G,\V\right)$ be ${}^*G$-invariant (that is, for every $x \in {}^*G$ and almost every $g \in {}^*G$, $F(xg)=F(g)$). Then $F$ is essentially constant. \qed
	\end{claim}
	A classical example of a doubly ergodic action for a semisimple group $G$ (as in \cref{rapinchuk2}) is its action on $G/P$ for $P$ being a minimal parabolic subgroup of $G$ (which implies that the action of $G$ on $G/Q$ is doubly ergodic for any parabolic subgroup $Q$ of $G$).
	
	We shall now use \cref{contract2} with \cref{rapinchuk2} to get a much stronger double ergodicity result for asymptotically ${}^*G$-\emph{equivariant} maps. Classically, this is referred to as double ergodicity \emph{with coefficients} (\cite[Chapter 4.11]{bookMonod}):
	
	A $G$-action on a regular $G$-space $S$ is said to be \emph{ergodic with coefficients in a Banach $G$-module $E$} if any weak-$*$ measurable $G$-equivariant map $f: S \to E$ is essentially constant. Note that this is a much stronger condition than ergodicity; it does not even hold for the transitive action of $G$ on itself. The action is called \emph{doubly ergodic with coefficients} if the corresponding condition holds for the diagonal $G$-action on $S \times S$.\\
	
	In our setting, the goal is to show that for $f \in \L^{\infty}_{b}\left({}^*(G/P)^2,\V \right)^{\sim {}^*G}$, there exists $F \in \V_b$ so that $f-F \in \L^{\infty}_{inf}\left({}^*(G/P)^2,\V \right)$. In other words, we would like to prove an asymptotic version of double ergodicity of the action of $G$ on $G/P$ with coefficients being the asymptotic Banach ${}^*G$-module $\V$. The crucial difference with \Cref{invconstant} is that we involve the $G$-structure on $\V$ this time.
	
	\begin{remark}
		We caution the reader that no (non-trivial) action whatsoever can be ergodic with \emph{arbitrary} coefficients. Therefore, the nature of the asymptotic $G$-module $\V$ intervenes in the result. Specifically, the fact that it involves a finite $G$-invariant measure on $D$ is used and it enters the proof through our appeal to the asymptotic Mautner's lemma (\cref{Mautner}), which ultimately relies on the $L^2$-topology used in \Cref{contract1}. That $L^2$-continuity argument, introduced in \S\ref{ssec-contraction}, is not available in the absence of a finite invariant measure (and indeed the statements can fail in that absence).\\
	\end{remark}
	
	We begin with a series of short claims that shall be used in the proof of \cref{lemma-ergodicity1}. The first of these states that it is sufficient to work with $f \in \L^{\infty}_{b}\left({}^*(G/A),\V \right)^{\sim {}^*G}$ for $A=P \cap wPw^{-1}$ for some $w \in G$.
	\begin{claim}\label{GPA}
		There exists $w \in G$ such that for $A=P \cap wPw^{-1}$, the map 
		$$G/A \to G/P \times G/P$$
		$$gA \mapsto (gP,gwP)$$
		is a measure space isomorphism respecting the action of $G$.
	\end{claim}
	\begin{proof}
		The result is a consequence of \cite[Theorem 3.13 and Corollary 3.15]{tits}. 
	\end{proof}
	The next claim is a special case of the asymptotic Mautner's lemma \cref{Mautner} for $A$.
	\begin{claim}\label{MautnerA}
		Let $A$ be as in \cref{GPA}. Then $\V_b^{\sim {}^*A} = \V_b^{\sim {}^*G}$.
	\end{claim}
	\begin{proof}
		This is again a consequence of \cref{rapinchuk2} and \cref{contract2}, since $A$ contains a maximal torus of $G$. 
	\end{proof}
	Now we obtain a correction of \emph{asymptotically} ${}^*G$-invariant maps to \emph{truly} ${}^*G$-invariant maps:
	\begin{claim}\label{inv1}
		Suppose $f \in \L_b^{\infty}\left( {}^*G,\V\right)$ is such that for every $x \in {}^*G$ and almost every $g \in {}^*G$, $f(xg)-f(g) \in \V_{inf}$. Then there exists $F \in \L_b^{\infty}\left( {}^*G,\V\right)$ such that $F-f \in \L_{inf}^{\infty}\left( {}^*G,\V \right)$ and for every $x \in {}^*G$ and almost every $g \in {}^*G$, $F(xg)=F(g)$.
	\end{claim}
	\begin{proof}
		The proof is the same as in \cref{correction}, but now using the fact that $\L^{\infty}\left( {}^*G,\V\right)$ is the ultraproduct of relatively injective Banach $G$-modules $\Linf_{w*}(G,\Linf_{w*}(D,\mathfrak{u}(k_n)))$.
	\end{proof}
	With the above statement at hand, we now prove that: 
	\begin{theorem}[Double Ergodicity with coefficients]\label{lemma-ergodicity1}
		Let $G$ be as in \cref{rapinchuk2}. For $f \in \L^{\infty}_{b}\left({}^*(G/P)^2,\V \right)^{\sim {}^*G}$, there exists $v \in \V_b^{\sim {}^*G}$ so that $f-v \in \L^{\infty}_{inf}\left({}^*(G/P)^2,\V \right)$. 
	\end{theorem}
	\begin{proof}
		From \cref{GPA}, we know that this is equivalent to proving that for $f \in \L^{\infty}_{b}\left({}^*(G/A),\V \right)^{\sim {}^*G}$, there exists $v \in \V_b$ so that $f-v \in \L^{\infty}_{inf}\left({}^*(G/A),\V \right)$. Define $f' \in \L^{\infty}_{b}\left({}^*G,\V \right)$ as
		$f'(g) \coloneqq \pi_G(g)^{-1}f(g \,{}^*A)$. Since $\pi_G$ is an asymptotic ${}^*G$-action and $f$ is asymptotically ${}^*G$-equivariant, for every $x \in {}^*G$ and almost every $g \in {}^*G$, 
		$$f'(xg) - f'(g) \in \V_{inf}$$
		In other words, $f'$, by construction, is asymptotically ${}^*G$-invariant. From \cref{inv1}, there exists $F \in \L_b^{\infty}\left( {}^*G,\V\right)$ such that $F-f \in \L_{inf}^{\infty}\left( {}^*G,\V \right)$ and for every $x \in {}^*G$ and almost every $g \in {}^*G$, $F(xg)=F(g)$. Now from \cref{invconstantint}, $F$ is essentially constant, and hence there exists $v \in \V_b$ such that for almost every $g \in {}^*G$, $f'(g)-v \in \V_{inf}$. Thus, for almost every $g \in {}^*G$,
		$$f(g {}^*A) - \pi_G(g)v \in \V_{inf}$$
		We now claim that $v \in \V_b^{\sim {}^*A}$. Then, by \cref{MautnerA}, we have $v \in \V_b^{\sim {}^*G}$, allowing us to conclude that for almost every $g \in {}^*G$, $f(g {}^*A) - v \in \V_{inf}$. Hence it is left to prove that $v \in \V_b^{\sim {}^*A}$.\\
		Define the pullback $\hat{f} \in \L^{\infty}_{b}\left({}^*G,\V \right)$ by $\hat{f}(g) \coloneqq f(g {}^*A)$, and fix a measurable section $s:G/A \to G$ to define a measure isomorphism 
		$$G/A \times A \to G$$
		$$(gA,a) \mapsto s(gA)a$$
		The above isomorphism can be defined internally to get an internal measure isomorphism ${}^*(G/A) \times {}^*A \to {}^*G$ (with the internal section map which we continue denoting $s$ for simplicity), and we have that for almost every $g{}^*A \in {}^*(G/A)$ and almost every $a \in {}^*A$,
		$$\hat{f}(s(g{}^*A)a) - \pi_G(s(g{}^*A)a)v \in \V_{inf}$$
		Since $\hat{f}(s(g{}^*A)a)=\hat{f}(s(g{}^*A))$, the above simplifies to the following: for almost every $g{}^*A \in {}^*(G/A)$ and almost every $a \in {}^*A$
		$$\pi(s(g{}^*A))v - \pi_G(s(g{}^*A)a)v \in \V_{inf}$$
		Now we fix some $g {}^*A \in {}^*(G/A)$ so that for this $g{}^*A$, we have that for almost every $a \in {}^*A$,
		$$\pi_G(s(g{}^*A))v - \pi_G(s(g{}^*A)a)v \in \V_{inf}$$
		Since $\pi_G$ is an asymptotic action of ${}^*G$, $\pi_G(s(g{}^*A)a)v-\pi_G(s(g{}^*A))\pi_G(a)v \in \V_{inf}$. And so, we conclude that, for almost every $a \in {}^*A$,
		$$v - \pi_G(a)v \in \V_{inf}$$
		This means that there exists an internal subset $\mathcal{A}=\{A'_n\}_{\U}$ of ${}^*A$ with $A'_n$ being a conull subset of $A$ for $n \in \U$, such that $v \in \V_b^{\sim \mathcal{A}}$. But note that $\mathcal{A}\cdot\mathcal{A} = {}^*A$ (since if $A_n'$ is a co-null subset of the locally compact group $A$, then $A_n'A_n' = A$). Hence $v \in \V_b^{\sim {}^*A}$ as claimed.
	\end{proof}
	\cref{lemma-ergodicity1} tells us that for any $f \in \L^{\infty}_{b}\left({}^*(G/P)^2,\V \right)^{\sim {}^*G}$ the induced map $\tilde{f} \in\tilde{\L}^{\infty}\left({}^*(G/P)^2,\V \right)^{{}^*G}$ is essentially constant. This is the asymptotic analogue of the classical result that the $G$-action on $G/P$ is doubly ergodic with coefficients in a suitable $G$-module.\\ 
	We conclude this subsection by showing that $\Ha^2(Q,\V)=0$ for a proper parabolic subgroup $Q$ of $G$ as in \cref{rapinchuk2} such that $\Hb^2(Q,\R)=0$ and $\Hb^3(Q,\R)$ is Hausdorff. We begin with the following proposition that is an application of \cref{Mautner}:
	\begin{proposition}\label{Qinv1}
		Let $G$ and a parabolic subgroup $Q \leq G$ be as in \cref{rapinchuk2}. Let $\omega \in \L^{\infty}_{b}(({}^*Q)^3,\V)^{\sim {}^*Q}$ be an asymptotic $2$-cocycle for $Q$ with coefficients in $\V$. Then there exists an asymptotic $2$-cocycle $\omega' \in \L^{\infty}_{b}(({}^*Q)^3,\V)^{\sim {}^*Q}$ for $Q$ that takes values in $\V_b^{\sim {}^*G}$ such that 
		$$\tilde{\omega} - \tilde{\omega}' = \tilde{d}^1\tilde{\alpha_1}$$
		where $\alpha_1 \in \L^{\infty}_{b}(({}^*Q)^2,\V)^{\sim {}^*Q}$.
	\end{proposition}
	\begin{proof}
		Let $N$ be the unipotent radical of $Q$, which is a normal amenable closed subgroup of $Q$. Since $N$ is amenable, \cref{zimmer} tells us that $\Ha^{\bullet}(Q,\V)$ can be computed as the asymptotic cohomology of the asymptotic ${}^*Q$-cochain complex
		$$\begin{tikzcd}
		0 \arrow[r,"d^0"] & \L^{\infty}({}^*(Q/N),\V)  \arrow[r, "d^{1}"]&  \L^{\infty}(({}^*(Q/N))^2,\V) \arrow[r, "d^2"] & \L^{\infty}(({}^*(Q/N))^3,\V) \arrow[r, "d^3"]  & \dots
		\end{tikzcd}$$
		Let $k^{\bullet}:\L^{\infty}(({}^*Q)^{\bullet},\V) \to \L^{\infty}(({}^*(Q/N))^{\bullet},\V)$ and $j^{\bullet}:\L^{\infty}(({}^*(Q/N))^{\bullet},\V) \to \L^{\infty}(({}^*Q)^{\bullet},\V)$ be the asymptotic ${}^*Q$-homotopy equivalences. Consider $k^2\omega \in \L^{\infty}_b(({}^*(Q/N))^3,\V)^{\sim {}^*Q}$. Since $N$ is normal in $Q$,  $Im(k^2\omega) \subseteq \V_b^{\sim {}^*N}$. By \cref{Mautner}, since $\V_b^{\sim {}^*N}=\V_b^{\sim {}^*G}$, this implies that $Im(k^2\omega) \subseteq \V_b^{\sim {}^*G}$. The conclusion then follows from (\cref{invariants}). 
	\end{proof}
	Thus, the obstacle to the asymptotic $2$-cocycle $\omega \in \L^{\infty}_{b}(({}^*Q)^3,\V)^{\sim {}^*Q}$ being an asymptotic $2$-coboundary is the asymptotic $2$-cocycle $\omega' \in \L^{\infty}_{b}(({}^*Q)^3,\V)^{\sim {}^*Q}$ that takes values in $\V_b^{\sim {}^*G}$. We shall now see how to handle this using assumptions on $Q$. \\
	Consider the asymptotic $2$-cocycle $\omega' \in \L^{\infty}_{b}(({}^*Q)^3,\V)^{\sim {}^*Q}$ for $Q$ that takes values in $\V^{\sim {}^*G}$. The first step is to correct $\omega'$ to an element $\omega'' \in \L^{\infty}_{b}(({}^*Q)^3,\W)^{{}^*Q}$ with $Im(\omega'') \subseteq \W_b^{\sim {}^*\Gamma}$. Recall the internal map $\theta: \L^{\infty}({}^*G,\W) \to \L^{\infty}({}^*D,\W)$ as defined in \cref{theta}. 
	\begin{lemma}\label{qinv2}
		Given $\omega' \in \L^{\infty}_{b}(({}^*Q)^3,\V)^{\sim {}^*Q}$ with $Im(\omega') \subseteq \V^{\sim {}^*G}$, there exists $$\omega'' \in \L^{\infty}_{b}(({}^*Q)^3,\L^{\infty}({}^*G,\W))^{{}^*Q}$$ with $Im(\omega'') \subseteq \left(\L^{\infty}({}^*G,\W)^{\sim {}^*\Gamma}\right)^{{}^*G} = \W_b^{\sim {}^*\Gamma}$ such that $\omega'-\theta \cdot \omega'' \in \L^{\infty}_{inf}(({}^*Q)^3,\V)$.
	\end{lemma}
	\begin{proof}
		This follows from \cref{WbVG} and \cref{almostGequi2}. 
	\end{proof}
	Now consider $\omega'' \in \L^{\infty}_{b}(({}^*Q)^3,\W)^{{}^*Q}$ with $Im(\omega'') \subseteq \W_b^{\sim {}^*\Gamma}$, and observe that 
	$$d^3\omega'' \in \L^{\infty}_{inf}\left(({}^*Q)^4, \W \right)$$
	That is, $\omega''$ can be thought of as an almost $2$-cocycle for $Q$ with coefficients in $\W$ (with a trivial action of ${}^*Q$). The underlying idea is that we are now within the domain of classical bounded cohomology of $Q$ with coefficients in a trivial $Q$-module, and can apply the results of \cref{trivialRV3}. 
	\begin{lemma}\label{qinv3}
		Suppose $\Hb^2(Q,\R)=0$ and $\Hb^3(Q,\R)$ is Hausdorff. Then there exists $\alpha' \in  \L^{\infty}_b\left(({}^*Q)^2,\W \right)^{{}^*Q}$ with $Im(\alpha') \subseteq \W_b^{\sim {}^*\Gamma}$ and $d^2\alpha' - \omega'' \in \L^{\infty}_{inf}\left(({}^*Q)^3, \W \right)$.
	\end{lemma}
	\begin{proof}
		From \cref{trivialRV3}, we know that there exists $\alpha' \in  \L^{\infty}_b\left(({}^*Q)^2,\W \right)^{{}^*Q}$ such that $d^2\alpha' - \omega'' \in \L^{\infty}_{inf}\left(({}^*Q)^3, \W \right)$. We only need to show that $Im(\alpha') \subseteq \W_b^{\sim {}^*\Gamma}$. That is, we want to show that for an internal cochain $\alpha' \in  \L^{\infty}_b\left(({}^*Q)^2,\W \right)^{{}^*Q}$, if $Im(d^2\alpha') \subseteq \W_b^{\sim {}^*\Gamma}$, then $Im(\alpha') \subseteq \W_b^{\sim {}^*\Gamma}$.\\
		Equivalently, consider the inhomogenous cochain (refer \cref{inhomo}) corresponding to $\alpha'$, namely $f \in  \L^{\infty}_b\left({}^*Q,\W \right)$ defined as follows: for $g \in {}^*Q$, $f(g)$ is the essential value $\alpha'(x,xg)$. Clearly, $Im(d^2\alpha')=Im(\delta^1 f)$ and $Im(\alpha')=Im(f)$. Hence it is sufficient to show that if $f \in \L^{\infty}_b\left({}^*Q,\W \right)$ such that $Im(\delta^1 f) \subseteq \W_b^{\sim {}^*\Gamma}$, then $Im(f) \subseteq \W_b^{\sim {}^*\Gamma}$. Note that $\tilde{\W}^{{}^*\Gamma}$ is a closed subspace of the real Banach space $\tilde{\W}$, hence $\tilde{f} \in \tilde{\L}^{\infty}({}^*Q,\tilde{\W})$ is such that $\tilde{\delta}^1\tilde{f}$ is $0$ in the quotient Banach space $\tilde{\W}/\tilde{\W}^{{}^*\Gamma}$. We would like to show that $\tilde{f}$ is $0$ in $\tilde{\W}/\tilde{\W}^{{}^*\Gamma}$.\\
		Consider the image, denoted $S$, of $Im(f) \subseteq \W_b$ in the quotient $\tilde{\W}/\tilde{\W}^{{}^*\Gamma}$ with the natural map $\W_b \to \tilde{\W}/\tilde{\W}^{{}^*\Gamma}$ denoted $w \mapsto \tilde{w}'$. Let $C=\sup \{\|v\|: v \in S\}$. Since $f \in \L^{\infty}_b\left({}^*Q,\W \right)$, we know that $C \in \R$. Let $w \in Im(f)$ be such that $\|\tilde{w}'\|\geq 0.9C$, and consider the ball $B_{0.1 C}(\tilde{w}')$ of radius $0.1C$ around $\tilde{w}'$, and let $Y=\{Y_n\}_{\U} \subseteq {}^*Q$ be an internal subset of ${}^*Q$ of positive (internal) measure such that the image of $\tilde{f}(Y)$ is in $B_{0.1 C}(\tilde{w}')$ (such an internal subset $Y$ of positive measure exists because the image in $\tilde{\W}/\tilde{\W}^{{}^*\Gamma}$ of the internal ball $B_{0.1C}(w) \subseteq \W_b$ is contained in $B_{0.1 C}(\tilde{w}')$, and $f^{-1}(B_{0.1C}(w))$ is of positive measure).\\
		Since $\tilde{\delta}^1\tilde{f}$ is $0$ in the quotient Banach space $\tilde{\W}/\tilde{\W}^{{}^*\Gamma}$, we know that there exists an internal subset $A=\{A_n\}_{\U} \subseteq {}^*Q \times {}^*Q$, where $A_n$ is co-null in $Q \times Q$ for $n \in \U$, such that for $(g,h) \in A$, $\tilde{f}(g)+\tilde{f}(h)-\tilde{f}(gh) \in \tilde{\W}^{{}^*\Gamma}$. Note that $A \cap Y \times Y$, denoted $D$, is an internal subset of positive (internal) measure in ${}^*Q \times {}^*Q$, and so is the product $D\cdot D \coloneqq \{xy: (x,y) \in D\} \subseteq {}^*Q$. It follows that the image $v \in \tilde{\W}/\tilde{\W}^{{}^*\Gamma}$ of $\tilde{f}(g)$ has norm $\|v\|\geq 1.6C$, for $g \in D\cdot D$. This implies that $C=0$, allowing us to conclude that $S \subseteq \tilde{\W}^{{}^*\Gamma}$ and hence $\tilde{f}$ is $0$ in $\tilde{\W}/\tilde{\W}^{{}^*\Gamma}$.
	\end{proof}
	
	We now combine the results above to conclude that $\Ha^2(Q,\V)=0$.
	\begin{theorem}\label{Qtheorem}
		Let $G$ and a parabolic subgroup $Q \leq G$ be as in \cref{rapinchuk2}, and suppose $\Hb^3(Q,\R)$ is Hausdorff and $\Hb^2(Q,\R)=0$. Then $\Ha^2(Q,\V)=0$.
	\end{theorem}
	\begin{proof}
		From \cref{Qinv1}, we know that exists an asymptotic $2$-cocycle $\omega_1 \in \L^{\infty}_{b}(({}^*Q)^3,\V)^{\sim {}^*Q}$ for $Q$ that takes values in $\V^{\sim {}^*G}$ such that $\tilde{\omega} - \tilde{\omega}' = \tilde{d}^1\tilde{\alpha}_1$ where $\alpha_1 \in \L^{\infty}_{b}(({}^*Q)^2,\V)^{\sim {}^*Q}$. From \cref{qinv2} and \cref{qinv3},  we conclude that there exists $\alpha_1' \in \L^{\infty}_{b}(({}^*Q)^2,\V)^{\sim {}^*Q}$ such that $\tilde{\omega}' = \tilde{d}^1\tilde{\alpha_1'}$. We now set $\alpha=\alpha_1+\alpha_1'$. 
	\end{proof}
	
	\subsection{Property-$G(\mathcal{Q}_1,\mathcal{Q}_2)$ and the Main Theorem}\label{ssec-mainproof}
	Recall that the results of \S\ref{ssec-invariants} assumed the existence of a parabolic subgroup $Q=\prod_{i=1}^k Q_i$ (where each $Q_i$ is a proper parabolic subgroup of ${\GG_i}(K_i)$), satisfying the conditions that $\Hb^3(Q,\R)$ is Hausdorff and $\Hb^2(Q,\R)=0$. This allowed us to prove that $\Ha^2(Q,\V)$ vanishes. This motivates the following definition:
	\begin{definition}
		A locally compact group $G$ has the \textbf{2\textonehalf-property} if $\Hb^2(G, \R)$ vanishes and $\Hb^3(G, \R)$ is Hausdorff.
	\end{definition}
	In order to demystify this definition, we should consider that it is a natural strengthening of the vanishing of $\Hb^2(G, \R)$. Indeed, recall that $\Hb^1(G, \R)$ always vanishes and that the Hausdorff condition on $\Hb^3(G, \R)$ means that the differential map on two-cochains is an open map. Thus the 2\textonehalf-property states that the augmented differential complex computing bounded cohomology starts of as an exact sequence up to degree two and retains a weaker consequence of exactness in degree three, namely that the differential is open (see \cite{matsu} for a detailed discussion of the openness of differentials in chain complexes of Banach spaces).
	\begin{proposition}\label{thm-2half3}
		The 2\textonehalf-property is preserved under extensions, and hence in particular, under finite direct products.
	\end{proposition}
	\begin{proof}
		This follows from the Hochschild-Serre spectral sequence as set up in ~\cite[\S12]{bookMonod}. Specifically, the Hausdorff assumption allows us to apply Proposition~12.2.2 in~\cite[\S12]{bookMonod}. 
	\end{proof}
	Furthermore, the results on Hochschild-Serre spectral sequences in \cite[\S12]{bookMonod} imply that a direct product of groups has the 2\textonehalf-property iff each factor has the 2\textonehalf-property.\\ 
	We recall that elementary properties of bounded cohomology allow us to disregard ``amenable pieces'' when it comes to the 2\textonehalf-property:
	\begin{lemma}\label{thm-2half2}
		\begin{enumerate}
			\item Suppose that $\widetilde G$ is an extension of $G$ by an amenable kernel, for instance a central extension of $G$. Then $\widetilde G$ has the 2\textonehalf-property if and only if $G$ does.
			\item Suppose that $G_1<G$ is a closed co-amenable subgroup of $G$, for instance a closed normal subgroup with amenable quotient. If $G_1$ has the 2\textonehalf-property, then so does $G$.
		\end{enumerate}
	\end{lemma}
	\begin{proof}
		The first point follows from the fact that the inflation map $\Hb^\bullet(G, \R) \to \Hb^\bullet(\widetilde G, \R)$ is an isometric isomorphism, see e.g.~\cite[8.5.2]{bookMonod}. For the second point, we recall that the restriction map $\Hb^\bullet(G, \R) \to \Hb^\bullet(G_1, \R)$ is isometrically injective, see e.g.~\cite[8.6.6]{bookMonod}. 
	\end{proof}
	So in \cref{invariants}, we considered a semisimple group $G = \prod^k_{i = 1} {\GG_i}(K_i)$ (where for $1\leq i \leq k$, $K_i$ is a local field, and $\GG_i$ is a connected, simply connected, almost $K_i$-simple group) and showed that $\Ha^2(Q,\V)=0$ for a parabolic subgroup $Q=\prod_{i=1}^k Q_i$ (where each $Q_i$ is a proper parabolic subgroup of ${\GG_i}(K_i)$) assuming $Q$ has the 2\textonehalf-property. To use this to prove that $\Ha^2(G,\V)=0$, we need the existence of \emph{two} such parabolic subgroups $Q_1$ and $Q_2$ both containing $P$ and boundedly generating $G$. This is described in the following definition:
	\begin{definition}\label{GQ1Q2}
		Let $G = \prod^k_{i = 1} {\GG_i}(K_i)$ be a semisimple group (where for $1\leq i \leq k$, $K_i$ is a local field, and $\GG_i$ is a connected almost $K_i$-simple group), and let $P$ be a minimal parabolic subgroup. Then $G$ is said to have \textbf{Property-$G(\mathcal{Q}_1,\mathcal{Q}_2)$} if there exist two parabolic subgroups $Q_1$ and $Q_2$ of $G$ satisfying the following properties:
		\begin{itemize}
			\item Both $Q_1$ and $Q_2$ are of the form $Q_1=\prod_{i=1}^k Q_{1,i}$ and $Q_2=\prod_{i=1}^k Q_{2,i}$ with $Q_{1,i}$ and $Q_{2,i}$ being proper parabolic subgroups of $\GG_i(K_i)$ for $1\leq i\leq k$.
			\item Both $Q_1$ and $Q_2$ have the 2\textonehalf-property.
			\item The intersection $Q_1 \cap Q_2$ is a parabolic subgroup that contains the minimal parabolic subgroup $P$ of $G$.
			\item The group $G$ is boundedly generated by the union of $Q_1$ and $Q_2$.
		\end{itemize}
	\end{definition}
	Observe that the existence of two distinct proper parabolic subgroups immediately implies that if $G$ has Property-$G(\mathcal{Q}_1,\mathcal{Q}_2)$, then each of its simple factors  has rank at least $2$. We shall see explicit examples of groups with the 2\textonehalf-property and Property-$G(\mathcal{Q}_1,\mathcal{Q}_2)$ in \S\ref{ssec-mainproof}.\\
	\begin{theorem}\label{maintheorem1}
		Let $\Gamma$ be a lattice in a semisimple group $G= \prod^k_{i = 1} {\GG_i}(K_i)$ (where for $1\leq i \leq k$, $K_i$ is a local field, and $\GG_i$ is a connected, simply connected, almost $K_i$-simple group) that has Property-$G(\mathcal{Q}_1,\mathcal{Q}_2)$. Then $\Ha^2(\Gamma,\W)=0$.
	\end{theorem}
	\begin{proof}
		By \cref{induction}, $\Ha^2(\Gamma,\W)=\Ha^2(G,\V)$. Consider an asymptotic $2$-cocycle $$\omega \in \L^{\infty}_{b}(({}^*(G/P))^2,\V)^{\sim {}^*G}$$
		By \cref{Qtheorem}, there exist $\alpha_1 \in \L_b^{\infty}(({}^*(G/P))^2,\V)^{\sim {}^*Q_1}$ and $\alpha_2 \in \L_b^{\infty}(({}^*(G/P))^2,\V)^{\sim {}^*Q_2}$ such that
		$$\tilde{\omega}=\tilde{d}^2\tilde{\alpha}_1=\tilde{d}^2\tilde{\alpha}_2$$
		Since $P \subseteq Q_1 \cap Q_2$, note that $\alpha_1-\alpha_2 \in \L_b^{\infty}(({}^*(G/P))^2,\V)^{\sim {}^*P}$ is an asymptotic $1$-cocyle for $P$. As $P$ is amenable, $\Ha^1(P,\V)=0$, and so there exists $\beta \in \L_b^{\infty}({}^*(G/P),\V)^{\sim {}^*P}$ such that
		$$\tilde{\alpha}_1-\tilde{\alpha}_2 = \tilde{d}^1\tilde{\beta}$$
		We now use $\beta$ to define the internal map
		$$\beta_1:({}^*(G/P))^2 \to \V$$
		$$\beta_1(g{}^*P,h{}^*P) \coloneqq \pi_G(g) \beta(g^{-1}h {}^*P)$$
		Note that $\beta_1 \in \L_b^{\infty}(({}^*(G/P))^2,\V)^{\sim {}^*G}$, and by \cref{lemma-ergodicity1}, there exists $F \in \V_b^{\sim {}^*G}$ so that $\beta_1-F \in \L_{inf}^{\infty}(({}^*(G/P))^2,\V)$. In particular, the same holds for $\beta$ as well: $\beta - F \in \L_{inf}^{\infty}({}^*(G/P),\V)$, implying that
		$$\tilde{d}^1\tilde{\beta}=0$$
		This means that $\tilde{\alpha}_1=\tilde{\alpha}_2$. Setting $\alpha \coloneq \alpha_1$, note that $\alpha$ is both asymptotically ${}^*Q_1$-equivariant and asymptotically ${}^*Q_2$-equivariant. Since $G$ is boundedly generated by elements of $Q_1$ and $Q_2$, this implies that $\alpha \in \L_b^{\infty}(({}^*(G/P))^2,\V)^{\sim {}^*G}$. 
	\end{proof}
	While the above theorem assumes that $\Gamma$ is a lattice in a semisimple group $G$ that is simply connected, we can extend this as follows:
	\begin{theorem}
		Let $\Gamma$ be a lattice in a semisimple group $G= \prod^k_{i = 1} {\GG_i}(K_i)$ (where for $1\leq i \leq k$, $K_i$ is a local field, and $\GG_i$ is a connected almost $K_i$-simple group) that has Property-$G(\mathcal{Q}_1,\mathcal{Q}_2)$. Then $\Gamma$ is uniformly $\mathfrak{U}$-stable with a linear estimate. 
	\end{theorem}
	\begin{proof}
		Let $\hat{G}$ be the universal cover of $G$, which is a semisimple simply connected group, and let $\hat{\Gamma}$ be the preimage of $\Gamma$ in $\hat{G}$. Note that $\Gamma$ and $\hat{\Gamma}$ are commensurable, hence \cref{remark-comm} tells us that it is sufficient to prove that $\hat{\Gamma}$ is is uniformly $\mathfrak{U}$-stable with a linear estimate. Note that since the fundamental group of $G$ is abelian, from \cref{thm-2half2} it follows that $\hat{G}$ has Property-$G(\mathcal{Q}_1,\mathcal{Q}_2)$. Hence we can now apply \cref{maintheorem1} to $\hat{\Gamma} \leq \hat{G}$ to conclude that $\Ha^2(\hat{\Gamma},\W)=0$, implying that $\hat{\Gamma}$ is uniformly $\mathfrak{U}$-stable with a linear estimate.  
	\end{proof}
	
	\subsection{Groups with Property-$G(\mathcal{Q}_1,\mathcal{Q}_2)$}\label{ssec-semisimple}
	In this subsection, we shall list classes of groups that satisfy Property-$G(\mathcal{Q}_1,\mathcal{Q}_2)$ used in the hypothesis of \cref{maintheorem1}. Since Property-$G(\mathcal{Q}_1,\mathcal{Q}_2)$ involves the existence of two proper parabolic subgroups with the 2\textonehalf-property, we first list classes of groups known to have this property.
	\subsubsection*{Simple Groups with the 2\textonehalf-property}\label{ssec-2half}
	We shall collect below the necessary statement from the existing literature, and complement them by an additional argument in the non-Archimedan case, to establish that a number of natural semisimple groups have the 2\textonehalf-property.\\
	Let us first consider simple groups over a non-archimedean field.  Here, the vanishing of $\Hb^2(G, \R)$ was established in~\cite{burgerMonod}. More precisely, this reference establishes the injectivity of the comparison map (and we recall that the case of trivial coefficients $\R$ for the semisimple group $G$ was actually the easy part of this result). On the other hand, the vanishing of the usual cohomology is well-known in this setting, see e.g.~\cite{borelwallach}.\\
	Therefore, what we need to justify is the condition on $\Hb^3(G, \R)$. It turns out that this vanishes: this result is established by the inductive method introduced in~\cite{monod2010}, the basis of the induction being provided by~\cite{buchermonod}.
	\begin{theorem}
		For $G=\GG(K)$, where $K$ is a non-Archimedean local field and $\GG$ is a connected, simply connected, semisimple $K$-group, the bounded cohomology $\Hb^3(G, \R)$ vanishes (and is therefore Hausdorff).
	\end{theorem}
	\begin{proof}
		The proof is by induction on the $K$-rank of $\GG$. The case of rank zero corresponds to $G=\GG(K)$ being compact, and thus having trivial bounded cohomology in all degrees.  The induction really starts with rank one. In that case, $G$ is an automorphism group of a Bruhat--Tits tree satisfying the assumptions of the main result of~\cite{buchermonod}, which states that $\Hb^n(G, \R)$ vanishes for all $n>0$.\\
		We now perform the induction step for $\GG$ of $K$-rank $r\geq 2$. We use (a minor variation of) the spectral sequence introduced in~\cite[\S6.A]{monod2010} as follows. We consider the Tits building $T$ of $\GG$ over $K$ and denote by $T^{(d)}$ its set of $d$-simplices. Thus $T^{(d)}$ is defined for all $d\leq r-1$ and is topologized by identifying it with the union of homogeneous spaces $G/P_I$, where $P_I$ ranges over standard parabolic subgroups of semisimple rank $r-1-d$ (see ~\cite{monod2010} for more details). By convention, we take the \emph{augmented} simplicial complex, namely we also consider the one-point space of negative simplices $T^{(-1)} = G/G$. In this set-up, an appropriate version of the Solomon-Tits theorem implies that the following sequence of spaces of continuous functions is exact:
		$$0 \to C(T^{(-1)}) \to  C(T^{(0)}) \to \cdots  \to C(T^{(r-1)})$$
		see Theorems~2.7 and~3.9 in~\cite{monod2010}. We denote by $\mathrm{St}_G$ the cokernel of the last map above, which is one version of the Steinberg representation of $G$. Finally, the spectral sequence that we consider is the first quadrant hypercohomology spectral sequence obtained by computing the bounded cohomology of $G$ with coefficients in the following complex of Banach $G$-modules
		\begin{equation}\label{treeeq1}
		0 \to C(T^{(-1)}) \to  C(T^{(0)}) \to \cdots  \to C(T^{(r-1)}) \to \mathrm{St}_G \to 0
		\end{equation}
		Concretely, the first page of the spectral sequence is by definition $E^{p,q}_1 = \Hb^q\left(G, C(T^{(p-1)}) \right)$ when $p\leq r$, for $p=r+1$ it is $E^{r+1,q}_1 = \Hb^q(G, \mathrm{St}_G)$, and for $p>r+1$ it is zero. The only difference with~\cite[\S6.A]{monod2010} is that the complex was truncated at $T^{(r-1)}$ there, not completing it with $ \mathrm{St}_G$. A crucial point to allow this definition is that $\mathrm{St}_G$ is Hausdorff, which follows from an alternative description of the Steinberg representation given by Borel--Serre, see Remark~2.8 in~\cite{monod2010}. 
		
		This spectral sequence abuts to zero since the above complex is exact. By cohomological induction, we have as in Theorem~6.1 of~\cite{monod2010} the identifications
		\begin{equation}\label{treeeqn2}
		E^{p,q}_1 = \Hb^q\left(G, C(T^{(p-1)}) \right) \cong \bigoplus \Hb^q (G_I, \R) \kern7mm (\forall\,p\leq r, \forall q)
		\end{equation}
		where $G_I$ is the semisimple part of the parabolic subgroup $P_I$ and the sum is taken over all such subgroups of semisimple rank $r-p$. Notice that the rank is indeed $r-p$ and not $r-1-p$ since we started with $T^{(-1)}$ corresponding to $p=0$.
		
		Our goal is to prove $E^{0,3}_1=0$ since this is $\Hb^3 (G, \R)$. We have $E^{0,3}_1=E^{0,3}_2$ since the right hand side is the kernel of the map $E^{0,3}_1 \to E^{1,3}_1$ which vanishes by \cref{treeeqn2} and the inductive hypothesis. Next, $E^{0,3}_2=E^{0,3}_3$ because the right hand side is the kernel of the map $E^{0,3}_2 \to E^{2,2}_2$ and already $E^{2,2}_1$ vanishes by \cref{treeeqn2} and the general vanishing of $\Hb^2$ mentioned earlier.
		
		The next differential to consider is $E^{0,3}_3 \to E^{3,1}_3$. If $r\geq 3$, we can still apply \cref{treeeqn2} and the general vanishing of $\Hb^1$ to conclude $E^{0,3}_3=E^{0,3}_4$. If however $r=2$, then we reach the same conclusion provided we justify that $\Hb^1(G, \mathrm{St}_G)$ vanishes. To this end, we first observe that the comparison map to ordinary cohomology is always injective in degree one. On the other hand, the ordinary first cohomology of $G$ with values in the Steinberg representation is known to vanish if (and only if!) the rank of $G$ is not one, see e.g. Theorem~4.12 p.~205 in~\cite{borelwallach}. Thus, in either case we have established $E^{0,3}_3=E^{0,3}_4$.
		
		The final differential is $E^{0,3}_4 \to E^{4,0}_4$. But we have already second page vanishing of all $E^{p,0}_2$. Indeed, by \cref{treeeqn2} this statement amounts to the acyclicity of the subcomplex of $G$-invariants of \cref{treeeq1}. Ignoring the last term $( \mathrm{St}_G)^G$ at first, this comes from the fact that the $G$-orbits in the Tits building form by definition a full simplex of dimension $r-1$ (augmented in dimension~$-1$), which is hence acyclic. It remains to argue acyclicity of the last term
		$$C(T^{(r-1)})^G \to (\mathrm{St}_G)^G \to 0$$
		which amounts to showing that $( \mathrm{St}_G)^G$ vanishes. This follows e.g. by realizing the Steinberg representation as a space of $L^2$ harmonic maps on the Bruhat--Tits building, see~\cite{klingler} for a concrete description of this isomorphism (also due to Borel--Serre).
		
		In conclusion, we have shown that $E^{0,3}_1=E^{0,3}_4 = E^{0,3}_\infty$ holds. Since the spectral sequence converges to zero, this establishes as desired the vanishing of $E^{0,3}_1=\Hb^3 (G, \R)$. 
	\end{proof}
	Next, let us consider simple groups over $\R$ or $\C$. Regarding $\Hb^2(G, \R)$, the injectivity of the comparison map was established for all connected semisimple Lie groups $G$ in~\cite{burgerMonod}. Therefore, the corresponding vanishing holds in all cases where the usual second bounded cohomology vansihes, which is always the case for $G =SL_n(\C)$ and is the case for $G =SL_n(\R)$ if and only if $n\neq 2$.
	It therefore remains to collect the following results from the existing literature:
	\begin{theorem}
		Let $n\geq 2$.
		\begin{enumerate}
			\item For $G=SL_n(\R)$, the bounded cohomology $\Hb^3(G, \R)$ vanishes (and is therefore Hausdorff).
			\item For $G=SL_n(\C)$, the bounded cohomology $\Hb^3(G, \R)$ is one-dimensional and Hausdorff.
		\end{enumerate}
	\end{theorem}
	\begin{proof}
		The case of $SL_n(\R)$ was established for $n=2$ in Theorem~1.5 of~\cite{burgermonod2} and for general $n$ in Theorem~1.2 of~\cite{monod2004stabilization}.
		
		Regarding $G=SL_n(\C)$, these two references established that the comparison map is an isomorphism from $\Hb^3(G, \R)$ to $H^3(G, \R)$, the latter being classically known to be one-dimensional (see~\cite[Theorem~1.2]{burgermonod2} and~\cite[Remark~3.5]{monod2004stabilization}). In our context, the only relevant point (and the hard one anyway) is that the comparison map is injective. Indeed, it is a general fact for any group and any degree $d$ that the injectivity of the comparison map $\Hb^d \to H^d$ implies that $\Hb^d$ is Hausdorff. This can be seen by combining Theorems~2.3 and~2.8 in~\cite{matsu}. 
	\end{proof}
	We now summarize the above results with the following list of simple groups that we know to have the 2\textonehalf-property:
	\begin{theorem}\label{thm-2half}
		The following groups $G$ have the 2\textonehalf-property:
		\begin{enumerate}
			\item $G=\GG(k)$, where $k$ is a non-Archimedean local field and $\GG$ is a connected semisimple $k$-group.
			\item $G =SL_n(\R)$ for any $n\neq 2$.
			\item  $G =SL_n(\C)$ for any $n$.
		\end{enumerate}
	\end{theorem}

	\subsubsection*{From the 2\textonehalf-property to  Property-$G(\mathcal{Q}_1,\mathcal{Q}_2)$}
	We can now list out simple groups having Property-$G(\mathcal{Q}_1,\mathcal{Q}_2)$ using the results discussed earlier in this section. Again, we begin with simple groups, and then extend the results to semisimple groups using \cref{thm-2half3}.\\
	Let us first consider the case of a simple group $G$ over a non-archimedean field, and $Q\leq G$ be a parabolic subgroup. Note that, using \cref{thm-2half}, \cref{thm-2half2} and \cref{thm-2half3}, we can conclude that  $Q$ has the 2\textonehalf-property (since, modulo its amenable radical, it is a semisimple group). Thus, since $G$ has rank at least $2$, we can always find two proper parabolic subgroups, both having the 2\textonehalf-property, generating $G$. Hence,
	\begin{proposition}
		The group $G=\GG(k)$, where $k$ is a non-Archimedean local field and $\GG$ is a connected, simply connected, semisimple $k$-group of rank at least $2$, has Property-$G(\mathcal{Q}_1,\mathcal{Q}_2)$. 
	\end{proposition}
	In the complex case, since we know that $SL_n(\C)$ has the 2\textonehalf-property for every $n$, we can use Dynkin diagrams to explicitly construct proper parabolic subgroups (which, modulo their amenable radical, would correspond to $SL_n(\C)$ that we know has the 2\textonehalf-property) that together generate the group, to conclude that the simple group has Property-$G(\mathcal{Q}_1,\mathcal{Q}_2)$.
	\begin{proposition}\label{complexlie}
		Let $G$ be a simple, simply connected, complex Lie group of rank $n \geq 2$. Then $G$ has Property-$G(\mathcal{Q}_1,\mathcal{Q}_2)$.
	\end{proposition}
	\begin{proof}
		Recall that simple, simply connected Lie groups over $\C$ split, and hence are classified by Dynkin diagrams, so let $X$ be the Dynkin diagram of $G$. Any subdiagram $Y$ of $X$ corresponds to a parabolic subgroup $Q \leq G$ containing a fixed minimal parabolic subgroup $P \leq G$, such that $Q$ modulo its amenable radical is a semisimple group corresponding to the subdiagram $Y$. Furthermore, if the subdiagram $Y$ is of type $A_n$ (for $n \geq 1$), then the parabolic subgroup $Q$ corresponding to $Y$, modulo its amenable radical, is $SL_{n+1}(\C)$, and so, by \cref{thm-2half} and \cref{thm-2half2},  $Q$ has the 2\textonehalf-property.\\
		So constructing two proper parabolic subgroups $Q_1$ and $Q_2$, that both contain $P$ and generate $G$, is equivalent to choosing two proper subdiagrams $Y_1$ and $Y_2$ of $X$ that the union of the vertex sets of $Y_1$ and $Y_2$ is the vertex set of $X$. Furthermore, if we can ensure both these diagrams are of type $A_n$, then this gives us two proper parabolic subgroups with the 2\textonehalf-property, ensuring that $G$ has Property-$G(\mathcal{Q}_1,\mathcal{Q}_2)$.
		We claim that for any connected Dynkin diagram $X$ corresponding to a simple complex Lie algebra of rank at least $2$, we can find two such subdiagrams $Y_1$ and $Y_2$ both of type $A_n$ (for $n \geq 1$). This can be seen by considering each case separately, and is illustrated in the figure below. Observe that for the Dynkin diagram of type $F_4$, we construct subdiagrams $Y_1$ and $Y_2$ both of type $A_2$, while for a Dynkin diagram of any other type, we can always choose $Y_1$ and $Y_2$ to be of type $A_1$ and $A_{m-1}$ (where $m$ is the rank of the group $G$). 
	\end{proof} 
	\begin{center} 
		\includegraphics[scale=0.30]{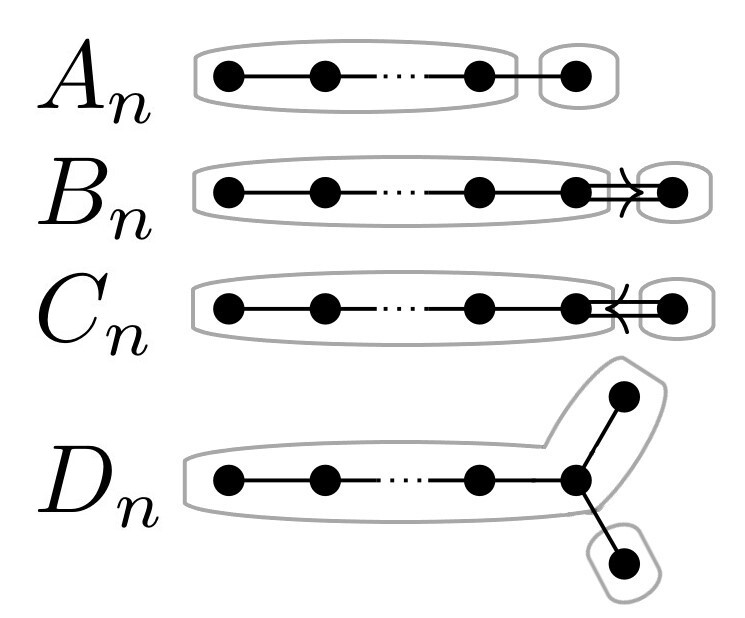} \quad \quad \quad
		\includegraphics[scale=0.25]{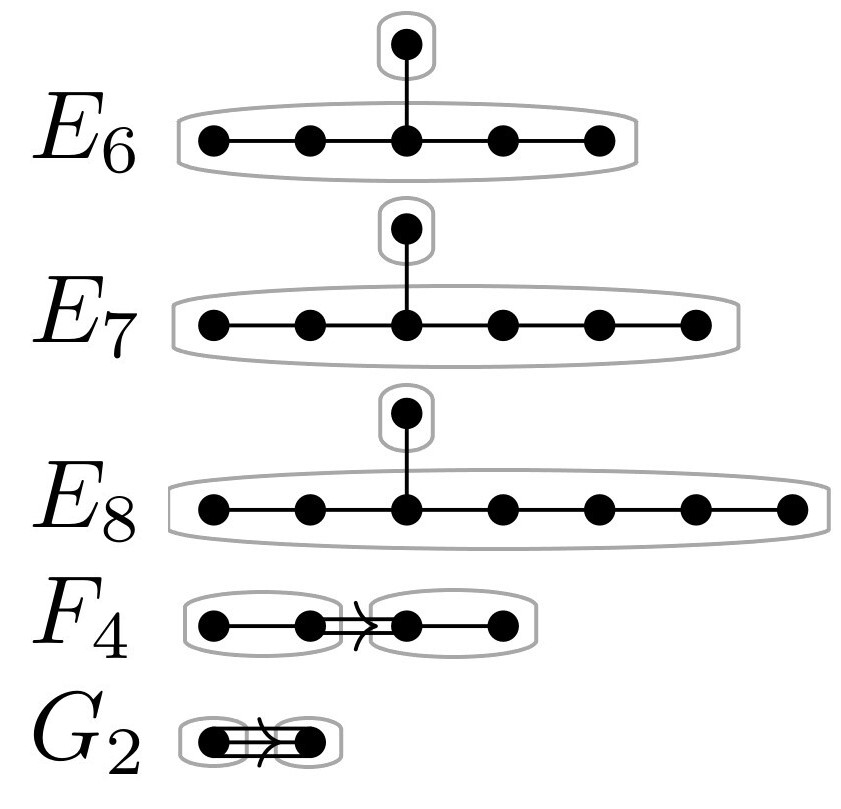}  
	\end{center} 
	We next consider the case of connected, simply connected groups over the reals, where the situation is more involved. Firstly, note that even $SL_3(\R)$ (which, by \cref{thm-2half}, has the 2\textonehalf-property) does not have Property-$G(\mathcal{Q}_1,\mathcal{Q}_2)$ even though it has rank $2$, because any proper parabolic subgroup of $SL_3(\R)$, modulo its amenable radical, is $SL_2(\R)$ for which $\Hb^2(SL_2(\R),\R) \neq 0$. However, since $SL_n(\R)$ has the 2\textonehalf-property for $n \geq 3$, we can apply the proof technique of \cref{complexlie} in the case of certain \emph{split} simple real Lie groups to show that:
	\begin{proposition}
		Let $G$ be a split simple real Lie group of type $A_n$, $D_n$ (for $n \geq 3$), $F_4$, $E_6$, $E_7$ or $E_8$.Then $G$ has Property-$G(\mathcal{Q}_1,\mathcal{Q}_2)$. 
	\end{proposition}
	\begin{proof}
		As in the proof of \cref{complexlie}, we construct two subdiagrams $Y_1$ and $Y_2$ of the Satake diagram $X$ of the split simple real group $G$, such that the subdiagrams whose vertex sets together cover the vertex set of $X$, such that the simple real Lie groups corresponding to the subdiagrams are both $SL_n(\R)$ for some $n \geq 3$. This is illustrated in the figure below, where the subdiagrams are encircled in grey. 
	\end{proof}
	\begin{center} 
		\includegraphics[scale=0.42]{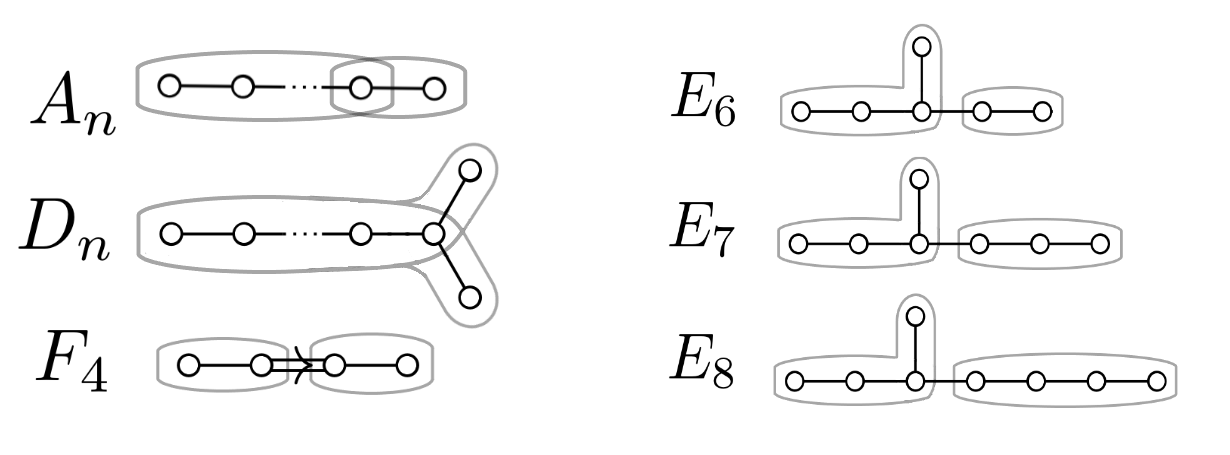}
	\end{center} 
	Note that our method will not work for the split simple real Lie groups of type $C_n$ or $G_2$. In the case of $C_n$, for any two two proper subdiagrams covering the vertices of the Satake diagram, at least one of them must either itself be of type $C_m$, which corresponds to the simple group $Sp(2m,\R)$ which does \emph{not} have the 2\textonehalf-property (as $\Hb^2(Sp_{2m}(\R),\R)\neq 0$), or of type $A_1$, which corresponds to the simple group $SL_2(\R)$. In the case of $G_2$, any proper parabolic subgroup is of type $A_1$. Thus, split simple real Lie groups of type $C_n$ or $G_2$ \emph{do not} have Property-$G(\mathcal{Q}_1,\mathcal{Q}_2)$.\\
	
	\section{Conclusions and Discussion}\label{sec-conclusions}
	The current paper presents many questions and directions for futher research. We highlight a few of them now.
	\begin{itemize}
		\item \textbf{Prove a complete Ulam-stability result for higher rank lattices in semisimple Lie groups}.\\
		One could hope to extend our main results (specifically \cref{maintheorem1}) by considering a larger family of groups for which the Property-$G(\mathcal{Q}_1,\mathcal{Q}_2)$ is known. However, this has its limitations since there are groups for which Property-$G(\mathcal{Q}_1,\mathcal{Q}_2)$ is simply not true. An important example of such a group is $SL_3(\R)$, where for a maximal parabolic subgroup $Q \leq SL_3(\R)$, $\Hb^2(Q,\R)=\Hb^2(SL_2(\R),\R) \neq 0$. We can ask if Property-$G(\mathcal{Q}_1,\mathcal{Q}_2)$ for the ambient group is even necessary for Ulam stability, or if it just happens to be an artifact of our proof technique.\\
		\item \textbf{Surjectivity/Injectivity of a comparison map $\Ha^2(\Gamma,\W) \mapsto \Hb^2(\Gamma,\tilde{\W})$}.\\
		There exists a natural forgetful map $\Ha^2(\Gamma,\W) \mapsto \Hb^2(\Gamma,\tilde{\W})$ (analogous to the comparison map $c:\Hb^2(\Gamma,W) \to H^2(\Gamma,W)$ for a $\Gamma$-module $W$). It is not immediate if this map is either surjective or injective. Suppose it were injective, then we could truly reduce the question of uniform stability with a linear estimate to the study of the second bounded cohomology group $\Hb^2(\Gamma,\tilde{W})$ which is a well-studied notion. For instance, it is known (\cite{burgerMonod}) that for a lattice $\Gamma$ in a higher rank simple Lie group $G$, $\Hb^2(\Gamma,W)=0$ for every dual separable Banach $\Gamma$-module $W$ (note, however, that the Banach space ultraproduct $\tilde{W}$ is not separable unless it is finite dimensional).\\
		\item \textbf{Uniform versus non-uniform stability}\\
		All along in this paper, we have dealt with the question of uniform stability as opposed to non-uniform, or pointwise, stability. The connection between pointwise  stability and vanishing second cohomology is studied in \cite{DCGLT}, \cite{oppen}. We stress that such stability results in the non-uniform setting is far from being known for most lattices. So far such results are known for many lattices in $p$-adic Lie groups (with respect to the Frobenius or $p$-Schatten norms for $p < \infty$), but almost nothing is known for lattices in real (or complex) Lie groups (see \cite{bader}). Moreover, when the family $M_n(\C)$ is endowed with the operator norm (i.e. the $p$-Schatten norm for $p=\infty$) then it is known that the stability result is \emph{not} true for most hgh rank lattices. This follows from the results in \cite{dadarlat} \cite{moscovici} that show that if $H^{2j}(\Gamma,\R) \neq 0$ for some positive integer $j$, then $\Gamma$ is not pointwise stable for the operator norm. Note that in our setting of uniform stability, the case of $p=\infty$ and $p <\infty$ are treated together without any problem.\\
		\item \textbf{Stability with respect to the (normalized) Hilbert-Schmidt norm}.\\
		In \cite{becker} it is shown that a lattice $\Gamma$ that has Property (T) is not pointwise stable for the (normalized) Hilbert-Schmidt norm. In \cite{dogon} it is shown that if a residually finite group is uniformly stable with respect to the (normalized) Hilbert-Schmidt norm, then it is virtually abelian. In particular, this means that no lattice in a non-compact semisimple Lie group is uniformly stable with respect to the (normalized) Hilbert-Schmidt norm.\\
		In both cases, the results still leave the possibility that higher rank lattices are \emph{flexibly} stable (pointwise or uniform) with respect to the (normalized) Hilbert-Schmidt norm. For more on flexible stability, refer \cite{beckerchapman} and \cite{becker}.\\
		\item \textbf{ Uniform stability with non-linear estimate}.\\
		Our machinery, whenever it can be applied, works to prove uniform stability with a linear estimate. We do not know of examples of groups that are uniformly stable, but without a linear estimate. It is interesting to compare this with \cite{beckermosheiff} where it is shown that $\Z^2$ exhibits pointwise stability (with respect to Hamming metric on $Sym(n)$) but with a non-linear estimate. However, in the case of uniform stability in our setting, $\Z^2$ (being amenable) is uniformly stable for any submultiplicative norm on $U(n)$. The quantitative aspects of stability are an active line of current research.\\ 
		\item \textbf{ Stability with respect to non-archimedean metrics}.\\
		A recent monograph of \cite{fournier1} studies stability with respect to $p$-adic groups. An interesting feature in this non-archimedean setting is that the ultrametric (strong triangle inequality) forces an equivalence between uniform and pointwise stability (for finitely presented groups). One can ask if further stability results in this setting too can be proved using the framework of asymptotic cohomology. 
	\end{itemize}

	\bibliography{references}

\end{document}